\documentclass[a4paper,12pt,leqno]{article}
\usepackage[utf8]{inputenc}

\usepackage[top=3cm, left=2.5cm, right=2.5cm, bottom=3cm]{geometry}

\usepackage{hyperref}
\hypersetup{
    colorlinks=true,
    linkcolor=dodger,
    filecolor=dodger,
    urlcolor=dodger,
    citecolor=dodger,
}

\usepackage{amsmath,mathtools}
\usepackage{amsthm, pb-diagram}
\usepackage{amssymb,comment}
\usepackage{amsfonts,graphicx,color}
\usepackage{enumitem, fancyhdr, dsfont}
\usepackage[normalem]{ulem}
\usepackage{fontawesome}
\usepackage{thmtools,cleveref}
\usepackage{stmaryrd}

\usepackage{lineno,multicol}

\usepackage{tikz,float}

\setlist{topsep=0ex}

\makeatletter
\def\namedlabel#1#2{\begingroup
    #2%
    \def\@currentlabel{#2}%
    \phantomsection\label{#1}\endgroup
}
\makeatother

\definecolor{sub0}{RGB}{29,32,137}
\definecolor{sub1}{RGB}{1,71,157}
\definecolor{sub2}{RGB}{1,104,183}
\definecolor{sub3}{RGB}{0,160,234}
\definecolor{sug}{RGB}{0,154,68}
\definecolor{suy}{RGB}{208,219,1}

    \DeclareMathOperator{\dom}{{\rm dom}}

    \newcommand{\thzfc}{\mathrm{ZFC}}

    \newcommand{\Bwf}{\mathcal{B}}
    \newcommand{\Cwf}{\mathcal{C}}
    \newcommand{\Dwf}{\mathcal{D}}
    \newcommand{\Ewf}{\mathcal{E}}
    \newcommand{\Fwf}{\mathcal{F}}

    \newcommand{\Iwf}{\mathcal{I}}

    \newcommand{\Mwf}{\mathcal{M}}
    \newcommand{\Nwf}{\mathcal{N}}

    \newcommand{\Pwf}{\mathcal{P}}
    \newcommand{\Scal}{\mathcal{S}}
    \newcommand{\Swf}{\mathcal{S}}

    \newcommand{\afrak}{\mathfrak{a}}
    \newcommand{\bfrak}{\mathfrak{b}}
    \newcommand{\cfrak}{\mathfrak{c}}
    \newcommand{\dfrak}{\mathfrak{d}}
    \newcommand{\efrak}{\mathfrak{e}}
    
    \newcommand{\gfrak}{\mathfrak{g}}
    \newcommand{\hfrak}{\mathfrak{h}}
    \newcommand{\mfrak}{\mathfrak{m}}
    
    \newcommand{\pfrak}{\mathfrak{p}}
    \newcommand{\rfrak}{\mathfrak{r}}
    \newcommand{\sfrak}{\mathfrak{s}}
    \newcommand{\tfrak}{\mathfrak{t}}
    \newcommand{\ufrak}{\mathfrak{u}}

    \newcommand{\Cbf}{\mathbf{C}}

    \newcommand{\Bor}{\mathbb{B}}
    \newcommand{\Cbb}{\mathbb{C}}
    \newcommand{\Cor}{\mathbb{C}}
    
    \newcommand{\Dor}{\mathbb{D}}
    \newcommand{\Ebb}{\mathbb{E}}

    \newcommand{\Ibb}{\mathbb{I}}

    \newcommand{\Loc}{\mathbb{L}\mathrm{c}}

    \newcommand{\Por}{\mathbb{P}}
    
    \newcommand{\Qor}{\mathbb{Q}}
    \newcommand{\Qnm}{\dot{\mathbb{Q}}}

    \newcommand{\menos}{\smallsetminus}
    
    \DeclareMathOperator{\pts}{\mathcal{P}}

    \newcommand{\Q}{\mathbb{Q}}
    \newcommand{\R}{\mathbb{R}}

    \newcommand{\imp}{\mathrel{\mbox{$\Rightarrow$}}}

    \newcommand{\sii}{\mathrel{\mbox{$\Leftrightarrow$}}}

    \newcommand{\la}{\langle}
    \newcommand{\ra}{\rangle}
    \DeclareMathOperator{\cf}{\mbox{\rm cf}}
    
    \newcommand{\frestr}{{\upharpoonright}}
    \DeclareMathOperator{\limdir}{\mathrm{limdir}}
    
    \DeclareMathOperator{\Fn}{\mbox{\rm Fn}}
    
    \DeclareMathOperator{\add}{\mathrm{add}}
    \DeclareMathOperator{\non}{\mbox{\rm non}}
    \DeclareMathOperator{\cov}{\mbox{\rm cov}}
    \DeclareMathOperator{\cof}{\mbox{\rm cof}}
    
    \newcommand{\leqT}{\mathrel{\mbox{$\preceq_{\mathrm{T}}$}}}
    \newcommand{\eqT}{\mathrel{\mbox{$\cong_{\mathrm{T}}$}}}

    \newcommand\subsetdot{\mathrel{\ooalign{$\subset$\cr
  \hidewidth\hbox{$\cdot\mkern3mu$}\cr}}}
  
    \newcommand{\Mg}{\mathrm{M}}
    \newcommand{\Ed}{\mathrm{Ed}}

    \newcommand{\seq}[2]{\la #1 \colon  #2\ra}
    \newcommand{\set}[2]{\{#1 \colon  #2\}}
    
    \newcommand{\lset}[2]{\left\{#1 \colon  #2\right\}}

    \newcommand{\Lc}{\mathrm{Lc}}

    \newcommand{\Lb}{\mathrm{Lb}}

    \newcommand{\id}{\mathrm{id}}

\newcommand{\baire}{\omega^\omega}
\newcommand{\cantor}{2^\omega}

    \definecolor{carrotorange}{rgb}{0.93, 0.57, 0.13}
    \definecolor{dodger}{rgb}{0.0,0.5,1.0}
    \definecolor{dred}{RGB}{139,0,0}
    
    \newcommand{\Spl}{\mathrm{Spl}}
    \newcommand{\nsplit}{\sqsubset^{\mathrm{nsp}}}

    \newcommand{\Pred}{\mathrm{Pred}}
    \newcommand{\predict}{\sqsubset^{\mathrm{pr}}}

    \DeclareMathOperator{\sigbar}{\boldsymbol{\bar\Sigma}}
    \newcommand{\varp}{\varepsilon}

    \newcommand{\Cn}{\mathrm{Cn}}
    \newcommand{\relCn}{\not\blacktriangleright}
    \newcommand{\relM}{\sqsubset^{\mathrm m}}

    \newcommand{\Mcal}{\mathcal{M}}
    \newcommand{\Ncal}{\mathcal{N}}
    \newcommand{\Ecal}{\mathcal{E}}

    \DeclareMathOperator{\inter}{int}

    \newcommand{\ltrg}{\mathrel{\lhd}}

\title{Forcing techniques for Cichoń's Maximum\\[1ex] \Large Lecture notes for the mini-course at the University of Vienna}
\author{Diego A.~Mejía%
        \thanks{Email: \href{mailto:damejiag@people.kobe-u.ac.jp}{\texttt{damejiag@people.kobe-u.ac.jp}}
}}

\date{{\normalsize
Graduate School of System Informatics, Kobe University\\ 
1-1 Rokkodai-cho, Nada-ku, Kobe, Hyogo 657-8501, Japan}\\
\ \\
\today
}

\begin{document}

\makeatletter
\def\@roman#1{\romannumeral #1}
\makeatother

\newcounter{enuAlph}
\renewcommand{\theenuAlph}{\Alph{enuAlph}}

\numberwithin{equation}{section}
\renewcommand{\theequation}{\thesection.\arabic{equation}}

\theoremstyle{plain}
  \newtheorem{theorem}[equation]{Theorem}
  \newtheorem{corollary}[equation]{Corollary}
  \newtheorem{lemma}[equation]{Lemma}
  \newtheorem{mainlemma}[equation]{Main Lemma}
  \newtheorem{prop}[equation]{Proposition}
  \newtheorem{clm}[equation]{Claim}
  \newtheorem{fct}[equation]{Fact}
  \newtheorem{fact}[equation]{Fact}
  \newtheorem{question}[equation]{Question}
  \newtheorem{problem}[equation]{Problem}
  \newtheorem{conjecture}[equation]{Conjecture}
  \newtheorem*{thm}{Theorem}
  \newtheorem{teorema}[enuAlph]{Theorem}
  \newtheorem*{corolario}{Corollary}
  \newtheorem*{scnmsc}{(SCNMSC)}
\theoremstyle{definition}
  \newtheorem{definition}[equation]{Definition}
  \newtheorem{example}[equation]{Example}
  \newtheorem{remark}[equation]{Remark}
  \newtheorem{notation}[equation]{Notation}
  \newtheorem{context}[equation]{Context}
  \newtheorem{exer}[equation]{Exercise}
  \newtheorem{exerstar}[equation]{Exercise*}

  \newtheorem*{defi}{Definition}
  \newtheorem*{acknowledgements}{Acknowledgements}
  
\def\sectionautorefname{Section}
\def\subsectionautorefname{Subsection}

\maketitle

\begin{abstract}
Cichoń's diagram describes the connections between combinatorial notions related to measure, category, and compactness of sets of irrational numbers. In the second part of the 2010's, Goldstern, Kellner and Shelah constructed a forcing model of Cichoń's Maximum (meaning that all non-dependent cardinal characteristics are pairwise different) by using large cardinals. Some years later, we eliminated this large cardinal assumption. In this mini-course, we explore the forcing techniques to construct the Cichoń's Maximum model and much more.
\end{abstract}

\section{Tukey connections and cardinal characteristics of the continuum}\label{sec:tukey}

A large part of the contents of this section are taken almost verbatim from: Section 1, up to Figure 3, of~\cite{CM22}; 
and Section 1, up to Fact 1.2, of~\cite{GKMSsplit}.



Many cardinal characteristics of the continuum and their relations can be represented by relational systems as follows. This presentation is based on~\cite{Vojtas,BartInv,blass}.

\begin{definition}\label{def:relsys}
We say that $R=\la X, Y, \sqsubset\ra$ is a \textit{relational system} if it consists of two non-empty sets $X$ and $Y$ and a relation $\sqsubset$.
\begin{enumerate}[label=(\arabic*)]
    \item A set $F\subseteq X$ is \emph{$R$-bounded} if $\exists\, y\in Y\ \forall\, x\in F\colon x \sqsubset y$. 
    \item A set $E\subseteq Y$ is \emph{$R$-dominating} if $\forall\, x\in X\ \exists\, y\in E\colon x \sqsubset y$. 
\end{enumerate}

We associate two cardinal characteristics with this relational system $R$: 
\begin{itemize}
    \item[{}] $\bfrak(R):=\min\{|F|\colon  F\subseteq X  \text{ is }R\text{-unbounded}\}$ the \emph{unbounding number of $R$}, and
    
    \item[{}] $\dfrak(R):=\min\{|D|\colon  D\subseteq Y \text{ is } R\text{-dominating}\}$ the \emph{dominating number of $R$}.
\end{itemize}
\end{definition}

Note that $\dfrak(R)=1$ iff $\bfrak(R)$ is undefined (i.e.\ there are no $R$-unbounded sets, which is the same as saying that $X$ is $R$-bounded).  Dually, $\bfrak(R)=1$ iff $\dfrak(R)$ is undefined (i.e.\ there are no $R$-dominating families).

A very representative general example of relational systems is given by directed preorders.

\begin{definition}\label{examSdir}
We say that $\la S,\leq_S\ra$ is a \emph{directed preorder} if it is a preorder (i.e.\ $\leq_S$ is a reflexive and transitive relation on $S$) such that 
\[\forall\, x, y\in S\ \exists\, z\in S\colon x\leq_S z\text{ and }y\leq_S z.\] 
A directed preorder $\la S,\leq_S\ra$ is seen as the relational system $S=\la S, S,\leq_S\ra$, and its associated cardinal characteristics are denoted by $\bfrak(S)$ and $\dfrak(S)$. The cardinal $\dfrak(S)$ is actually the \emph{cofinality of $S$}, typically denoted by $\cof(S)$ or $\cf(S)$.
\end{definition}

\begin{fact}\label{basicdir}
If a directed preorder $S$ has no maximum element then $\bfrak(S)$ is infinite and regular, and $\bfrak(S)\leq\cf(\dfrak(S))\leq\dfrak(S)\leq|S|$. Even more, if $L$ is a linear order without maximum then $\bfrak(L)=\dfrak(L)=\cof(L)$.
\end{fact}
\begin{proof}
    First notice that $\dfrak(S)$ is infinite, otherwise, by directedness, $\dfrak(S)=1$ and $S$ would have a top element.

    We prove the less obvious $\bfrak(S) \leq \cf(\dfrak(S))$. Assume that $\lambda< \bfrak(S)$ is a cardinal and $\la A_\alpha\colon  \alpha<\lambda\ra$ is a sequence of subsets of $S$ of size ${<}\dfrak(S)$. It is enough to show that $A:= \bigcup_{\alpha<\lambda} A_\alpha$ is not cofinal in $S$. For each $\alpha<\lambda$, since $|A_\alpha|<\dfrak(S)$, $A_\alpha$ is not cofinal in $S$, so there is some $x_\alpha\in S$ such that $x_\alpha \nleq_S y$ for all $y\in A_\alpha$. Now, $|\set{x_\alpha}{\alpha<\lambda}|\leq \lambda < \bfrak(S)$, so there is some $x^*\in S$ such that $x_\alpha \leq_S x^*$ for all $\alpha<\lambda$. Then, $x^*\nleq_S y$ for all $y\in A$, i.e.\ $A$ is not cofinal in $S$.

    A similar argument shows that $\bfrak(S)$ is regular.
\end{proof}

\begin{example}\label{exm:Baire}
Consider $\omega^\omega=\la\omega^\omega,\leq^*\ra$, 
which is a directed preorder. The cardinal characteristics $\bfrak:=\bfrak(\omega^\omega)$ and $\dfrak:=\dfrak(\omega^\omega)$ are the well-known \emph{bounding number} and \emph{dominating number}, respectively.
\end{example}

\begin{example}\label{exm:Iwf}
For any ideal $\Iwf$ on $X$, we consider the following relational systems. 
\begin{enumerate}[label=(\arabic*)]
    \item $\Iwf:=\la\Iwf,\subseteq\ra$ is a directed partial order. Note that 
    \begin{align*}
        \bfrak(\Iwf) & =\add(\Iwf) := \min\lset{|\Fwf|}{\Fwf\subseteq\Iwf,\ \bigcup\Fwf\notin\Iwf} \text{ the \emph{additivity of $\Iwf$},}\\
        \dfrak(\Iwf) & =\cof(\Iwf).
    \end{align*}
    
    \item $\Cbf_\Iwf:=\la X,\Iwf,\in\ra$. 
    When $\bigcup\Iwf = X$,
    \begin{align*}
        \bfrak(\Cbf_\Iwf) & =\non(\Iwf) = \min\set{|F|}{F\subseteq X,\ F\notin \Iwf} \text{ the \emph{uniformity of $\Iwf$},}\\
        \dfrak(\Cbf_\Iwf) & =\cov(\Iwf) = \min\lset{|\Cwf|}{\Cwf\subseteq\Iwf,\ \bigcup\Cwf = X} \text{ the \emph{covering of $\Iwf$}.}
    \end{align*}    
\end{enumerate}
\end{example}





\begin{definition}
    Let $\Xi\colon \Bor\to[0,\infty]$ be a fam (finitely additive measure) on a Boolean algebra $\Bor$. We define the \emph{$\Xi$-null ideal} by
    \[\Nwf(\Xi):=\set{a\in\Bor}{\Xi(a)=0}.\]
    When $\Bor$ is a field of sets over $X$, we extend the definition to
    \[\Nwf(\Xi):=\set{a\subseteq X}{\exists\, b\in\Bor\colon a\subseteq b \text{ and }\Xi(b) = 0}.\]
    This is clearly an ideal on $X$. When $\bigcup\Nwf(\Xi)= X$, i.e.\ every singleton has measure zero, we say that the fam $\Xi$ is \emph{free}.

    Denote by $\Lb$ the Lebesgue measure on $\R$, and let $\Nwf:=\Nwf(\Lb)$.
\end{definition}

\begin{definition}
    Let $X$ be a topological space. We say that $F\subseteq X$ is \emph{nowhere dense (nwd)} if, for any non-empty open $U\subseteq X$, there is some non-empty open $U'\subseteq U$ disjoint from $F$. We say that $A\subseteq X$ is \emph{meager (or of first category)} if $A=\bigcup_{n<\omega}F_n$ for some nwd $F_n$ ($n<\omega$).

    Denote by $\Mwf(X)$ the collection of all meager subsets of $X$, and let $\Mwf:=\Mwf(\R)$.
\end{definition}

\begin{definition}
    Define by $\Ewf$ the ideal generated by the $F_\sigma$ measure zero subsets of $\R$.
\end{definition}

It is clear that $\Ewf\subseteq\Mwf\cap\Nwf$, even more, $\Ewf\subsetneq \Mwf\cap\Nwf$ (\cite[Lem.~2.6.1]{BJ}, see also~\cite[Thm.~3.7]{GaMej}).

\begin{definition}\label{exm:Spl}
\ 
\begin{enumerate}[label= (\arabic*)]
     \item For $a,b\in[\omega]^{\aleph_0}$, we define $a\subseteq^* b$ iff $a\smallsetminus b$ is finite;
     \item and we say that \emph{$a$ splits $b$} if both $a\cap b$ and $b\smallsetminus a$ are infinite, that is, $a\nsupseteq^* b$ and $\omega\smallsetminus a\nsupseteq^* b$.

    \item $F\subseteq[\omega]^{\aleph_0}$ is a \emph{splitting family} if every $y\in[\omega]^{\aleph_0}$ is split by some $x\in F$. The \emph{splitting number} $\sfrak$ is the smallest size of a splitting family.
    \item $D\subseteq[\omega]^{\aleph_0}$ is an \emph{unreaping family} if no $x\in[\omega]^{\aleph_0}$ splits every member of $D$. The \emph{reaping number} $\rfrak$ is the smallest size of an unreaping family.
    
    \item Define the relational system $\Spl:=\la[\omega]^{\aleph_0},[\omega]^{\aleph_0},\nsplit\ra$ by
\[a\nsplit b \text{ iff either $a\supseteq^* b$ or $\omega\menos a \supseteq^* b$.}\]
Note that $a\not\nsplit b$ iff $a$ splits $b$, so $\bfrak(\Spl) = \sfrak$ and $\dfrak(\Spl) = \rfrak$. 
\end{enumerate}
\end{definition}

Inequalities between cardinal characteristics associated with relational systems can be determined by the dual of a relational system and also via Tukey connections, which we introduce below.

\begin{definition}\label{def:dual}
If $R=\la X,Y,\sqsubset\ra$ is a relational system, then its \emph{dual relational system} is defined by $R^\perp:=\la Y,X,\sqsubset^\perp\ra$ where $y \sqsubset^\perp x$ if $\neg(x \sqsubset y)$.
\end{definition}

\begin{fact}\label{fct:dual}
Let $R=\la X,Y,\sqsubset\ra$ be a relational system.
\begin{enumerate}[label=(\alph*)]
    \item $(R^\perp)^\perp=R$.
    \item The notions of $R^\perp$-dominating set and $R$-unbounded set are equivalent.
    \item The notions of $R^\perp$-unbounded set and $R$-dominating set are equivalent.
    \item $\dfrak(R^\perp)=\bfrak(R)$ and $\bfrak(R^\perp)=\dfrak(R)$.
\end{enumerate}
\end{fact}

\begin{definition}\label{def:Tukey}
Let $R=\la X,Y,\sqsubset\ra$ and $R'=\la X',Y',\sqsubset'\ra$ be relational systems. We say that $(\Psi_-,\Psi_+)\colon R\to R'$ is a \emph{Tukey connection from $R$ into $R'$} if 
 $\Psi_-\colon X\to X'$ and $\Psi_+\colon Y'\to Y$ are functions such that  \[\forall\, x\in X\ \forall\, y'\in Y'\colon \Psi_-(x) \sqsubset' y' \Rightarrow x \sqsubset \Psi_+(y').\]
The \emph{Tukey order} between relational systems is defined by
$R\leqT R'$ iff there is a Tukey connection from $R$ into $R'$. \emph{Tukey equivalence} is defined by $R\eqT R'$ iff $R\leqT R'$ and $R'\leqT R$
\end{definition}

\begin{fact}\label{fct:Tukey}
Assume that $R=\la X,Y,\sqsubset\ra$ and $R'=\la X',Y',\sqsubset'\ra$ are relational systems and that $(\Psi_-,\Psi_+)\colon R\to R'$ is a Tukey connection.
\begin{enumerate}[label=(\alph*)]
    \item If $D'\subseteq Y'$ is $R'$-dominating, then $\Psi_+[D']$ is $R$-dominating.
    \item $(\Psi_+,\Psi_-)\colon (R')^\perp\to R^\perp$ is a Tukey connection.
    \item If $E\subseteq X$ is $R$-unbounded then $\Psi_-[E]$ is $R'$-unbounded.
\end{enumerate}
\end{fact}

\begin{corollary}\label{cor:Tukeyval}
\begin{enumerate}[label=(\alph*)]
    \item $R\leqT R'$ implies $(R')^\perp\leqT R^\perp$.
    \item $R\leqT R'$ implies $\bfrak(R')\leq\bfrak(R)$ and $\dfrak(R)\leq\dfrak(R')$.
    \item $R\eqT R'$ implies $\bfrak(R')=\bfrak(R)$ and $\dfrak(R)=\dfrak(R')$.
\end{enumerate}
\end{corollary}

\begin{example}\label{ex:trivialTukey}
The diagram in \autoref{diag:idealI} can be expressed in terms of the Tukey order since $\Cbf_\Iwf\leqT\Iwf$ and $\Cbf_\Iwf^\perp\leqT\Iwf$ when $\Iwf$ is an ideal on $X$ such that $\bigcup\Iwf = X$. The first inequality is obtained via the Tukey connection $x\in X\mapsto \{x\}\in\Iwf$ and $A\in\Iwf \mapsto A\in \Iwf$, and the second is obtained via $A\in\Iwf \mapsto A\in\Iwf$ and $B\in\Iwf\mapsto y_B\in X$ such that $y_B\notin B$.
\end{example}

\begin{figure}[ht]
  \centering
\begin{tikzpicture}[xscale=2/1]
\small{
\node (azero) at (0,1) {$\aleph_0$};
\node (addI) at (1,1) {$\add(\Iwf)$};
\node (covI) at (2,2) {$\cov(\Iwf)$};
\node (nonI) at (2,0) {$\non(\Iwf)$};
\node (cofI) at (3,2) {$\cof(\Iwf)$};
\node (sizX) at (3,0) {$|X|$};
\node (sizI) at (4,1) {$|\Iwf|$};

\draw (azero) edge[->] (addI);
\draw (addI) edge[->] (covI);
\draw (addI) edge[->] (nonI);
\draw (covI) edge[->] (sizX);
\draw (nonI) edge[->] (sizX);
\draw (covI) edge[->] (cofI);
\draw (nonI) edge[-,draw=white,line width=3pt] (cofI);
\draw (nonI) edge[->] (cofI);
\draw (sizX) edge[->] (sizI);
\draw (cofI) edge[->] (sizI);
}
\end{tikzpicture}
    \caption{Diagram of the cardinal characteristics associated with $\Iwf$. An arrow  $\mathfrak x\rightarrow\mathfrak y$ means that (provably in ZFC) 
    $\mathfrak x\le\mathfrak y$.}
    \label{diag:idealI}
\end{figure}
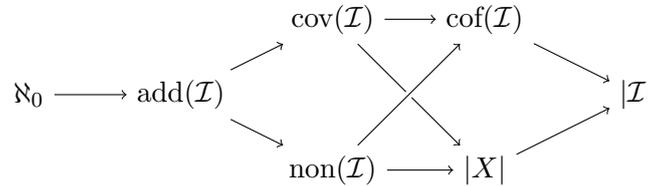



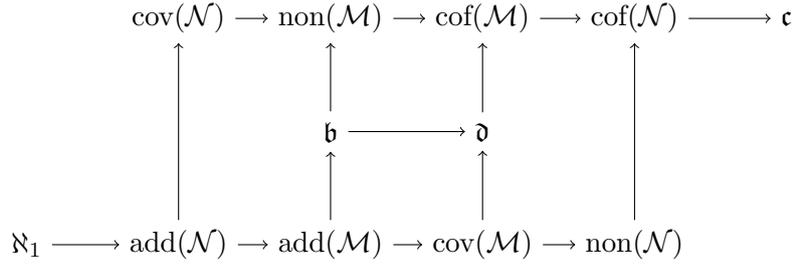
\begin{figure}[ht]
\centering
\begin{tikzpicture}[scale=1.0, transform shape]
\small{
\node (aleph1) at (-1,3) {$\aleph_1$};
\node (addn) at (1,3){$\add(\Ncal)$};
\node (covn) at (1,6){$\cov(\Ncal)$};
\node (nonn) at (7,3) {$\non(\Ncal)$} ;
\node (cfn) at (7,6) {$\cof(\Ncal)$} ;
\node (addm) at (3,3) {$\add(\Mcal)$} ;
\node (covm) at (5,3) {$\cov(\Mcal)$} ;
\node (nonm) at (3,6) {$\non(\Mcal)$} ;
\node (cfm) at (5,6) {$\cof(\Mcal)$} ;
\node (b) at (3,4.5) {$\bfrak$};
\node (d) at (5,4.5) {$\dfrak$};
\node (c) at (9,6) {$\cfrak$};
\draw (aleph1) edge[->] (addn)
      (addn) edge[->] (covn)
      (covn) edge [->] (nonm)
      (nonm)edge [->] (cfm)
      (cfm)edge [->] (cfn)
      (cfn) edge[->] (c);
\draw
   (addn) edge [->]  (addm)
   (addm) edge [->]  (covm)
   (covm) edge [->]  (nonn)
   (nonn) edge [->]  (cfn);
\draw (addm) edge [->] (b)
      (b)  edge [->] (nonm);
\draw (covm) edge [->] (d)
      (d)  edge[->] (cfm);
\draw (b) edge [->] (d);


}
\end{tikzpicture}
  \caption{Cicho\'n's diagram. The arrows mean $\leq$ and dotted arrows represent
  $\add(\Mwf)=\min\{\bfrak,\cov(\Mwf)\}$ and $\cof(\Mwf)=\max\{\dfrak,\non(\Mwf)\}$, which we call the \emph{dependent values}.}
  \label{FigCichon}
\end{figure}

Cicho\'n's diagram (\autoref{FigCichon}) illustrates the inequalities between the cardinal characteristics associated with measure and category of the real numbers. The initial study of this diagram was completed between~1981 and~1993. Inequalities were proved by Bartoszy\'nski, Fremlin, Miller, Rothberger and Truss. The name ``Cicho\'n's diagram" was given by Fremlin~\cite{FrCichon}. On the other hand, the diagram is complete in the sense that no more arrows can be added. Moreover, for any $\aleph_1$-$\aleph_2$ assignment to the cardinals in Cicho\'n's diagram that does not contradict the arrows (and the dependent values), there is a forcing poset that forces the corresponding model. This part of the study was completed by Bartoszy\'nski, Judah, Miller and Shelah.
In fact, the inequalities in Cicho\'n's diagram can be obtained via the Tukey connections as illustrated in \autoref{fig:cichontukey}. See e.g.~\cite{BJ,blass} for all the details.

\begin{figure}[ht]
\centering
\begin{tikzpicture}
\small{
\node (aleph1) at (-1,3) {$\Cbf_{[\R]^{<\aleph_1}}^\perp$};
\node (addn) at (1,3){$\Nwf^\perp$};
\node (covn) at (1,7){$\Cbf_\Nwf$};
\node (nonn) at (9,3) {$\Cbf_\Nwf^\perp$} ;
\node (cfn) at (9,7) {$\Nwf$} ;
\node (addm) at (3,3) {$\Mwf^\perp$} ;
\node (covm) at (7,3) {$\Cbf_\Mwf$} ;
\node (nonm) at (3,7) {$\Cbf_\Mwf^\perp$} ;
\node (cfm) at (7,7) {$\Mwf$} ;
\node (b) at (3,5) {$(\baire)^\perp$};
\node (d) at (7,5) {$\baire$};
\node (c) at (11,7) {$\Cbf_{[\R]^{<\aleph_1}}$};
\draw (aleph1) edge[->] (addn)
      (addn) edge[->] (covn)
      (covn) edge [->] (nonm)
      (nonm)edge [->] (cfm)
      (cfm)edge [->] (cfn)
      (cfn) edge[->] (c);
\draw
   (addn) edge [->]  (addm)
   (addm) edge [->]  (covm)
   (covm) edge [->]  (nonn)
   (nonn) edge [->]  (cfn);
\draw (addm) edge [->] (b)
      (b)  edge [->] (nonm);
\draw (covm) edge [->] (d)
      (d)  edge[->] (cfm);
\draw (b) edge [->] (d);
}
\end{tikzpicture}
\caption{Cicho\'n's diagram via Tukey connections. Any arrow represents a Tukey connection in the given direction.}\label{fig:cichontukey}
\end{figure}
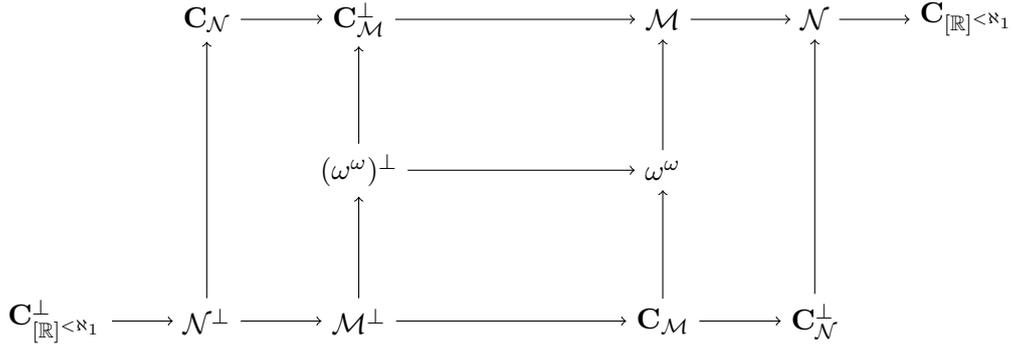

We look at more classical cardinal characteristics. Concerning those associated with $\Ewf$:

\begin{lemma}[{\cite[Lem. 7.4.3]{BJ}}]
    $\Cbf_\Ewf\leqT \Spl$.
\end{lemma}

\begin{theorem}[{\cite{BartSh}}, see also~{\cite[Sec.~2.6]{BJ}}]\label{thm:E}
\ 
  \begin{enumerate}[label =\rm (\alph*)]
      \item $\min\{\bfrak,\non(\Nwf)\} \leq\non(\Ewf)\leq\min\{\non(\Mwf),\non(\Nwf)\}$.
      \item $\max\{\cov(\Mwf),\cov(\Nwf)\}\leq\cov(\Ewf)  \leq\max\{\dfrak,\cov(\Nwf)\}$.
      \item $\add(\Ewf)=\add(\Mwf)$ and $\cof(\Ewf)=\cof(\Mwf)$.
  \end{enumerate}
\end{theorem}

\begin{definition}\label{def:cardchar}
\ 
   \begin{enumerate}[label= (\arabic*)]
    \item $D\subseteq[\omega]^{\aleph_0}$ is \emph{groupwise dense} when:
      \begin{enumerate}[label= (\roman*)]
        \item if $a\in[\omega]^{\aleph_0}$, $b\in D$ and $a\subseteq^* b$, then $a\in D$,
        \item if $\seq{ I_n}{n<\omega}$ is an interval partition of $\omega$ then $\bigcup_{n\in a}I_n\in D$ for some $a\in[\omega]^{\aleph_0}$.
      \end{enumerate}
      The \emph{groupwise density number $\gfrak$} is the smallest size of a collection of groupwise dense sets whose intersection is empty.
    \item The \emph{distributivity number $\mathfrak{h}$} is the smallest size of a collection of dense subsets of $\la[\omega]^{\aleph_0},\subseteq^*\ra$ whose intersection is empty.
    \item Say that $a\in[\omega]^{\aleph_0}$ is a \emph{pseudo-intersection of $F\subseteq[\omega]^{\aleph_0}$} if $a\subseteq^* b$ for all $b\in F$.
    \item The \emph{pseudo-intersection number} $\pfrak$ is the smallest size of a filter base of subsets of $[\omega]^{\aleph_0}$ without pseudo-intersection.
    \item The \emph{tower number} $\tfrak$ is the smallest length of a (transfinite) $\subseteq^*$-decreasing sequence in $[\omega]^{\aleph_0}$ without pseudo-intersection.
    \item Given a class $\Pwf$ of forcing notions, $\mfrak(\Pwf)$ denotes the minimal cardinal $\kappa$ such that, for some $Q\in\Pwf$, there is some collection $\Dwf$ of size $\kappa$ of dense subsets of $Q$ without a filter in $Q$ intersecting every member of $\Dwf$.
    \item Let $\Por$ be a poset. A set $A\subseteq \Por$ is \emph{$k$-linked (in $\Por$)} if every $k$-element subset of $A$
    has a  lower bound in $\Por$.
    $A$ is \emph{centered} if it is $k$-linked for all $k\in\omega$.
    \item
    A poset $\Por$ is \emph{$k$-Knaster}, if for each uncountable $A\subseteq \Por$
    there is a $k$-linked uncountable $B\subseteq A$.
    And $\Por$ \emph{has precaliber $\aleph_1$} if such a $B$ can be chosen centered.
    For notational convenience, \emph{$1$-Knaster} means ccc, and \emph{$\omega$-Knaster} means precaliber $\aleph_1$.
    \item
    For $1\leq k\leq \omega$ denote $\mfrak_k:=\mfrak(k\text{-Knaster})$ and $\mfrak:=\mfrak_1$. We also set $\mfrak_0:=\aleph_1$.

    \item Define the relational system $\Pred=\la \omega^\omega, \Pr, \predict\ra$ where $\Pr$ is the set of functions $\pi$ (called \emph{predictors}) into $\omega$ with domain $\bigcup_{n\in D_\pi}\omega^n$ for some $D_\pi\in[\omega]^{\aleph_0}$, and
    \[x \predict \pi \text{ iff }\exists\, m<\omega\ \forall\, n\geq m\colon n\in D_\pi \imp x(n) = \pi(x\frestr n),\]
    in which case we say that \emph{$\pi$ predicts $x$}. We define $\efrak := \bfrak(\Pred)$ the \emph{evasion number}.
    
    \item Two sets $a$ and $b$ are \emph{almost disjoint} if $a\cap b$ is finite. A family $A$ of sets is an \emph{almost disjoint (a.d.) family} if any pair of members of $A$ are almost disjoint. We say that $A\subseteq[\omega]^{\aleph_0}$ is a \emph{maximal almost disjoint (mad) family} if $A$ is $\subseteq$-maximal among the a.d.~families contained in $[\omega]^{\aleph_0}$. The \emph{almost disjointness number $\afrak$} is the smallest size of an infinite mad family in $[\omega]^{\aleph_0}$.
    
    \item The \emph{ultrafilter number} $\ufrak$ is the smallest size of a filter base generating an ultrafilter contained in $[\omega]^{\aleph_0}$, i.e.\ a \emph{non-principal ultrafilter}.

    \item $\Fn(A,2)$ denotes the set of finite partial functions from $A$ into $2=\{0,1\}$ (see \autoref{Cohen}). When $A\subseteq \pts(\omega)$, for $s\in \Fn(A,2)$ denote
    \[a^s:= \bigcap_{x\in s^{-1}[\{0\}]} x \cap \bigcap_{x\in s^{-1}[\{1\}]}(\omega\smallsetminus x).\]
    A family $A\subseteq\pts(\omega)$ is said to be \emph{independent} if $a^s$ is infinite for all $s\in \Fn(A,2)$, and we say that it is a \emph{maximal independent family} if it is $\subseteq$-maximal among the independent families in $\pts(\omega)$. The \emph{independence number $\mathfrak{i}$} is defined as the smallest size of a maximal independent family.
  \end{enumerate}
\end{definition}

\begin{figure}[ht]
    \centering
    \includegraphics{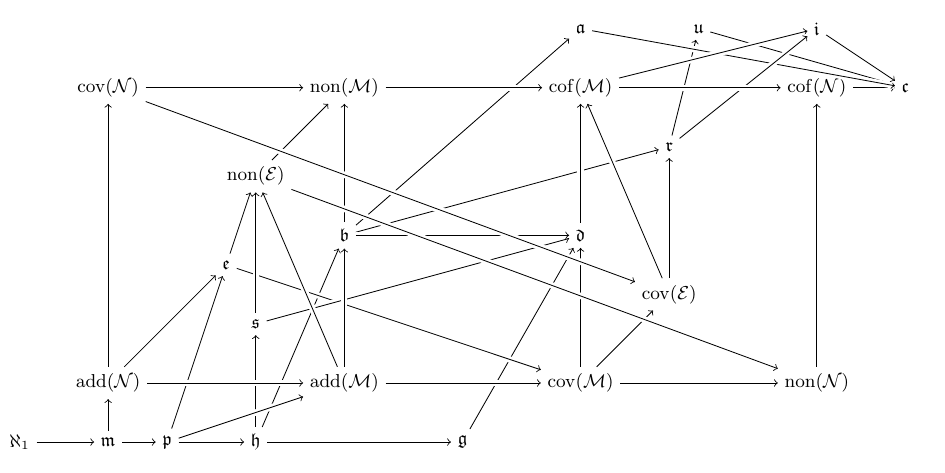}
    \caption{Diagram of inequalities between classical cardinal characteristics}
    \label{fig:manyinv}
\end{figure}

The inequalities between the cardinal characteristics presented so far are summarized in~\autoref{fig:manyinv}. See~\cite{blass,BJ} for the definitions and the proofs for the inequalities (with the exception of $\cof(\Mwf)\leq\mathfrak{i}$, which was proved in~\cite{BHH}). 
Below we list some additional properties of these cardinals.
Unless noted otherwise, proofs can be found in~\cite{blass}.
\begin{fact}\label{fact:blass}
\ 
\begin{multicols}{2} 
\begin{enumerate}[label =\normalfont (\arabic*)]
  \item \cite{MSpt} $\pfrak=\tfrak$.\footnote{Only the trivial inequality $\pfrak\leq\tfrak$ is used in this mini-course.}
  \item The cardinals $\add(\Nwf)$, $\add(\Mwf)$, $\bfrak$, $\tfrak$, $\hfrak$ and $\gfrak$ are regular.
  \item $\cf(\sfrak)\geq\tfrak$ (see \cite{DowShelah}).
  \item $2^{{<}\tfrak}=\cfrak$.
  \item $\cf(\cfrak)\geq\gfrak$.
  \item For $1\leq k\leq k'\leq\omega$, $\mfrak_k\leq\mfrak_{k'}$.
  \item\label{mart} For $1\leq k\leq\omega$, $\mfrak_k>\aleph_1$ implies $\mfrak_k=\mfrak_\omega$ (well-known, but see e.g.\ \cite[Lemma~4.2]{GKMSnoreal}).
\end{enumerate}
\end{multicols}
\end{fact}

Concerning cofinalities:

\begin{fact}
    Let $\Iwf$ be an ideal on $X$ such that $\bigcup\Iwf = X$.
    \begin{enumerate}[label = \normalfont (\arabic*)]
        \item $\add(\Iwf)$ is regular, $\cf(\cof(\Iwf))\geq \add(\Iwf)$ and $\cf(\non(\Iwf))\geq \add(\Iwf)$.
        \item $\cf(\cov(\Mwf))\geq\add(\Nwf)$ (Bartoszy\'nski and Judah~1989, see~\cite[Thm.~5.1.5]{BJ}).
        \item If $\cov(\Nwf)\leq \bfrak$ then $\cf(\cov(\Nwf))>\omega$ (Bartoszy\'nski~1988, see~\cite[Thm.~5.1.17]{BJ}).
        \item If $\cov(\Ewf)\leq \dfrak$ then $\cf(\cov(\Ewf))>\omega$ (Miller, see~\cite[Thm.~5.1.18]{BJ}).
    \end{enumerate}
\end{fact}

The problem of the cofinality of $\cov(\Nwf)$ was settled with the following result.

\begin{theorem}[Shelah~{\cite{ShCov}}]
    It is consistent with $\thzfc$ that $\cf(\cov(\Nwf)) = \omega$.
\end{theorem}

The following question is still unsolved.

\begin{question}
    Is it consistent with $\thzfc$ that $\cf(\cov(\Ewf))=\omega$?
\end{question}

To solve this problem in the positive, it is necessary to force $\dfrak < \cov(\Ewf)$, which implies $\cov(\Ewf) = \cov(\Nwf)$~(see~\autoref{thm:E}), so it would be necessary to force $\cf(\cov(\Nwf)) = \omega$ via an $\omega^\omega$-bounding forcing.

\autoref{fig:manyinv} is quite complete, but the following is still unknown.

\begin{question}
    Is $\afrak\leq \mathfrak{i}$?
\end{question}

It is not even known how to solve:

\begin{question}
    Does $\mathfrak{i}=\aleph_1$ imply $\afrak =\aleph_1$?
\end{question}

A positive answer to this problem is implied by a positive answer to the following famous problem in set theory.

\begin{question}[Roitman's problem]
    Does $\dfrak=\aleph_1$ imply $\afrak = \aleph_1$?
\end{question}

The following strengthening of Roitman's problem was formulated by Brendle and Raghavan~\cite{BrR}.

\begin{question}
    Does $\bfrak=\sfrak =\aleph_1$ imply $\afrak = \aleph_1$?
\end{question}

\section{Finite Support iterations}

\subsection{Generic reals}\label{subsec:genreals}

We first look at the types of generic reals we intend to add by forcing.
Recall that a \emph{Polish space} is a separable completely metrizable space. The real line $\R$ and any product $\prod_{n<\omega} b(n)$ of countable discrete spaces, such as the Cantor space $2^\omega$ and the Baire space $\omega^\omega$, are canonical examples. Polish spaces share many of the  combinatorial and descriptive set-theoretic properties of the real line.

For a Polish space $Z$, denote by $\sigbar(Z)$ the field of sets generated by the analytic subsets of $Z$.

\begin{definition}\label{realsys}\label{DefPolish}
    We say that $R = \la X,Y,\sqsubset\ra$ is a \emph{relational system of the reals} if
    \begin{enumerate}[label = \normalfont (\roman*)]
        \item $X\in \sigbar(Z_1)$ and $Y\in\sigbar(Z_2)$ for some Polish spaces $Z_1$ and $Z_2$, and
        \item ${\sqsubset} \in \sigbar(Z_1\times Z_2)$.
    \end{enumerate}
    In most of the cases, $X = Z_1$ is a perfect Polish space and, for any $y\in Y$, $\set{x\in X}{x\sqsubset y}$ is meager in $X$. More precisely, we are interested in the following notion:  
We say that $R=\langle X,Y,\sqsubset\rangle$ is a \textit{Polish relational system (Prs)} if it satisfies the requirements below:
\begin{enumerate}[label=(\roman*)]
\item $X$ is a perfect Polish space,
\item $Y$ is a non-empty analytic subspace of some Polish space $Z$ and
\item\label{Fsigma} ${\sqsubset} =\bigcup_{n<\omega}{\sqsubset_{n}}$ where $\la{\sqsubset_{n}}\ra_{n<\omega}$ is some increasing sequence of closed subsets of $X\times Z$ such that $({\sqsubset_{n}})^{y}=\{x\in X\colon  x\sqsubset_{n}y \}$ is closed nowhere dense for any $n<\omega$ and $y\in Y$.
\end{enumerate}
\end{definition}

By~\ref{Fsigma}, the maps $x\mapsto x$ and $y\mapsto \set{x\in X}{x\sqsubset y}\in\Mwf(X)$ form a Tukey connection for $\Cbf_{\Mwf(X)}\leqT R$.

Although it is enough to look at Prs's for the applications, the results of this section are valid for relational systems of the reals (and sometimes with extra assumptions). The strength of the notion of Prs is necessary for the iteration theorems in \autoref{sec:pres}.

The reason we use $\sigbar$ in the definition of relational system of the reals is to have absoluteness of the statements ``$x\in X$", ``$y\in Y$" and ``$x\sqsubset y$". In general, we can just use definable sets $X$, $Y$ and $\sqsubset$ such that the previous statements are absolute for the arguments we are carrying out.

For the rest of this section, we fix a relational system of the reals $R = \la X,Y,\sqsubset\ra$. We introduce the following type of (generic) reals related to $R$.

\begin{definition}
    Let $M$ be a (transitive) model of ZFC.\footnote{Since such set models cannot exist, most of the time this expression means that $M$ satisfies a large enough fragment of ZFC to perform the arguments at hand.}
    \begin{enumerate}[label = \normalfont (\arabic*)]
        \item A point $y^*\in Y$ is \emph{$R$-dominating over $M$} if $\forall\, x\in X\cap M\colon x\sqsubset y^*$.
        \item A point $x^*\in X$ is \emph{$R$-unbounded over $M$} if $\forall\, y\in Y\cap M\colon x^*\not\sqsubset y$.
    \end{enumerate}
\end{definition}

We look at many Prs related to the cardinals in Cicho\'n's diagram.

\begin{definition}[Localization]\label{defLoc}
    For $h\in \omega^\omega$ and $H\subseteq \omega^\omega$, define
    \begin{align*}
        \Swf(\omega,h) & := \prod_{i<\omega}[\omega]^{\leq h(i)},\\
        \Swf(\omega,H) & := \bigcup_{h\in H}\Swf(\omega,h).
    \end{align*}
    Objects in these sets are usually called \emph{slaloms}.
    
    For functions $x$ and $y$ with domain $\omega$, we define the relation ``\emph{$y$ localizes $x$}" by
    \[x\in^* y \text{ iff }\exists\, m<\omega\ \forall\, i\geq m\colon x(i)\in y(i).\]
    Define the following \emph{localization relational systems}:
    \begin{align*}
        \Lc(\omega,h) & := \la\omega^\omega, \Swf(\omega,h), \in^*\ra,\\
        \Lc(\omega,H) & := \la\omega^\omega, \Swf(\omega,H), \in^*\ra.
    \end{align*}
    It is easy to check that these are Prs when $H$ is countable. We often work with $H_*:=\set{\id^{k+1}}{k<\omega}$ (powers of the identity function $\id$ on $\omega$)
\end{definition}

The localization relational systems work to easily characterize the cardinal characteristics associated with $\Nwf$.

\begin{theorem}[Bartoszy\'nski {\cite{Ba1984}}(1984), see also {\cite[Sec.~4]{CMlocalc}}]\label{Nchar}
\ \\
    If $h\to \infty$ and $H\subseteq\omega^\omega$ is a countable set containing some function diverging to infinity, then
    \[\Lc(\omega,h)\eqT \Lc(\omega,H) \eqT \Nwf.\]
    In particular, $\bfrak(\Lc(\omega,h))= \bfrak(\Lc(\omega,H)) =\add(\Nwf)$ and $\dfrak(\Lc(\omega,h))= \dfrak(\Lc(\omega,H)) =\cof(\Nwf)$.
\end{theorem}

We now introduce a forcing to modify $\Lc(\omega,h)$. In the context of forcing, $V$ always refers to the \emph{ground model}.

\begin{definition}[Localization forcing]\label{Locposet}
    For $h\in\omega^\omega$, define the poset\footnote{The $m$ in $\Swf(\omega,m)$ refers to the constant function with value $m$.}
    \[\Loc_h := \lset{(n,\varphi)\in\omega\times \Swf(\omega,h)}{\exists\, m<\omega\colon \varphi \in \Swf(\omega,m)}\]
    ordered by
    \[(n',\varphi')\leq (n,\varphi) \text{ iff } n\leq n',\ \varphi'\frestr n = \varphi\frestr n \text{ and } \forall\, i<\omega\colon \varphi(i) \subseteq \varphi'(i).\]
    When $h\to\infty$ we have that $\Loc_h$ is ccc (even $\sigma$-$k$-linked for any $k<\omega$) and it adds a generic slalom $\varphi_*\in\Swf(\omega,h)$ which localizes all functions in the ground model, i.e.\ it is $\Lc(\omega,h)$-dominating over the ground model. If $G$ is $\Lc(\omega,h)$-generic over $V$, the generic slalom is defined by $\varphi_*(i):=\varphi(i)$ when $(n,\varphi)\in G$ and $i<n$ (this value is the same for any such $(n,\varphi)$). 
\end{definition}

We present a Prs that represents the relational system $\Cbf_\Nwf$ (more precisely, its dual). For this purpose, we code measure zero sets as follows.

\begin{definition}\label{Lb2}
    For any topological space $X$, denote by $\Bwf(X)$ the $\sigma$-algebra of Borel subsets of $X$. 
    Let $\Lb_2$ be the measure on $\Bwf(2^\omega)$ defined as the product measure of the uniform measure on $2=\{0,1\}$.\footnote{The uniform measure on a finite non-empty set $b$ assigns probability $\frac{1}{|b|}$ to each point.}
    Recall that $\set{[s]}{s\in 2^{<\omega}}$ forms a base of $2^\omega$ and that each $[s]$ is clopen in $2^\omega$. Then, $\Lb_2$ is the unique measure on $\Bwf(2^\omega)$ such that $\Lb_2([s]) = 2^{-|s|}$ for any $s\in 2^{<\omega}$.\footnote{For $s\in 2^{<\omega}$, $|s| = |\dom s|$ is the length of $s$ as a sequence.}

    We abuse notation and denote $[F]:=\bigcup_{s\in F}[s]$ for $F\subseteq 2^{<\omega}$. Since $2^\omega$ is compact, we have that the clopen sets are precisely of the form $[c]$ for $c\subseteq 2^{<\omega}$ finite.

    We code measure zero subsets of $2^\omega$ in the following way. Fix a sequence $\bar \varp = \la \varp_n\colon  n<\omega\ra$ of positive real numbers such that $\sum_{n<\omega}\varp_n < \infty$. Define
    \[\Omega_{\bar\varp} := \lset{\bar c = \la c_n \colon n<\omega\ra}{\forall\, n<\omega\colon c_n\in[2^{<\omega}]^{<\aleph_0} \text{ and }\Lb_2([c_n])<\varp_n}.\]
    For any sequence $\bar c=\la c_n\colon n<\omega\ra$ of finite subsets of $2^{<\omega}$, denote
    \[N(\bar c) = \bigcap_{m<\omega}\bigcup_{n\geq m}[c_n],\]
    i.e.\ for $x\in 2^\omega$, $x\in N(\bar c)$ iff $x\in [c_n]$ for infinitely many $n$.

    Define the relational system $\Cn:=\la \Omega_{\bar \varp},2^\omega,\relCn\ra$  such that $\bar c \relCn y$ iff $y \notin N(\bar c)$. This is a Prs. 
\end{definition}

The sequences in $\Omega_{\bar \varp}$ are simple codes of (a base of) measure zero sets in $2^\omega$.

\begin{fact}[See e.g.\ {\cite[Lemma~2.3.10]{BJ}}]\label{Cncodes}
    If $\bar{c} \in \Omega_{\bar{\varp}},$ then $N(\bar{c}) \in \mathcal{N}(\Lb_{2})$ and, for any $A \in \mathcal{N}(\Lb_{2}),$ there exists $\bar{c} \in \Omega_{\bar{\varp}}$ such that $A \subseteq N(\bar{c})$. 
\end{fact}

In combinatorics of the reals, working in the Cantor space is the same as working in $\R$ because functions in $2^\omega$ represent the numbers in $[0,1]$ when expressed in base $2$. For this reason, the measure theory of $2^\omega$ is equivalent to the one of $[0,1]$ (with the Lebesgue measure), so $\Nwf(\R) \eqT \Nwf(2^\omega)$ and $\Cbf_{\Nwf(\R)} \eqT \Cbf_{\Nwf(2^\omega)}$. See details in~\cite[Ch.~VII,~\S3]{Levy}.

As a direct consequence of \autoref{Cncodes}, we obtain:

\begin{fact}\label{Cn=CN}
    $\Cn\eqT \Cbf^\perp_{\Nwf(2^\omega)}$, so $\bfrak(\Cn) = \cov(\Nwf)$ and $\dfrak(\Cn) = \non(\Nwf)$.
\end{fact}

\begin{definition}\label{random}


    Random forcing is $\mathcal{B}(\cantor) \menos \mathcal{N}(\cantor)$ ordered by $\subseteq.$  If $G$ is a generic set over $V$, then we can define $r\in 2^\omega$ by $r:=\bigcup \set{s\in 2^{<\omega}}{[s]\in G}$. Such $r$ is called a \emph{random real (over $V$)}.

    Random forcing is ccc (even $\sigma$-$k$-linked for any $k<\omega$).
\end{definition}

\begin{fact}\label{randomreal}
    If $r$ is a random real over $V$, then $r\notin N(\bar c)$ for any $\bar c\in \Omega_{\bar \varp}\cap V$, i.e.\ any random real over $V$ is $\Cn$-dominating over $V$.
\end{fact}

This says that any random real over $V$ evades the Borel measure zero sets coded in the ground model.

Concerning the directed preorder $\la\omega^\omega,\leq^*\ra$:

\begin{fact}
	$\la\omega^\omega,\leq^*\ra$ is a Prs.
\end{fact}

\begin{definition}\label{Hechler}
    \emph{Hechler forcing} is the poset $\Dor := \omega^{<\omega}\times \omega^\omega$ ordered by
    \[(t,g) \leq (s,f) \text{ iff }s\subseteq t,\ \forall\, i<\omega\colon f(i)\leq g(i), \text{ and }\forall\, i\in |t|\menos |s|\colon t(i)\geq f(i).\]
    This poset is ccc (even $\sigma$-centered). 
    
    If $G$ is $\Dor$-generic over $V$, then $d:= \bigcup\set{s}{\exists\, f\colon (s,f)\in G}$ is $\la\omega^\omega,\leq^*\ra$-dominating over $V$.
\end{definition}

We now turn to $\Cbf_\Mwf$. First, we introduce a useful characterization of its cardinal characteristics.

\begin{definition}\label{Ed}
    Define the relational system $\Ed:=\la\omega^\omega,\omega^\omega,\neq^*\ra$ where
    \[x \neq^* y \text{ iff }\exists\,  m<\omega\ \forall\, i\geq m\colon x(i)\neq y(i).\]
    This is a Prs.
\end{definition}

\begin{theorem}[{Miller~\cite{Mi1982}, Bartoszy\'nski~\cite{Ba1987}}, see also~{\cite[Thm.~5.3]{CMlocalc}}]\label{BarMill}
\ \\
    $\bfrak(\Ed)=\non(\Mwf)$ and $\dfrak(\Ed)=\cov(\Mwf)$.   
\end{theorem}

\begin{definition}\label{edposet}
    Define the \emph{eventually different real forcing} by 
    \[\Ebb:=\lset{(s,\varphi)}{s\in\omega^{<\omega},\ \varphi\in\bigcup_{m<\omega}\Scal(\omega,m)}\]
    ordered by
    \[(t,\psi)\leq (s,\varphi) \text{ iff } s\subseteq t,\ \forall\, i<\omega\colon \varphi(i)\subseteq \psi(i), \text{ and }\forall\, i\in |t|\menos |s|\colon t(i)\notin \varphi(i).\]
    This forcing is ccc (even $\sigma$-centered). 
    
    If $G$ is $\Ebb$-generic over $V$ then $e:= \bigcup\set{s}{\exists\, \varphi\colon (s,\varphi)\in G}$ is $\Ed$-dominating over $V$.
\end{definition}

Therefore, by~\autoref{BarMill}, $\Ebb$ can be used to increase $\non(\Mwf)$. But it actually does more:

\begin{theorem}[Cardona \& Mej\'ia~{\cite[Clm.~4.11]{CarE-RIMS}}]\label{edE}
    $\Ebb$ adds a $\Cbf_\Ewf$-dominating real over $V$, i.e.\ an $F_\sigma$ null subset of $\R$ that covers $\R\cap V$.
\end{theorem}

As for measure, we have that $\Mwf(\R)\eqT \Mwf(2^\omega)$ and $\Cbf_{\Mwf(\R)}\eqT \Cbf_{\Mwf(2^\omega)}$, so we obtain the same cardinal characteristics for the meager ideal using the Cantor space instead of $\R$. More generally, as a consequence of~\cite[Ex.~8.32 \&~Thm.~15.10]{Ke2}:

\begin{theorem}
    For any perfect Polish space $X$, $\Mwf(X) \eqT \Mwf(\R)$ and $\Cbf_{\Mwf(X)} \eqT \Cbf_{\Mwf(\R)}$.
\end{theorem}

In particular:

\begin{fact}\label{Cohenunb}
Let $R=\la X,Y,\sqsubset\ra$ be a Prs in $V$. Then, 
any Cohen real $x\in X$ over $V$ is $R$-unbounded over $V$. 
\end{fact}

We now look at the effect of Cohen forcing on meager sets. As we did with measure zero, we introduce a coding of (a base of) meager subsets of $2^\omega$.

\begin{definition}\label{matching}
    Let $\Ibb$ be the set of interval partitions of $\omega$. Define the relational system $\Mg := \la 2^\omega, 2^\omega\times \Ibb, \relM\ra$ where
    \[x \relM (y,\bar I) \text{ iff }\exists\, m<\omega\ \forall\, n\geq m\colon x\frestr I_n \neq y\frestr I_n.\]
    This is a Prs. The members of $2^\omega\times \Ibb$ are usually called \emph{matching reals}. For any matching real $(y,\bar I)$, define $M(y,\bar I):= \set{x\in 2^\omega}{x \relM (y,\bar I)}$.
\end{definition}

\begin{fact}[See e.g.~\cite{blass}]\label{codeM}
    For any matching real $(y,\bar I)$, $M(y,\bar I)$ is meager in $2^\omega$. And, for any $A\in\Mwf(2^\omega)$, there is some matching real $(y,\bar I)$ such that $A\subseteq M(y,\bar I)$.
\end{fact}

\begin{corollary}\label{M=CM}
    $\Mg \eqT \Cbf_{\Mwf(2^\omega)}$. In particular, $\bfrak(\Mg) = \non(\Mwf)$ and $\dfrak(\Mg) = \cov(\Mwf)$.
\end{corollary}

\begin{definition}\label{Cohen}
    Let $I$ be a set and $\bar b =\la b(i)\colon  i\in I\ra$ a sequence of non-empty sets. Define the poset
    \[\Fn(\bar b) := \lset{p}{ p \text{ is a finite function, $\dom p \subseteq I$ and }\forall\, i\in \dom p\colon p(i)\in b(i)}\]
    ordered by $\supseteq$. The generic real added by this poset is $g:=\bigcup G \in \prod \bar b:= \prod_{i\in I} b(i)$ whenever $G$ is $\Fn(\bar b)$-generic over $V$.

    We use this forcing to add Cohen reals, not just over $2^\omega$ or $\omega^\omega$, but over any perfect space of the form $\prod_{n<\omega}b(n)$, endowed with the product topology for countable discrete spaces $b(n)$ ($n<\omega$).

    Fix a countable sequence $\bar b:= \la b(n)\colon  n<\omega\ra$ of countable non-empty sets. Note that $\prod \bar b$ is a perfect Polish space iff $|b(n)|\geq 2$ for infinitely many $n<\omega$. In this case, we call $\Fn(\bar b)$ the \emph{forcing adding a Cohen real in $\prod \bar b$}, usually referred to as \emph{Cohen forcing.} We use $c$ to denote the generic real in $\prod\bar b$ added by this poset, which we often call \emph{Cohen real}. For example, $\Omega_{\bar \varp}$ is such a space, and a Cohen real in $\Omega_{\bar\varp}$ over $V$ codes a measure zero set that covers $2^\omega\cap V$. The letter $\Cbb$ is reserved for any version of Cohen forcing.

    For any set $I$, denote $\Cor_I := \Fn(\hat b)$ where $\hat b := \la b(i,n)\colon i\in I,\ n<\omega \ra$ is defined by $b(i,n):= b(n)$. This poset adds a sequence $\la c_i\colon i\in I\ra$ where each $c_i\in \prod_{n<\omega}b(n)$ is a Cohen real over $V$ (and even over $V^{\Cor_{I\menos\{i\}}}$).    
\end{definition}

All the versions of Cohen forcing are forcing equivalent:

\begin{theorem}\label{Cohenctbl}
    Any countable atomless forcing notion is forcing equivalent with $\Cor$.
\end{theorem}

In general, for any perfect Polish space $X$, it is possible to define a countable atomless forcing that adds a generic real $c\in X$.\footnote{Using finite fragments of Cauchy sequences coming from a countable dense subset of $X$.} The main property of this generic real is that it evades all the Borel meager subsets of $X$ coded in the ground model. In particular,

\begin{theorem}\label{Cohenreal}
    If $c\in 2^\omega$ is a Cohen real over $V$, then $c\notin M(y,\bar I)$ for any matching real $(y,\bar I)\in V$. In particular, any Cohen real is $\Mg$-unbounded over $V$. 
\end{theorem}

\begin{remark}
	Both $\Mg$ and $\Ed$ have the same cardinal characteristics as $\Cbf_\Mwf$, $\Mg\eqT\Cbf_\Mwf$ and $\Cbf_\Mwf\leqT\Ed$. 
	The converse Tukey-inequality is not provable in $\thzfc$. 
	This follows by Zapletal's result stating that there is a proper poset adding an $\Ed$-unbounded real but not adding Cohen reals~\cite{ZapDim}.
\end{remark}

\subsection{FS iterations}\label{subsec:FS}

We now turn to FS (finite support) iterations. 
Any FS iteration $\la \Por_\alpha,\Qnm_\beta\colon \alpha\leq\pi,\ \beta<\pi\ra$ of length $\pi$ is defined by recursion as follows: 

\begin{enumerate}[label = \normalfont (\Roman*)]
    \item $\Por_0:=\{\la\ \ra\}$ is the poset containing the empty sequence $\la\ \ra$, usually called the \emph{trivial poset}.
    \item When $\Por_\alpha$ has been defined, we pick a $\Por_\alpha$-name $\Qnm_\alpha$ of a poset and define $\Por_{\alpha+1} = \Por_\alpha\ast \Qnm_\alpha$.
    \item For limit $\gamma\leq\pi$, $\Por_\gamma:= \limdir_{\alpha<\gamma}\Por_\alpha = \bigcup_{\alpha<\gamma}\Por_\alpha$ ordered by
    \[q \leq_\gamma p \text{ iff }\exists\, \alpha<\gamma\colon p,q\in\Por_\alpha \text{ and }q\leq_\alpha p.\]
\end{enumerate}
Here, $\leq_\alpha$ denotes the preorder of $\Por_\alpha$. It can be proved by induction that $\Por_\alpha\subsetdot \Por_\beta$ whenever $\alpha\leq \beta\leq \pi$, where $\subsetdot$ denotes the \emph{complete-subposet} relation.\footnote{$\Por\subsetdot \Qor$ iff $\Por$ is a suborder of $\Qor$, the incompatibility relation is preserved, and any predense subset of $\Por$ is predense in $\Qor$.}

If $G$ is $\Por_\pi$-generic over $V$ and $\alpha\leq \pi$, then $G_\alpha:=\Por_\alpha\cap G$ is $\Por_\alpha$-generic over $V$, so $G_\pi = G$. In the context of FS iterations, we denote $V_\alpha:= V[G_\alpha]$, so $V_0 = V$. The relation $\subsetdot$ indicates that $V_\alpha\subseteq V_\beta$ whenever $\alpha\leq\beta\leq \pi$. So,
when $\alpha<\pi$, we call $V_\alpha$ an \emph{intermediate generic extension}, and $V_\pi$ the \emph{final generic extension}.

In this context, we abbreviate the forcing relation $\Vdash_{\Por_\alpha}$ by $\Vdash_\alpha$.

We review some basic facts about FS iterations of ccc posets.

\begin{lemma}\label{FSccc}
    Any FS iteration of ccc posets is ccc, i.e.\ if $\Vdash_\beta \Qnm_\beta$ is ccc for all $\beta<\pi$, then $\Por_\alpha$ is ccc for all $\alpha\leq\pi$.
\end{lemma}

\begin{lemma}\label{FSnoreal}
    In any FS iteration of ccc posets of length $\pi$: if $\cf(\pi)>\omega$ then $\R\cap V_\pi = \bigcup_{\alpha<\pi}\R\cap V_\alpha$.
\end{lemma}

\begin{lemma}\label{FSCohen}
    Any FS iteration of non-trivial\footnote{A poset is \emph{trivial} if all its conditions are pairwise compatible. This is equivalent to saying that the poset is forcing equivalent with the trivial poset.} posets adds Cohen reals at limit stages. Concretely, $\Por_{\alpha+\omega}$ adds a Cohen real over $V_\alpha$.
\end{lemma}

The Cohen reals added by a FS iteration determine a Tukey connection for $\Cbf_\Mwf$ as follows.

\begin{corollary}\label{FScfpi}
    Any FS iteration of ccc posets of length $\pi$ with uncountable cofinality forces $\non(\Mwf) \leq \cf(\pi) \leq \cov(\Mwf)$, even more, $\pi\leqT \Cbf_\Mwf$ and $\gfrak\leq\cf(\pi)$.
\end{corollary}
\begin{proof}
    Work in $V_\pi$. For any matching real $(y,\bar I)$, by \autoref{FSnoreal} there is some $\alpha_{y,\bar I}<\pi$ such that $(y,\bar I)\in V_{\alpha_{y,\bar I}}$. On the other hand, by \autoref{FSCohen}, there is some Cohen real $c_\alpha\in 2^\omega\cap V_{\alpha+\omega}$ over $V_\alpha$. Then by \autoref{Cohenreal}, $c_\alpha \not\relM (y,\bar I)$ whenever $(y,\bar I) \in V_\alpha$, which happens when $\alpha_{y,\bar I}\leq\alpha$. This indicates that the maps $\alpha\mapsto c_\alpha$ and $(y,\bar I)\mapsto \alpha_{y,\bar I}$ form a Tukey connection for $\pi\leqT \Mg$.
    
    We now show that $\gfrak\leq \cov(\Mcal)$ in $V_\pi$. Let $\seq{\alpha_\xi}{\xi<\cf(\pi)}$ be an increasing sequence of limit ordinals with limit $\nu:=\cf(\pi)$. Set $V'_\xi:=V_{\alpha_\xi}$ for $\xi<\nu$ and $V'_\nu:=V_\pi$. Then $\seq{\baire\cap V'_\xi}{\xi<\nu}\in V'_\nu$ is strictly increasing and $\bigcup_{\xi<\nu}\baire\cap V'_\xi = \baire\cap V'_\nu$. Hence, by a result of Blass~\cite[Thm.~2]{BlassSP} (see also Brendle's proof~\cite[Lem.~1.17]{BrHejnice}), it follows that $\gfrak\leq\nu$ in $V'_\nu$.
\end{proof}

The previous result puts a restriction on the models of Cicho\'n's diagram that can be obtained via FS iterations of ccc posets (of uncountable cofinality), since they force the inequalities $\non(\Mwf)\leq\cov(\Mwf)$ and $\gfrak\leq\cov(\Mwf)$. Therefore, in such models, the diagram of cardinal characteristics presented in \autoref{fig:manyinv} takes the form as in \autoref{fig:FSmany}.

\begin{figure}[ht]
\centering
\begin{tikzpicture}
\small{
 \node (aleph1) at (-1.5,0) {$\aleph_1$};
 \node (addn) at (0,2.5){$\add(\Nwf)$};
 \node (covn) at (2,5){$\cov(\Nwf)$};
 \node (cove) at (7,3.75) {$\cov(\Ecal)$};
 \node (b) at (2,0) {$\bfrak$};
 \node (nonm) at (4,2.5) {$\non(\Mcal)$} ;
 \node (none) at (3,1.25) {$\non(\Ecal)$};
 \node (d) at (8,5) {$\dfrak$};
 \node (covm) at (6,2.5) {$\cov(\Mcal)$} ;
 \node (nonn) at (8,0) {$\non(\Ncal)$} ;
 \node (cfn) at (10,2.5) {$\cof(\Ncal)$} ;
 \node (s) at (3,0) {$\sfrak$};
 \node (r) at (7,5) {$\rfrak$};
 \node (c) at (11.5,5) {$\cfrak$};
 \node (m) at (0,0) {$\mathfrak m$};
 \node (p) at (1,-1) {$\mathfrak p$};
 \node (h) at (2.5,-1) {$\mathfrak h$};
 \node (g) at (5,1.5) {$\mathfrak g$};
 \node (e) at (2,2.5) {$\mathfrak e$};
 \node (a) at (3.5,6) {$\mathfrak a$};
 \node (u) at (8,6) {$\mathfrak u$};
 \node (i) at (10,6) {$\mathfrak i$};

\foreach \from/\to in {
m/addn, cfn/c, aleph1/m, m/p, p/h, h/s, h/g, r/u, u/c,
h/b, 
g/covm, p/e, addn/e, e/none, 
b/a, a/c, d/i, r/i, i/c,
s/none, 
cove/r,
b/none, cove/d,
covm/cove, none/nonm,
addn/covn, 
addn/b, covn/nonm, covm/nonn, d/cfn, nonn/cfn,
nonm/covm}
{
\path[-,draw=white,line width=3pt] (\from) edge (\to);
\path[->,] (\from) edge (\to);
}
\path[->,draw=red,line width=1pt] (nonm) edge (covm);
\path[->,draw=red,line width=1pt] (g) edge (covm);
}
\end{tikzpicture}
\caption{Cicho\'n's diagram with other classical cardinal characteristics after a FS iteration of ccc (non-trivial) posets of length with uncountable cofinality, as an effect of the forced inequalities $\non(\Mwf) \leq \cov(\Mwf)$ and $\gfrak \leq \cov(\Mwf)$.}\label{fig:FSmany}
\end{figure}
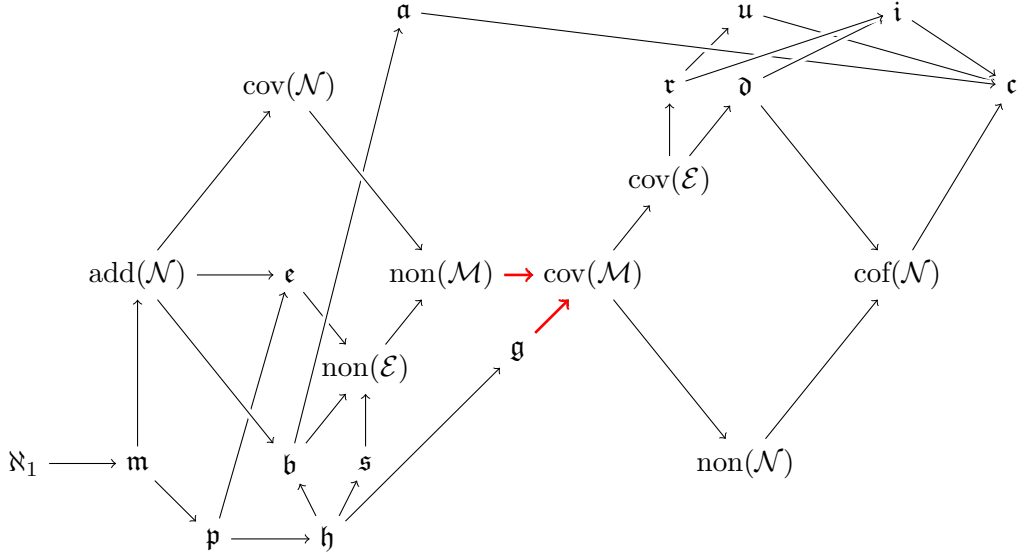

Below, we summarize the effect of the forcings introduced in this section to modify the cardinals in Cicho\'n's diagram:

\begin{enumerate}[label = \normalfont (\arabic*)]
    \item When $h\to \infty$, $\Loc_h$ adds $\Lc(\omega,h)$-dominating reals (so it affects $\add(\Nwf)$ and $\cof(\Nwf)$).
    \item Random forcing adds $\Cn$-dominating reals (affecting $\cov(\Nwf)$ and $\non(\Nwf)$).
    \item Hechler forcing adds $\la \omega^\omega,\leq^*\ra$-dominating reals (affecting $\bfrak$ and $\dfrak$).
    \item The forcing $\Ebb$ adds $\Ed$-dominating reals, and also $\Cbf_\Ewf$-dominating reals (affecting $\non(\Mwf)$, $\non(\Ewf)$ and $\cov(\Mwf)$, $\cov(\Ewf)$).
    \item Cohen forcing adds $\Mg$-unbounded reals (affecting $\cov(\Mwf)$ and $\non(\Mwf)$).
\end{enumerate}

We are going to use these forcings to modify the cardinals in Cicho\'n's diagram. However, we cannot just simply add dominating reals without any particular restriction, as indicated in the following result.

\begin{lemma}
    Let $\la \Por_\alpha,\Qnm_\beta\colon \alpha\leq\pi,\ \beta<\pi\ra$ be a FS iteration of ccc posets. 
    Assume that $\cf(\pi)>\omega$, $K\subseteq \pi$ is cofinal, $R = \la X,Y,\sqsubset\ra$ is a relational system of the reals (see \autoref{realsys}), and assume that, for $\alpha\in K$, $\Qnm_\alpha$ adds an $R$-dominating real over $V_\alpha$.

    Then $\Por_\pi$ forces $\dfrak(R) \leq \cf(\pi) \leq \bfrak(R)$, even $R\leqT \pi$. 
    
    Moreover, if $\thzfc$ proves $\Cbf_\Mwf\leqT R$, then $\Por_\pi$ forces $R\eqT \Cbf_\Mwf \eqT \pi$, so $\bfrak(R)=\dfrak(R)=\non(\Mwf)=\cov(\Mwf) = \cf(\pi)$.
\end{lemma}
\begin{proof}
Work in $V_\pi$. If $x\in X$, by \autoref{FSnoreal} there is some $\alpha_x<\pi$ such that $x\in V_{\alpha_x}$. On the other hand, for any $\alpha<\pi$, there is some $\beta_\alpha\in K$ above $\alpha$, so $\Qnm_{\beta_\alpha}$ adds an $R$-dominating real $y_\alpha\in Y$ over $V_{\beta_\alpha}$. Then the maps $x\mapsto \alpha_x$ and $\alpha\mapsto y_\alpha$ form the Tukey connection for $R\leqT \pi$.

The rest is a consequence of \autoref{FScfpi}.
\end{proof}

When aiming to force many different values to cardinal characteristics, we cannot add \emph{full} dominating reals as in the previous lemma. However, there is a way to add \emph{restricted} dominating reals, allowing better control of the cardinal characteristics. We develop this technique in the following part.

\subsection{Book-keeping arguments}

Fix, for the rest of this section:
\begin{enumerate}[label = \normalfont (\arabic*)]
    \item A relational system $R=\la X,Y,\sqsubset\ra$ of the reals (see~\autoref{realsys}) such that $|X|=\cfrak =2^{\aleph_0}$ 

    \item A very definable (i.e.\ Suslin) ccc poset $\Qor_R$ adding $R$-dominating reals over the ground model, such that $|\Qor_R|\leq\cfrak$. Note that $\Loc_h$, $\Dor$, $\Ebb$, random forcing and Cohen forcing satisfy these conditions (for certain $R$ as in the previous subsection).

    \item An infinite cardinal $\theta$.
\end{enumerate}

We aim to force $\bfrak(R) =\theta$. For the rest of this section we deal with $\bfrak(R)\geq\theta$, and from the following section onwards we deal with the converse inequality.

Forcing $\bfrak(R)\geq\theta$ means to force that $\forall\, F\in [X]^{<\theta}\ \exists\, y\in Y\ \forall\, x\in F\colon x\sqsubset y$. One way is to deal with one $F$ at a time along a FS iteration. Concretely, if we are at step $\alpha$ of a FS iteration, we pick some $F_\alpha\in [X]^{<\theta}\cap V_\alpha$ and aim to add a $y_\alpha\in Y\cap V_{\alpha+1}$ that $R$-dominates all members of $F_\alpha$. A very effective idea to do this comes from Brendle~\cite{Br}: in $V_\alpha$, pick a transitive model $N_\alpha$ of ZFC such that $F_\alpha\subseteq N_\alpha$ and $|N_\alpha|=\max\{\aleph_0,|F_\alpha|\}<\theta$.\footnote{This is possible because the members of $X$ are ``reals".} So forcing with $\Qor_\alpha=\Qor^{N_\alpha}_R$ (which is ccc) does the job: it adds an $R$-dominating $y_\alpha$ over $N_\alpha$, hence it dominates all members of $F_\alpha$. The hope is that this $y_\alpha$ does not dominate much larger fragments of $X$.

Now, assume that $\pi$ is an ordinal of uncountable cofinality, and that we perform a FS iteration of ccc posets of length $\pi$ as explained before. To force $\bfrak(R)\geq\theta$, it is enough to guarantee that, in $V_\pi$, $\set{F_\alpha}{\alpha<\pi}$ is \emph{cofinal in $[X]^{<\theta}$}. Indeed, if $F\in [X]^{<\theta}$ then $F\subseteq F_\alpha$ for some $\alpha<\pi$, so $y_\alpha$ dominates all members of $F_\alpha$, and then all members of $F$.

In the practice, we do not use all steps $\alpha<\pi$ to take care of $\bfrak(R)\geq\theta$, but only steps $\alpha\in K$ for some (cofinal) $K\subseteq \pi$, while other steps can be used to take care of something else.
So we explain how to construct an iteration as above ensuring that some choice of $\set{F_\alpha}{\alpha\in K}$ is cofinal in $[X]^{<\theta}$.

To do this, we first have to look at what happens to $|X|$ in the final extension. Recall that $|X|=\cfrak$. Assume $\theta\leq\lambda=\lambda^{\aleph_0}$. Then, in a FS iteration of length $\pi<\lambda^+$ of ccc posets, we can ensure that $|\Por_\alpha|\leq \lambda$ and $\Vdash_\alpha \cfrak\leq\lambda$ as long as we have $\Vdash_\alpha |\Qnm_\alpha|\leq\lambda$ for all $\alpha<\pi$. This is fine in this context because all forcings we use to iterate have size ${\leq}\cfrak$.

If $\lambda\leq\pi<\lambda^+$ then, in $V_\pi$, $|X| = \cfrak =\lambda$, so $\cof([X]^{<\theta}) = \cof([\lambda]^{<\theta})$. Now, the existence of a collection $\set{F_\alpha}{\alpha\in K}$ cofinal in $[X]^{<\theta}$ for some $K\subseteq \pi$, $K\in V$, implies that $\cof([\lambda]^{<\theta}) = \cof([X]^{<\theta}) \leq |K| \leq |\pi| \leq \lambda$. Hence, a requirement to obtain such a cofinal family is that $\cof([\lambda]^{<\theta}) = \lambda$ and $|K| = \lambda$.

We now show that the assumptions $\theta\leq\cf(\lambda)\leq\lambda=\lambda^{\aleph_0}$ and $\cof([\lambda]^{<\theta}) = \lambda$ are enough to construct such an iteration via a \emph{book-keeping argument}. We only explain the construction for $\pi=\lambda$, but it works when $\pi=\lambda\delta$ (ordinal product) for some $0<\delta<\lambda^+$ and $\cf(\pi)\geq\theta$. Let $K\subseteq\lambda$ of size $\lambda$ and
fix a bijection $h\colon K\to \lambda\times\lambda$ such that $h(\alpha) = (\xi,\eta)$ implies $\xi\leq \alpha$. Now, perform a FS iteration of ccc posets and assume we have reached the stage $\alpha<\lambda$. Since $\cov([\lambda]^{<\theta})$ is not modified by ccc forcing\footnote{Because, when $\theta$ is uncountable, in any ccc generic extension, any set of ordinals of size ${<}\theta$ is covered by a set in the ground model of size ${<}\theta$.}, in $V_\alpha$ we can pick a cofinal $\set{F_{\alpha,\eta}}{\eta<\lambda}$ in $[X]^{<\theta}\cap V_\alpha$ (because $|X|=\cfrak\leq \lambda$). In the previous steps $\xi\leq\alpha$, in the same way we had picked in $V_\xi$ a cofinal $\set{F_{\xi,\eta}}{\eta<\lambda}$ in $[X]^{<\theta}\cap V_\xi$. 
If $\alpha\notin K$ then we can force with any ccc poset, but when $\alpha\in K$,
the book-keeping function $h$ makes the choice: letting $h(\alpha)=(\xi,\eta)$, pick $F_{h(\alpha)} = F_{\xi,\eta}$ (which exists because $\xi\leq\alpha$). As before, let $N_\alpha$ be a transitive model of $\thzfc$ such that $F_{h(\alpha)} \subseteq N_\alpha$ and $|N_\alpha|<\theta$, and we force with $\Qor_\alpha:= \Qor_R^{N_\alpha}$ to go to the step $\alpha+1$.

At the end of the iteration, in $V_\lambda$, we have ensured that each member of $\set{F_{h(\alpha)}}{\alpha\in K} = \set{F_{\xi,\eta}}{\xi,\eta<\lambda}$ is $R$-bounded. It remains to ensure that this family is cofinal in $[X]^{<\theta}$: If $F\in[X]^{<\theta}$ then, since $\cf(\lambda)\geq\theta$, we have that $F\in V_\xi$ for some $\xi<\lambda$, so $F\subseteq F_{\xi,\eta}$ for some $\eta<\lambda$.

Concerning $\cof([\lambda]^{<\theta})$, by \autoref{basicdir} we have that $\cf(\cof([\lambda]^{<\theta})) \geq \add([\lambda]^{<\theta}) = \cf(\theta)$, so $\cof([\lambda]^{<\theta}) = \lambda$ implies $\cf(\lambda)\geq \cf(\theta)$, which is $\cf(\lambda)\geq\theta$ in the case when $\theta$ is regular.

Using the book-keeping argument presented above, we are now ready to present the first important construction of models with several pairwise different cardinal characteristics.

\begin{theorem}\label{main1}
    Let $\aleph_1\leq \theta_1\leq \theta_2\leq \theta_3 \leq \theta_4$ be regular cardinals, and assume $\lambda$ is a cardinal such that $\lambda = \lambda^{\aleph_0}$ and $\cf([\lambda]^{<\theta_i}) = \lambda$ for $i=1,\ldots,4$. Then, we can construct a FS iteration of length (and size) $\lambda$ of ccc posets forcing $\add(\Nwf) = \theta_1$, $\cov(\Nwf) = \theta_2$, $\bfrak \geq \theta_3$, $\non(\Ewf) =\non(\Mwf) = \theta_4$ and $\cov(\Mwf) = \cfrak = \lambda$ (see \autoref{fig:main1}). 
\end{theorem}

The argument of the proof only gives $\bfrak\geq\theta_3$. In \autoref{sec:uflim} we are going to show how to obtain $\bfrak=\theta_3$, in addition.

In the first part of the proof we only deal with equalities of the form $\bfrak(R)\geq \theta_i$ and $\cfrak =\lambda$. In the next section, we deal with the rest of the proof.

\begin{proof}[Proof of~\autoref{main1}, part 1.]
 Partition $\lambda = K_1\cup K_2\cup K_3 \cup K_4$ with $|K_i| = \lambda$. Proceed in two steps:

 \textbf{Step 1.} Force with $\Cor_\lambda$ (i.e.\ add $\lambda$-many Cohen reals).

 \textbf{Step 2.} In $V_0:=V^{\Cor_\lambda}$, using book-keeping as before at each $K_i$, iterate with length $\lambda$ and at:
 \begin{enumerate}[label = $\alpha\in K_{\arabic*}$:, left=0pt]
     \item  force with $\Loc^{N_\alpha}_\id$, 
     $|N_\alpha|<\theta_1$, which guarantees $\add(\Nwf)\geq \theta_1$ in the final extension;

     \item force with $(\Bwf(2^\omega)\menos\Nwf(2^\omega))^{N_\alpha}$, $|N_\alpha|<\theta_2$, which guarantees $\cov(\Nwf)\geq \theta_2$ in the final extension;

     \item force with $\Dor^{N_\alpha}$, $|N_\alpha|<\theta_3$, which guarantees $\bfrak\geq \theta_3$ in the final extension;

     \item force with $\Ebb^{N_\alpha}$, $|N_\alpha|<\theta_4$, which guarantees $\non(\Mwf)\geq \theta_4$, and even $\non(\Ewf)\geq \theta_4$ (by \autoref{edE}), in the final extension.
 \end{enumerate}
 It is clear by the construction that, in $V_\lambda$, $\cfrak = \lambda$.
\end{proof}

Note that we have not used the Cohen reals from step~1. These will be used to prove the converse inequalities in the next section.

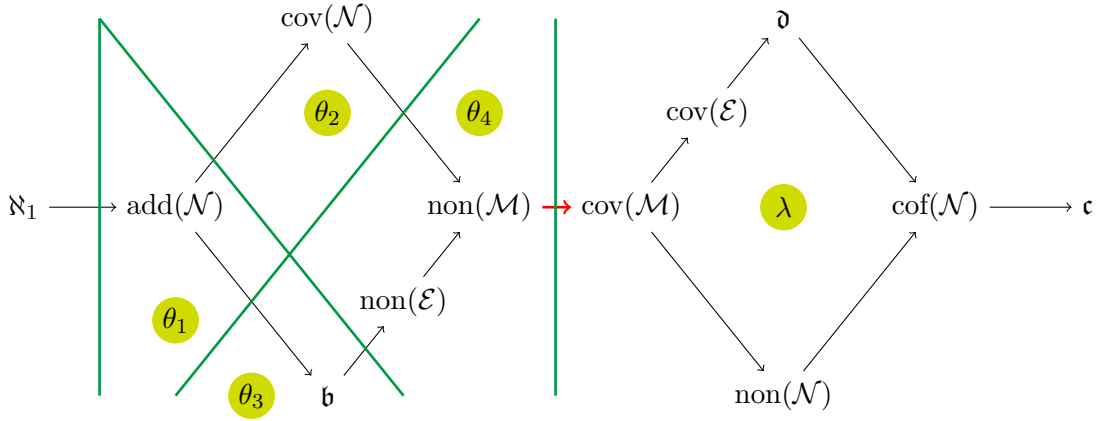
\begin{figure}[ht]
\centering
\begin{tikzpicture}
\small{
 \node (aleph1) at (-2,2.5) {$\aleph_1$};
 \node (addn) at (0,2.5){$\add(\Nwf)$};
 \node (covn) at (2,5){$\cov(\Nwf)$};
 \node (cove) at (7,3.75) {$\cov(\Ecal)$};
 \node (b) at (2,0) {$\bfrak$};
 \node (nonm) at (4,2.5) {$\non(\Mcal)$} ;
 \node (none) at (3,1.25) {$\non(\Ecal)$};
 \node (d) at (8,5) {$\dfrak$};
 \node (covm) at (6,2.5) {$\cov(\Mcal)$} ;
 \node (nonn) at (8,0) {$\non(\Ncal)$} ;
 \node (cfn) at (10,2.5) {$\cof(\Ncal)$} ;
  \node (c) at (12,2.5) {$\cfrak$};

\foreach \from/\to in {
aleph1/addn, addn/covn, addn/b, covn/nonm, b/none, none/nonm, covm/nonn, covm/cove, cove/d, nonn/cfn, d/cfn, cfn/c}
{
\path[-,draw=white,line width=3pt] (\from) edge (\to);
\path[->,] (\from) edge (\to);
}
\path[->,draw=red,line width=1pt] (nonm) edge (covm);

\draw[color=sug,line width=1]
(-1,0)--(-1,5)
(-1,5)--(3,0)
(0,0)--(4,5)
(5,5)--(5,0);

\draw[circle, fill=suy,color=suy] 
(0,1) circle (0.3)
(2,3.75) circle (0.3)
(1,0) circle (0.3)
(4,3.75) circle (0.3)
(8,2.5) circle (0.3);

\node at (0,1) {$\theta_1$};
\node at (2,3.75) {$\theta_2$};
\node at (1,0) {$\theta_3$};
\node at (4,3.75) {$\theta_4$};
\node at (8,2.5) {$\lambda$};

}
\end{tikzpicture}
\caption{The constellation of Cicho\'n's diagram forced in \autoref{main1}. For $\bfrak$ we can only guarantee that $\bfrak\leq\theta_3$, but the converse inequality is settled in \autoref{main2}.}\label{fig:main1}
\end{figure}

\section{Preservation theory for cardinal characteristics}\label{sec:pres}

We deal with the problem of forcing $\bfrak(R)\leq \theta$ (and much more) to conclude the proof of \autoref{main1}. We proceed in two steps: we first add a strong type of $R$-unbounded family (using Cohen reals), and then show that this strong unbounded family is not destroyed in the remaining part of the iteration. 

The strong type of unbounded family is defined as follows.

\begin{definition}
    Let $R=\la X,Y,\sqsubset\ra$ be a relational system, and $\theta$ an infinite  cardinal. We say that $\set{x_i}{i\in I}\subseteq X$ is a \emph{$\theta$-$R$-unbounded family} if $|I|\geq \theta$ and $|\set{i\in I}{x_i\sqsubset y}|<\theta$ for all $y\in Y$.
\end{definition}

Although a $\theta$-$R$-unbounded family is quite large, it has the property that any subset of size $\theta$ is $R$-unbounded, which guarantees $\bfrak(R)\leq \theta$. But we get much more, as indicated in the following result.

\begin{lemma}\label{strongTukey}
    Assume that $|I|\geq \theta$. Then there exists a $\theta$-$R$-unbounded family $\set{x_i}{i\in I}$ iff $\Cbf_{[I]^{<\theta}}\leqT R$. In particular, $\bfrak(R)\leq \non([I]^{<\theta}) = \theta$ and $\cov([I]^{<\theta})\leq \dfrak(R)$.
\end{lemma}
\begin{proof}
    If $\set{x_i}{i\in I}$ is a $\theta$-$R$-unbounded family then the maps 
    \[i\mapsto x_i \text{ and } y\mapsto \set{i\in I}{x_i\sqsubset y}\] 
    yield the desired Tukey connection.

    Conversely, assume that $\Cbf_{[I]^{<\theta}}\leqT R$ is witnessed by the Tukey connection $i\mapsto x_i$ and $y\mapsto A_y\in [I]^{<\theta}$, i.e.\ $x_i\sqsubset y$ implies $i\in A_y$. Therefore $\set{i\in I}{x_i\sqsubset y} \subseteq A_y$, so it has size ${<}\theta$. Hence, $\set{x_i}{i\in I}$ is a $\theta$-$R$-unbounded family.
\end{proof}

Concerning $\cov([I]^{<\theta})$, we have
\[\cov([I]^{<\theta}) = 
\begin{cases}
    |I| & \text{if $\theta<|I|$,}\\
    \cf(\theta) & \text{if $\theta=|I|$.}
\end{cases}\]
Hence $\cov([I]^{<\theta})=|I|$ when $\theta$ is regular.

%
%
%
%

Recall from the previous section the Polish relational systems describing the cardinal characteristics in Cicho\'n's diagram, as illustrated in \autoref{fig:diamRS}.

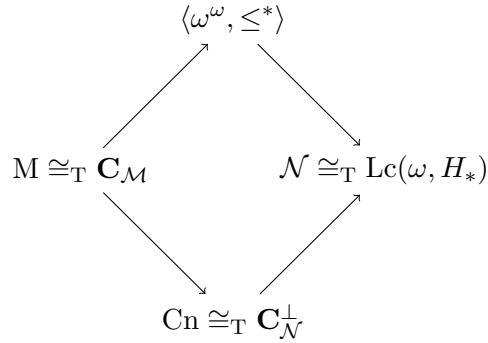
\begin{figure}[ht]
\centering
\begin{tikzpicture}
\small{
  \node (CM) at (0,0) {$\Mg \eqT \Cbf_\Mwf$};
  \node (CNp) at (2,-2) {$\Cn \eqT \Cbf_\Nwf^\perp$}; 
  \node (d) at (2,2) {$\la \omega^\omega,\leq^*\ra$};
  \node (Lc) at (4,0) {$\Nwf \eqT \Lc(\omega, H_*)$};
\foreach \from/\to in {
CM/CNp, CM/d, CNp/Lc, d/Lc}
{
\path[->,] (\from) edge (\to);
}
}
\end{tikzpicture}
\caption{Tukey connections between the relational systems determining the non-dependant values in Cicho\'n's diagram, along with their equivalent Polish relational systems.}\label{fig:diamRS}
\end{figure}

For the rest of this section, we fix a Polish relational system $R = \la X,Y,\sqsubset\ra$. In this case, $\theta$-$R$-unbounded families can easily be added using Cohen reals.

\begin{lemma}\label{Cohenthetaunb}
   Let $\lambda$ be an uncountable cardinal. Then the Cohen reals $\set{c_\alpha}{\alpha<\lambda}\subseteq X$ added by $\Cor_\lambda$ form an $\aleph_1$-$R$-unbounded family in $V^{\Cor_\lambda}$.  
\end{lemma}
\begin{proof}
    Working in $V^{\Cor_\lambda}$, let $y\in Y$. Since $y$ is a real, it only depends on countably many maximal antichains, so there is some $C\in [\lambda]^{<\aleph_1}\cap V$ such that $y\in V^{\Cor_{C}}$. For any $\alpha\in \lambda\menos C$, $c_\alpha$ is Cohen over $V^{\Cor_{C}}$, hence $R$-unbounded over $V^{\Cor_{C}}$ by \autoref{Cohenunb}, so $c_\alpha\not\sqsubset y$. Therefore, $\set{\alpha<\lambda}{c_\alpha \sqsubset y}\subseteq C$, which is countable.
\end{proof}

Note that $\theta\leq\theta'$ implies that any $\theta$-$R$-unbounded family $\set{x_i}{i\in I}$ is $\theta'$-$R$-unbounded, as long as $\theta'\leq |I|$. Therefore, the Cohen reals added by $\Cor_\lambda$ form a $\theta$-$R$-unbounded family for all $\aleph_1\leq\theta\leq\lambda$.

The reason we start with $\Cor_\lambda$ in Step 1 of the proof of \autoref{main1} is to add $\theta_i$-unbounded families. Now we aim to show how to preserve them in the iteration of Step 2. For this purpose, we introduce the preservation theory from Judah and Shelah~\cite{JS} and Brendle~\cite{Br}.

\begin{definition}\label{def:good}
Let $\kappa$ be an infinite cardinal.
A poset $\Por$ is \textit{$\kappa$-$R$-good} if, for any $\Por$-name $\dot{y}$ for a member of $Y$, there is a non-empty set $H\subseteq Y$ (in the ground model) of size ${<}\kappa$ such that, for any $x\in X$, if $x$ is $R$-unbounded over $H$ then $\Vdash x\not\sqsubset \dot{h}$.

We say that $\Por$ is \textit{$R$-good} if it is $\aleph_1$-$R$-good.
\end{definition}

Note that $\kappa\leq\kappa'$ implies that any $\kappa$-$R$-good poset is $\kappa'$-$R$-good.

Goodness guarantees the preservation of strong unbounded families as follows.

\begin{lemma}\label{mainpres}
    If $\kappa$ and $\theta$ are infinite cardinals, and $\kappa\leq\cf(\theta)$, then any $\kappa$-$R$-good poset preserves all the $\theta$-$R$-unbounded families from the ground model.
\end{lemma}
\begin{proof}
    Let $\Por$ be a $\kappa$-$R$-good poset. 
    Assume that $\set{x_i}{i\in I}\subseteq X$ is a $\theta$-$R$-unbounded family. Let $\dot y$ be a $\Por$-name of a member of $Y$. Find $H\in [Y]^{<\kappa}$ non-empty as in \autoref{def:good}. For each $y\in H$ let $A^y:=\set{i\in I}{x_i\sqsubset y}$ and $A:= \bigcup_{y\in H}A^y$. Then $|A^y|<\theta$ and $|A|<\theta$, the latter because $\cf(\theta)\geq \kappa$.

    We claim that $\Vdash \set{i\in I}{x_i\sqsubset \dot y}\subseteq A$. Indeed, if $i\in I\menos A$, $x_i\not\sqsubset y$ for all $y\in H$, so $\Vdash x_i\not \sqsubset \dot y$. 
\end{proof}

Now, goodness is preserved along FS iterations.

\begin{theorem}\label{goodit}
Let $\kappa$ be an uncountable regular cardinal. Then,
any FS iteration of $\kappa$-cc $\kappa$-$R$-good posets is again $\kappa$-$R$-good.
\end{theorem}
\begin{proof}
See e.g.~\cite[Thm.~4.15]{CM}.
\end{proof}

This result can be weakened as follows.

\begin{theorem}
    Let $\kappa$ and $\theta$ be uncountable cardinals such that $\kappa$ is regular and $\cf(\theta)\geq\kappa$. Then, any FS iteration of $\kappa$-cc posets preserving $\theta$-$R$-unbounded families, preserves $\theta$-$R$-unbounded families.
\end{theorem}
\begin{proof}
   Let $\la \Por_\alpha,\Qnm_\beta\colon \alpha\leq\pi,\ \beta<\pi\ra$ be a FS iteration of $\kappa$-cc posets preserving $\theta$-$R$-unbounded families, and let $\set{x_i}{i\in I}$ be a $\theta$-$R$-unbounded family (in $V$).   
   We show by recursion on $\alpha\leq\pi$ that $\Por_\alpha$ forces that $\set{x_i}{i\in I}$ is $\theta$-$R$-unbounded. This is clear for $\alpha=0$ and for the successor step. 
   
   The limit step $\alpha$ with $\cf(\alpha)\geq\kappa$ is easy: if $y\in Y\cap V_\alpha$ then $y\in V_\xi$ for some $\xi<\alpha$ (becase a nice name of the real $y$ depends on ${<}\kappa$-many conditions, cf.~\autoref{FSnoreal}), so $V_\xi\models |\set{i\in I}{x_i\sqsubset y}|<\theta$ by inductive hypothesis. The same is satisfied in $V_\alpha$.
   
   We have to work more when $\cf(\alpha)<\kappa$. Fix an increasing cofinal sequence $\set{\alpha_\xi}{\xi<\cf(\alpha)}$ in $\alpha$. Since $Y$ is analytic, there is some continuous surjection $f\colon\baire\to Y$. Let $\dot y$ be a $\Por_\alpha$-name of a member of $Y$ and pick some $\Por_\alpha$-name $\dot z$ of a real in $\baire$ such that $\Vdash_\alpha f(\dot z) = \dot y$. For each $\xi<\cf(\alpha)$ pick $\Por_{\alpha_\xi}$-names $\dot z^\xi$ and $\seq{\dot p^\xi_n}{n<\omega}$ such that $\Por_{\alpha_\xi}$ forces that $\seq{\dot p^\xi_n}{n<\omega}$ is a decreasing sequence in $\Por_\alpha/\Por_{\alpha_\xi}$, $\dot z^\xi\in\baire$, and $\dot p^\xi_n\Vdash_{\Por_\alpha/\Por_{\alpha_\xi}} \dot z\frestr n = \dot z^\xi\frestr n$ for all $n<\omega$. By induction hypothesis, $\Vdash_{\alpha_\xi} |\set{i\in I}{x_i\sqsubset f(\dot z^\xi)}|<\theta$. 
   
   To finish the proof, if is enough to show that, for any $i\in I$, $\Por_\alpha$ forces that $x_i\not\sqsubset f(\dot z)$ whenever $x_i$ is $R$-unbounded over $\set{f(\dot z^\xi)}{\xi<\cf(\alpha)}$ (because $\cf(\alpha)<\kappa\leq\cf(\theta)$). Fix $p\in\Por_\alpha$, $\ell<\omega$ and assume that $p\Vdash x_i\sqsubset_\ell f(\dot z)$. Find some $\xi<\cf(\alpha)$ such that $p\in\Por_{\alpha_\xi}$. Let $G_{\alpha_\xi}$ be a $\Por_{\alpha_\xi}$-generic set containing $p$ and work in $ V_{\alpha_\xi} = V[G_{\alpha_\xi}]$. Since $\Vdash_{\Por_\alpha/\Por_{\alpha_\xi}} x_i\sqsubset_\ell f(\dot z)$, we have that $x\sqsubset_\ell f(\dot z^\xi[G_{\alpha_\xi}])$ (otherwise, by~\cite[Lem.~4.9]{CM}, there would be some $n<\omega$ such that $\dot p^\xi_n[G_{\alpha_\xi}]$ forces $x_i\not\sqsubset_\ell f(\dot z)$). Therefore, in $V$, $p\Vdash_{\alpha_\xi} x_i\sqsubset_\ell f(\dot z^\xi)$.   
\end{proof}

We now turn to particular cases. One very useful fact is that small posets are good.

\begin{lemma}\label{smallgood}
    Any poset $\Por$ is $\kappa$-$R$-good for any infinite $\kappa>|\Por|$. 
    
    In particular, Cohen forcing is $\kappa$-$R$-good for all uncountable $\kappa$.
\end{lemma}
\begin{proof}
    See~\cite[Lem.~4]{M}, also~\cite[Lem.~4.10]{CM}.
\end{proof}

More concrete examples of $R$-good posets come from the connection between the combinatorics of a forcing and $R$. We formalize this with the following notions.

\begin{definition}[{\cite{mejiavert}}]\label{gamma-lk}
  We say that $\Gamma$ is a \emph{linkedness property} if $\Gamma(\Por)\subseteq \pts(\Por)$ for any poset $\Por$.

  Let $\mu$ and $\kappa$ be infinite cardinals.
  \begin{enumerate}[label = \normalfont (\arabic*)]
    \item A poset $\Por$ is \emph{$\mu$-$\Gamma$-linked} if it can be covered by ${\leq}\mu$-many subsets in $\Gamma(\Por)$.

    When $\mu = \aleph_0$, we write \emph{$\sigma$-$\Gamma$-linked}.
    \item A poset $\Por$ is \emph{$\kappa$-$\Gamma$-Knaster} if
        $\forall\, B\in[\Por]^\kappa\ \exists\, A\in[B]^\kappa\colon A\in\Gamma(\Por).$

        When $\kappa=\aleph_1$, we just write \emph{$\Gamma$-Knaster}.
  \end{enumerate}
\end{definition}

If $\Gamma$ satisfies that $Q'\subseteq Q\in\Gamma(\Por)$ implies $Q'\in\Gamma(\Por)$, then any $\mu$-$\Gamma$-linked poset is $\mu^+$-$\Gamma$-Knaster. A more concrete discussion about linkedness properties and iterations can be found in~\cite[Sec.~5]{mejiavert}.

\begin{example}
    The following are examples of linkedness properties. 
    Here, $\Por$ denotes an arbitrary poset.
    \begin{enumerate}[label = \normalfont (\arabic*)]
        \item $\Lambda_{<\omega}$: \emph{Centered}. $Q\in\Lambda_{<\omega}(\Por)$ iff $Q$ is a centered subset of $\Por$, i.e.\ for any finite $F\subseteq Q$, there is a $q\in\Por$ stronger than all members of $F$.

        Then, $\mu$-$\Lambda_{<\omega}$-linked means $\mu$-centered, and $\kappa$-$\Lambda_{<\omega}$-Knaster means precaliber $\kappa$.

        \item $\Lambda_{\inter}$: \emph{Positive intersection number}. For $n<\omega$ non-zero and $s\in\Por^n$, define
        \[\iota_*(s):= \max\lset{|e|}{e\subseteq n \text{ and }\set{s_i}{i\in e} \text{ has a lower bound in }\Por}.\]
        For $Q\subseteq\Por$, define \emph{the intersection number of $Q$ in $\Por$} by
        \[\inter_\Por(Q):=\inf\lset{\frac{\iota_*(s)}{n}}{s\in Q^n,\ 0<n<\omega}.\]

        We say that $Q\in\Lambda_{\inter}(\Por)$ iff $\inter_\Por(Q)>0$.
    \end{enumerate}
    Notice that $\Lambda_{<\omega}(\Por)\subseteq \Lambda_{\inter}(\Por)$ because any centered poset has intersection number $1$.
\end{example}

According to the following result, \emph{$\Lambda_{<\omega}$ is good for $\Cn$}:

\begin{theorem}[Brendle~{\cite{Br}}]\label{centCn}
    Any $\mu$-centered poset is $\mu^+$-$\Cn$-good.
    In particular, any $\sigma$-centered poset is $\Cn$-good.
\end{theorem}

Inspired by a result of Kamburelis~\cite{Ka}, we have that \emph{$\Lambda_{\inter}$ is good for $\Lc(\omega,H_*)$}. Recall that $H_* = \set{\id^{k+1}}{k<\omega}$.

\begin{theorem}\label{interLc}
    Any $\mu$-$\Lambda_{\inter}$-linked poset is $\mu^+$-$\Lc(\omega,H_*)$-good. 
\end{theorem}
\begin{proof}
	 Let $\Por$ be a $\mu$-$\Lambda_{\inter}$-linked poset witnessed by $\seq{Q_\alpha}{\alpha<\mu}$ and let $\dot \varphi$ be a $\Por$-name of a member of $\Scal(\omega,H_*)$. Notice that $\Por$ is $\mu^+$-cc since no member of $\Lambda_{\inter}(\Por)$ contains infinite antichains. So find a maximal antichain $A\subseteq \Por$ and a sequence $\seq{k_p}{p\in A}$ of natural numbers such that $p\Vdash\dot\varphi\in \Scal(\omega,\id^{k_p+1})$ for all $p\in A$. 
    
    Fix $\alpha<\mu$ and $p\in A$. Since $\inter_\Por(Q_\alpha)>0$, there is some $m_\alpha<\omega$ satisfying $\frac{1}{m_\alpha}< \inter_\Por(Q_\alpha)$. Define
    \[\varphi_{\alpha,p}(n):=\left\{
    \begin{array}{ll}
         \set{i<\omega}{\exists\, q\in Q_\alpha\colon q\leq p,\ q\Vdash i\in\dot\varphi(n)} & \text{if $n\geq m_\alpha$,}\\
         \emptyset & \text{if $n<m_\alpha$.}
    \end{array}\right.\]
    We show that $|\varphi_{\alpha,p}(n)|\leq n^{k_p+2}$ for all $n<\omega$, i.e.\ $\varphi_{\alpha,p}\in\Scal(\omega,H_*)$. This is clear for $n<m_\alpha$. In the case $n\geq m_\alpha$, for each $i\in\varphi_{\alpha,p}(n)$ pick some $q_i\in Q_\alpha$, stronger than $p$, forcing $i\in\dot\varphi(n)$, and let $\bar q:=\seq{q_i}{i\in\varphi_{\alpha,p}(n)}$. If $|\varphi_{\alpha,p}(n)|\geq n^{k_p+2}$ then 
    \[\iota_*(\bar q)\geq \inter_\Por(Q) |\varphi_{\alpha,p}(n)|\geq \inter_\Por(Q) n^{k_p+2} > n^{k_p+1},\]
    so there is some $q'\in \Por$ such that $a:=\set{i\in\varphi_{\alpha,p}(n)}{q'\leq q_i}$ has size ${>}\, n^{k_p+1}$. Then, $q'\Vdash a\subseteq \dot\varphi(n)$, which contradicts that $q'\Vdash |\dot \varphi(n)|\leq n^{k_p+1}$ (because $q'\leq p$). Therefore, $|\varphi_{\alpha,p}(n)|<n^{k_p+2}$.

    We show that $S:=\set{\varphi_{\alpha,p}}{\alpha<\mu,\ p\in A}$ (which has size ${\leq}\mu$) witnesses goodness. Let $x\in\baire$ be $\Lc(\omega,H_*)$-unbounded over $S$, $n_0<\omega$ and $p_0\in\Por$. Then $p_0$ is compatible with some $p\in A$, so there is a common stronger condition $q\in\Por$. Now, pick an $\alpha<\mu$ such that $q\in Q_\alpha$. Then, there is some $n\geq\max\{m_\alpha,n_0\}$ such that $x(n)\notin\varphi_{\alpha,p}(n)$, so $q\nVdash x(n)\in\dot\varphi(n)$, i.e.\ there is some $q'\leq q$ forcing $x(n)\notin \dot\varphi(n)$. This shows that $\Vdash x\notin^* \dot \varphi$.
\end{proof}

\begin{corollary}\label{centLc}
    Any $\mu$-centered poset is $\mu^+$-$\Lc(\omega,H_*)$-good.
\end{corollary}

Other examples are obtained using Boolean algebras with finitely additive measures.

\begin{theorem}[Kelley~\cite{Kelley}]\label{thm:kelley}
    Let $\Bor$ be a Boolean algebra. Then $\Bor\menos\{0_\Bor\}$ is $\sigma$-$\Lambda_{\inter}$-linked iff there is a strictly positive fam $\Xi\colon \Bor\to[0,1]$ (i.e. $\Xi(b)=0$ iff $b=0$).
\end{theorem}

In combination with \autoref{interLc}, we obtain

\begin{corollary}\label{spfamLc}
    If $N$ is a transitive model of $\thzfc$, then $(\Bwf(2^\omega)\menos \Nwf(2^\omega))^N$ is $\Lc(\omega,H_*)$-good.
\end{corollary}

In the next section, we will present a good linkedness property for $\la \omega^\omega,\leq^*\ra$. For the moment, we present the following examples.

\begin{theorem}[Miller~{\cite{Mi1981}}]\label{Egood}
    $\Ebb$ is $\la\omega^\omega,\leq^*\ra$-good.
\end{theorem}

\begin{theorem}\label{Bgood}
    Random forcing is $\la\omega^\omega,\leq^*\ra$-good.\footnote{This easily follows from the fact that random forcing is ccc and $\omega^\omega$-bounding.}
\end{theorem}

We are finally ready to conclude the proof of \autoref{main1}.

\begin{proof}[Proof of \autoref{main1}, part~2.]
    It remains to show that, in $V_\lambda$, $\add(\Nwf) \leq \theta_1$, $\cov(\Nwf)\leq \theta_2$, $\non(\Mwf)\leq \theta_4$ and $\lambda\leq \cov(\Mwf)$.

    By \autoref{Cohenthetaunb}, the Cohen reals added at step~1 give us $\aleph_1$-$R$-unbounded families of size $\lambda$ for any Polish relational system $R$, in particular, we obtain in $V^{\Cor_\lambda}$ an $\aleph_1$-$\Lc(\omega,H_*)$-unbounded $\set{c^1_\alpha}{\alpha<\lambda}$, an $\aleph_1$-$\Cn$-unbounded $\set{c^2_\alpha}{\alpha<\lambda}$, and an $\aleph_1$-$\Mg$-unbounded $\set{c^4_\alpha}{\alpha<\lambda}$. Now, if we prove that the iteration of step~2 is $\theta_1$-$\Lc(\omega,H_*)$-good, $\theta_2$-$\Cn$-good and $\theta_4$-$\Mg$-good, we obtain by \autoref{mainpres} that the previous families are, in the final extension, $\theta_1$-$\Lc(\omega,H_*)$-unbounded, $\theta_2$-$\Cn$-unbounded, and $\theta_4$-$\Mg$-unbounded, respectively. Therefore, by \autoref{strongTukey}, $\add(\Nwf)=\bfrak(\Lc(\omega,H_*))\leq\theta_1$, $\cof(\Nwf)=\bfrak(\Cn)\leq\theta_2$, $\non(\Mwf) = \bfrak(\Mg) \leq \theta_4$ and $\cov(\Mwf) =\dfrak(\Mg)\geq \cov([\lambda]^{<\theta_4}) = \lambda$.

    By virtue of \autoref{goodit}, it is enough to prove that all the iterands used in step~2 are $\theta_1$-$\Lc(\omega,H_*)$-good, $\theta_2$-$\Cn$-good and $\theta_4$-$\Mg$-good. Indeed, for:
    \begin{enumerate}[label = $\alpha\in K_{\arabic*}$:, left=0pt]
        \item $\Qor_\alpha = \Loc_\id^{N_\alpha}$ has size ${<\theta_1}$ because $|N_\alpha|<\theta_1$, so it is $\theta_1$-$R$-good (and $\kappa$-$R$-good for any $\kappa\geq\theta_1$) for any Polish relational system $R$ (by \autoref{smallgood}).

        \item $\Qor_\alpha = (\Bwf(2^\omega)/\Nwf(2^\omega))^{N_\alpha}$ has size ${<}\theta_2$, so it is $\theta_2$-$R$-good for any Polish relational system $R$. On the other hand, by \autoref{spfamLc}, $\Qor_\alpha$ is $\Lc(\omega,H_*)$-good.

        \item $\Qor_\alpha = \Dor^{N_\alpha}$ has size ${<}\theta_3$, so it is $\theta_3$-$R$-good for any Polish relational system $R$. On the other hand, $\Qor_\alpha$ is $\Lc(\omega,H_*)$-good and $\Cn$-good by \autoref{centCn} and~\autoref{centLc}, respectively.

        \item $\Qor_\alpha = \Ebb^{N_\alpha}$ has size ${<}\theta_4$, so it is $\theta_4$-$R$-good for any Polish relational system $R$. On the other hand, $\Qor_\alpha$ is $\Lc(\omega,H_*)$-good and $\Cn$-good by \autoref{centCn} and~\autoref{centLc}, respectively.\qedhere
    \end{enumerate}    
\end{proof}

In the previous proof, we have that $\Qor_\alpha$ is $\theta_3$-$\la\omega^\omega,\leq^*\ra$-good for $\alpha\in K_1\cup K_2\cup K_3$. However, although $\Ebb$ is $\la\omega^\omega,\leq^*\ra$-good, there are examples of restrictions of the form $\Ebb^N$ for transitive models $N$ of $\thzfc$ that are not $\la\omega^\omega,\leq^*\ra$-good. 

\begin{theorem}[Pawlikowski~{\cite{Paw-Dom}}]\label{Pawctr}
    There is a proper-$\omega^\omega$-bounding generic extension $W$ of $V$ in which $\Ebb^V$ and $(\Bwf(2^\omega)\menos \Nwf(2^\omega))^V$ add dominating reals over $W$.
\end{theorem}

Even more, in the case of random forcing:

\begin{theorem}[Judah and Shelah~{\cite{JShDOm}}]\label{JSctr}
    There is a ccc forcing extension $W$ of $V$ such that $(\Bwf(2^\omega)\menos \Nwf(2^\omega))^V$ adds dominating reals over $W$.
\end{theorem}

In the next section, we modify the forcing construction of \autoref{main1} for $\alpha\in K_4$ to guarantee that the set of Cohen reals $\set{c^3_\alpha}{\alpha<\lambda}$ added by $\Cor_\lambda$ stays $\theta_3$-$\la\omega^\omega,\leq^*\ra$-unbounded in the final extension.

\section{FS iterations with measures and ultrafilters on the natural numbers}\label{sec:uflim}

\newcommand{\Fr}{\mathrm{Fr}}

We show how to modify the iteration in \autoref{main1} to force, in addition, $\bfrak\leq\theta_3$. We start by introducing the following good property for $\la \omega^\omega,\leq^*\ra$.

\begin{definition}[{\cite{mejiavert,BCM}}]\label{def:Fr-l}
    Let $F\subseteq\pts(\omega)$ be a filter. We assume that all filters are \emph{free}, i.e.\ they contain the \emph{Frechet filter} $\Fr:=\set{\omega\menos a}{a\in[\omega]^{<\aleph_0}}$. A set $a\subseteq \omega$ is \emph{$F$-positive} if it intersects every member of $F$. Denote by $F^+$ the collection of $F$-positive sets. 

    We define the linkedness property  $\Lambda_F$, which we call \emph{$F$-linked}: given a poset $\Por$ and $Q\subseteq \Por$, $Q$ is $F$-linked if, for any $\la p_n\colon n<\omega\ra\in Q^\omega$, there is some $q\in \Por$ such that
    \[q\Vdash \{n<\omega\colon p_n\in \dot G\} \in F^+,\text{ i.e.\ it intersects every member of $F$.}\]
    Note that, in the case $F=\Fr$, the previous equation is ``$q\Vdash \{n<\omega\colon p_n\in \dot G\}$ is infinite".

    We also define $\Lambda_{\rm uf}$, which we call \emph{uf-linked} (\emph{ultrafilter-linked}): $Q\in \Lambda_{\rm uf}(\Por)$ if $Q\in\Lambda_F(\Por)$ for every (ultra)filter $F$ on $\omega$.
\end{definition}

If $F\subseteq F'$ are filters on $\omega$, it is clear that $\Lambda_{\rm uf}(\Por) \subseteq \Lambda_{F'}(\Por) \subseteq \Lambda_{F}(\Por) \subseteq \Lambda_\Fr(\Por)$. But, for ccc posets:

\begin{lemma}[{\cite{mejiavert}}]\label{lem:Frccc}
    If $\Por$ is ccc then $\Lambda_{\rm uf}(\Por) = \Lambda_\Fr(\Por)$.
\end{lemma}
\begin{proof}
	Assume that $F$ is a filter on $\omega$ and that $Q\subseteq\Por$ is not $F$-linked in $\Por$. We show that $Q$ is not $\Fr$-linked. Since $Q$ is not $F$-linked, 
there is some sequence $\la p_n\colon n<\omega\ra \in Q^\omega$ such that 	
$\Por$ forces that $\dot w:=\{n<\omega \colon p_n\in \dot G\}$ does not intersect some member of $F$. Since $F$ is in the ground model, we can find a maximal antichain $A$ in $\Por$ deciding a member of $F$ disjoint with $\dot w$, namely, there is some $h\colon A\to F$ such that each $p\in A$ forces $\dot w\cap h(p)=\emptyset$. 
	
	Since $A$ is countable and $\pfrak$ is uncountable, we can find some $c\in[\omega]^{\aleph_0}$ such that $c\subseteq^* h(p)$ for all $p\in A$. In detail: if $A$ is finite, let $c:=\bigcap_{p\in A}h(p)$; otherwise, enumerate $A=\{p_n \colon n<\omega\}$ and, by recursion, construct an increasing sequence $\la k_n \colon n<\omega\ra$ of natural numbers such that $k_n\in\bigcap_{\ell\leq n}h(p)$ for each $n<\omega$, and set $c:=\{k_n\colon n<\omega\}$.
	
	Thus, each $p\in A$ forces $\dot w\cap c \subseteq^* \dot w\cap h(p)=\emptyset$, i.e.\ $\dot w\cap c$ is finite. Since $A$ is a maximal antichain, $\Por$ forces that $\dot w\cap c = \{n\in c\colon p_n\in\dot G\}$ is finite. Therefore, $\la p_n\colon n\in c\ra$ witnesses that $Q$ is not $\Fr$-linked in $\Por$. 
\end{proof}

\begin{example}\label{exm:uf-l}
    \ 
    \begin{enumerate}[label = \normalfont (\arabic*)]
        \item Any singleton is uf-linked. Hence, any poset $\Por$ is $|\Por|$-uf-linked. In particular, Cohen forcing is $\sigma$-uf-linked.

        \item Random forcing is $\sigma$-uf-linked, in fact, any measure algebra is $\sigma$-uf-linked. Indeed, if $\Bor$ is a complete Boolean algebra and $\mu\colon \Bor\to [0,1]$ is a $\sigma$-additive measure such that $\mu(p)\neq 0$ for all $p\neq 0_\Bor$, then,
        for any $\delta>0$, $B_\delta:=\set{p\in\Bor}{\mu(p)\geq \delta}$ is $\Fr$-linked.
        
        We use the characterization in \autoref{lemFr} to show that $B_\delta$ is $\Fr$-linked. Let $\dot m$ be a $\Bor$-name of a natural number. Then $\sum_{n<\omega}\mu(\|\dot m = n\|_\Bor)=1$, so there is some $m'<\omega$ such that $\sum_{n\geq m'}\mu(\|\dot m = n\|_\Bor)<\delta$. As a consequence, $\mu(\|\dot m< m' \|_\Bor)>1-\delta$. Since any $q\in B_\delta$ is compatible with $\|\dot m< m' \|_\Bor$, it follows that $q$ does not force $m'\leq \dot m$.

        \item The forcing $\Ebb$ (see \autoref{edposet}) is $\sigma$-uf-linked. We show later that this poset satisfies a stronger property.
    \end{enumerate}
\end{example}

The following series of results indicate that $\Lambda_\Fr$ is good for $\la\omega^\omega,\leq^*\ra$.

\begin{lemma}\label{lemFr}
    Let $\Por$ be a poset and $Q\subseteq\Por$. Then $Q$ is $\Fr$-linked iff, for any $\Por$-name $\dot m$ for a natural number, there is some $m'\in\omega$ (in the ground model) such that no $p\in Q$ forces $m'\leq \dot m$.
\end{lemma}
\begin{proof}
    $(\Rightarrow)$~\cite{mejiavert} Assume that, for any $n<\omega$, there is some $p_n\in Q$ forcing $n\leq \dot m$. Then, if $G$ is $\Por$-generic over $V$, then $\set{n<\omega}{p_n\in G}$ must be finite because $p_n\in G \imp n\leq \dot m[G]<\omega$. Therefore, in $V$, $Q$ cannot be $\Fr$-linked.

    $(\Leftarrow)$ (with Cardona) Assume that $Q$ is not $\Fr$-linked, so there is some $\la p_n\colon n<\omega\ra \in Q^\omega$ such that $\Vdash$``$\{n<\omega\colon p_n\in \dot G\}$ is finite". So pick some $\Por$-name $\dot m$ of a natural number such that $\Vdash$``$\{n<\omega\colon p_n\in \dot G\}\subseteq \dot m$. Note that $p_n\Vdash n<\dot m$.
\end{proof}

\begin{lemma}\label{Frunb}
    Let $\Por$ be a poset and $Q$ be an $\Fr$-linked subset of $\Por$. If $\dot y$ is a $\Por$-name of a member of $\omega^\omega$, then there is some $y'\in \omega^\omega$ (in the ground model) such that, for any $x\in\omega^\omega$
    \[x\nleq^* y' \imp \forall\, n<\omega\ \forall\, p\in Q\colon p\nVdash \forall\, k\geq n\colon x(k) \leq \dot y(k).\]
\end{lemma}
\begin{proof}
    Using \autoref{lemFr}, for each $k<\omega$ find $y'(k)<\omega$ such that no $p\in Q$ forces $y'(k) \leq \dot y(k)$. This defines $y'\in \omega^\omega$.

    Now assume that $x\in \omega^\omega$ and $x\nleq^* y'$. Let $n<\omega$ and $p\in Q$, so there is some $k\geq n$ such that $x(k) > y'(k)$. On the other hand, $p\nVdash y'(k) \leq \dot y(k)$, so there is some $q\leq p$ forcing $\dot y(k) < y'(k) < x(k)$, so $p\nVdash \forall\, k\geq n\colon x(k) \leq \dot y(k)$.
\end{proof}

\begin{theorem}[{\cite{mejiavert}}]\label{Frgood}
    Any $\mu$-$\Fr$-linked poset is $\mu^+$-$\omega^\omega$-good.
\end{theorem}

This theorem is an easy consequence of \autoref{Frunb}. However, we do not know how to modify the construction in \autoref{main1} to obtain a $\theta_3$-$\omega^\omega$-good iteration. But we have some other way to preserve unbounded families, as in the following result.

\begin{theorem}[{\cite{BCM}}]\label{FrKpres}
   Let $\theta$ be an uncountable regular cardinal. Then any $\theta$-$\Fr$-Knaster poset preserves $\theta$-$\omega^\omega$-unbounded families.
\end{theorem}
\begin{proof}
    Assume that $\set{x_i}{i\in I}$ is a $\theta$-$\omega^\omega$-unbounded family, and that there is some $p\in\Por$ forcing that it is not, i.e.\ for some $\Por$-name $\dot y$ of a member of $\omega^\omega$, $p$ forces that $|\set{i\in I}{x_i\leq^* \dot y}|\geq \theta$. This implies that the set
    \[I_0 := \lset{i\in I}{\exists\, p'\leq p\colon p'\Vdash x_i \leq^* \dot y}\]
    has size ${\geq}\theta$. Pick $I_1\subseteq I_0$ of size $\theta$ and, for each $i\in I_1$, choose $p_i\leq p$ and $n_i<\omega$ such that $p_i\Vdash x_i(k) \leq \dot y(k)$ for all $k\geq n_i$. Since $\cf(\theta)>\omega$, we can find $n<\omega$ and $I_2\subseteq I_1$ of size $\theta$ such that $n_i = n$ for all $i\in I_2$.

    Since $\theta$ is regular and $\Por$ is $\theta$-$\Fr$-Knaster, there is some $I'\subseteq I_2$ such that the set $Q:=\set{p_i}{i\in I'}$ is $\Fr$-linked. Now find $y'\in\omega^\omega$ as in \autoref{Frunb} for $\dot y$ and $Q$. Then, $|\set{i\in I}{x_i\leq^* y'}|<\theta$, so there is some $i\in I'$ such that $x_i\nleq^* y'$. Hence, by \autoref{Frunb}, no $p\in Q$ forces $\forall\, k\geq n\colon x_i(k) \leq \dot y(k)$. But $p_i$ forces this, a contradiction.
\end{proof}

Because of the previous theorem, the plan now is to modify the construction of \autoref{main1} to obtain a $\theta_3$-$\Fr$-Knaster poset. To achieve this, we strengthen ultrafilter-linkedness as follows.

\begin{definition}[cf.~{\cite{GMS}}]\label{defuflim}
    Given a (non-principal) ultrafilter $D$ on $\omega$, define the linkedness property
    $\Lambda^{\lim}_D$, called \emph{$D$-$\lim$-linked}: $Q\in\Lambda^{\lim}_D(\Por)$ if there are a $\Por$-name $\dot D'$ of an ultrafilter on $\omega$ extending $D$ and a map $\lim^D\colon Q^\omega\to \Por$ such that, whenever $\bar p = \la p_n\colon n<\omega\ra\in Q^\omega$,
    \[{\lim}^D\, \bar p \Vdash \{n<\omega\colon p_n\in \dot G\} \in \dot D'.\]
    Define the linkedness property $\Lambda^{\lim}_{\rm uf}$, called \emph{uf-$\lim$-linked}, by $Q\in \Lambda^{\lim}_{\rm uf}(\Por)$ iff $Q \in \Lambda^{\lim}_{D}(\Por)$ for any ultrafilter $D$ on $\omega$.

    In addition, for an infinite cardinal $\mu$, we say that a poset $\Por$ is \emph{uniformly $\mu$-$D$-$\lim$-linked} if if is $\mu$-$\Lambda^{\lim}_D$-linked witnessed by some $\la Q_\alpha\colon \alpha<\mu\ra$, but the $\dot D'$ above can be the same for any $Q_\alpha$. And we say that $\Por$ is \emph{uniformly $\mu$-uf-$\lim$-linked} if there is some $\la Q_\alpha\colon \alpha<\mu\ra$ witnessing that $\Por$ is uniformly $\mu$-$D$-$\lim$-linked for any ultrafilter $D$ on $\omega$.
\end{definition}

\begin{example}\label{exm:singleton}
    Any singleton is uf-$\lim$-linked. As a consequence, any poset $\Por$ is uniformly $|\Por|$-uf-$\lim$-linked, witnessed by its singletons: for $p\in P$, let $Q_p:=\{p\}$, and $\lim^D$ on $Q_p$ is just the constant map with value $p$, when $D$ is an ultrafilter on $\omega$. Since $\lim^D \bar p\Vdash \{n<\omega\colon p_n\in \dot G\} = \omega$ for all $\bar p \in Q_p^\omega$, $\dot D'$ can be any $\Por$-name of an ultrafilter extending $D$.
\end{example}

\begin{theorem}[{\cite{GMS,BCM}}]\label{Euflim}
    $\Ebb$ is uniformly $\sigma$-uf-$\lim$-linked.
\end{theorem}
\begin{proof}
    We only indicate the components and the limit functions. For $s\in\omega^{<\omega}$ and $m\in\omega$, consider the set $E_{s,m}$ of conditions in $\Ebb$ of the form $(s,\varphi)$ with $\varphi\in \Scal(\omega,m)$. If $D$ is an ultrafilter on $\omega$ and $\bar p = \la p_n\colon n<\omega\ra\in E^\omega_{s,m}$, $p_n = (s,\varphi_n)$, define $\lim^D \bar p := (s,\varphi)$ where
    \[k\in \varphi(i) \text{ iff }\set{n<\omega}{k\in\ \varphi_n(i)}\in D.\]
    It is clear that $(s,\varphi)\in E_{s,m}$.

    The sequence $\la E_{s,m}\colon s\in\omega^{<\omega},\ m<\omega\ra$ witnesses that $\Ebb$ is uniformly $\sigma$-$D$-$\lim$-linked for any ultrafilter $D$ on $\omega$. This is proved by showing that, whenever $G$ is $\Por$-generic over $V$, the set
    \[D\cup\bigcup_{s,m}\lset{\{n<\omega\colon p_n\in G\}}{\bar p\in E_{s,m}^\omega\cap V,\ {\lim}^D\, \bar p\in G}\]
    has the finite intersection property.
\end{proof}

Based on~\cite{GMS,Uribethesis},
we present a framework to construct FS iterations that allow ultrafilter limits. The candidates for such iterations can be presented in a more general fashion. For an infinite cardinal $\theta$, denote
\[\theta^- = 
\left\{
\begin{array}{ll}
    \theta_0 &  \text{if $\theta = \theta_0^+$ for some cardinal $\theta_0$,}\\
    \theta & \text{if $\theta$ is not a successor cardinal.}
\end{array}
\right.\]

\begin{definition}\label{Gamma-it}
Let $\theta$ be an uncountable cardinal. 
A FS iteration $\la\Por_\alpha,\Por^-_\alpha,\Qnm_\xi\colon \alpha\leq\pi,\ \xi<\pi\ra$ is a \emph{$\theta$-$\Gamma$-iteration} if it satisfies:
\begin{enumerate}[label=\normalfont(\roman*)]
\item $\Por^-_\xi \subsetdot \Por_\xi$ for all $\xi<\pi$, and

\item $\Por^-_\xi$ forces that $\Qnm_\xi$ is $\mu_\xi$-$\Gamma$-linked witnessed by a sequence of $\Por^-_\xi$-names $\la \dot Q_{\xi,\zeta}\colon \zeta<\mu_\xi\ra$, where $\mu_\xi < \theta$ (known from the ground model).
\end{enumerate}
Associated with this iteration,
we define the following notions.
\begin{enumerate}[label = \normalfont (\arabic*)]
    \item A function $h\colon d_h\to\theta^-$ with $\pi\subseteq d_h$ is usually called a \emph{guardrail for the iteration}.\footnote{The assumption $\pi\subseteq d_h$ is for practicality, e.g.\ to have that it is still a guardrail of the iteration up to some $\xi<\pi$.}
    \item For $\alpha\leq\pi$ and $h$ as above, let
$\Por_\alpha^h$ be the set of conditions $p\in\Por_\alpha$ \emph{following $h$}, i.e.\ for $\xi\in\dom p$, $h(\xi)<\mu_\xi$, $p(\xi)$ is a $\Por^-_\xi$-name and $\Vdash_{\Por^-_\xi} p(\xi)\in \dot Q_{\xi,h(\xi)}$.
	
			A guardrail $h$ can have coordinates $\xi$ satisfying $h(\xi)\geq\mu_\xi$. Although they are not interesting, allowing them considerably simplifies writing the details of the construction of a $\theta$-$\Gamma$-iteration.

    \item $\Por^*_\alpha := \bigcup_{h\in {\theta^-}^\pi}\Por^h_\alpha$.

    \item Let $L$ be a linear order and $\seq{p_\ell}{\ell\in L}$ a sequence of conditions in $\Por_\pi$. We say that $\seq{p_\ell}{\ell\in L}$ is a \emph{uniform $\Delta$-system} if it satisfies the following:
    \begin{enumerate}[label = \normalfont (\roman*)]        
        \item All $\dom p_\ell$ ($\ell\in L$) have the same size $n$: $\dom p_\ell = \set{\alpha_{\ell,k}}{k<n}$ (increasing enumeration).

        \item There is some $v\subseteq n$ such that, for each $k\in v$, the sequence $\seq{\alpha_{\ell,k}}{\ell \in L}$ is constant with value $\alpha_{*,k}$.

        \item $\seq{\dom p_\ell}{\ell\in L}$ forms a $\Delta$-system with root $\set{\alpha_{*,k}}{k\in v}$.

        \item For $k\in n\menos v$, the sequence $\seq{\alpha_{\ell,k}}{\ell\in L}$ is increasing.

        \item There is some guardrail $h$ such that $\set{p_\ell}{\ell\in L} \subseteq \Por^h_\pi$.
    \end{enumerate}
\end{enumerate}
\end{definition}

By recursion on $\alpha\leq\pi$, we can show:

\begin{fact}\label{guard-dense}
    For any $\theta$-$\Gamma$-iteration as in \autoref{Gamma-it}, $\Por^*_\alpha$ is dense in $\Por_\alpha$.
\end{fact}

We focus on the case $\Gamma = \Lambda^{\lim}_{\rm uf}$. We plan to construct a $\theta$-$\Lambda^{\lim}_{\rm uf}$-iteration which is $\theta$-$\Fr$-Knaster (in our case, $\theta=\theta_3$).

\begin{lemma}\label{mainuf}
    For a $\theta$-$\Lambda_{\Fr}$-iteration as in \autoref{Gamma-it}: Let $H$ be a set of guardrails, $\theta'\geq\theta$ regular, and assume:
    \begin{enumerate}[label = \normalfont (\roman*)]
        \item\label{mainuf1} Any countable partial function from $\pi$ into $\theta^-$ can be extended by some $h\in H$.
        \item\label{mainuf2} If $h\in H$ and $\bar p =\la p_n\colon n<\omega\ra\subseteq \Por^h_\pi$ forms a uniform $\Delta$-system, then there is some $q\in\Por_\pi$ forcing that $\{n<\omega\colon p_n\in\dot G\}$ is infinite.
    \end{enumerate}
    Then $\Por_\pi$ is $\theta'$-$\Fr$-Knaster.
\end{lemma}
\begin{proof}
    Let $A\subseteq \Por_\pi$ have size $\theta'$. Since $\theta'$ is regular uncountable, we can find an uniform $\Delta$-system $B\subseteq A$ of size $\theta'$. Condition \ref{mainuf2} implies that $B$ is $\Fr$-linked.
\end{proof}

The $q$ in~\ref{mainuf2} is found as an \emph{ultrafilter limit} similar to \autoref{defuflim}, so this requires to construct ultrafilters along the iteration. For the successor step, the following lemma is useful.

\begin{lemma}[{\cite[Lem.~3.20]{BCM}}]\label{amalguf}
    Let $M\subseteq N$ be transitive models of $\thzfc$ and $\Qor\in M$ be a poset. Assume that $M\models$``$D^-$ is an ultrafilter on $\omega$", $M\models$``$\dot D^+$ is a $\Qor$-name of an ultrafilter on $\omega$ extending $D^-$", and $N\models$``$D$ is an ultrafilter on $\omega$ extending $D^-$". Then, in $N$, $\Qor$ forces that $D\cup \dot D^+$ has the finite intersection property, i.e.\ it can be extended to an ultrafilter (see \autoref{fig:diaguf}).
\end{lemma}

\begin{figure}[ht]
\centering
\begin{tikzpicture}
\small{
  \node (M) at (0,0) {$M$};
  \node (N) at (0,2) {$N$}; 
  \node (MG) at (2,0) {$M^\Qor$};
  \node (NG) at (2,2) {$N^\Qor$};
  \node (D0) at (-1,0) {$D^-\in$};
  \node (D) at (-1,2) {$D^{\phantom{-}} \in $};
  \node (D2) at (3,0) {$\ni D^+$};
\foreach \from/\to in {
M/N, M/MG, N/NG, MG/NG}
{
\path[->,] (\from) edge (\to);
}
}
\end{tikzpicture}
\caption{The situation in \autoref{amalguf}}\label{fig:diaguf}
\end{figure}

\begin{definition}\label{ituflim}
A $\theta$-$\Lambda^{\lim}_{\rm uf}$-iteration as in \autoref{Gamma-it} \emph{has ultrafilter limits for $H$} when:
\begin{enumerate}[label = \normalfont (\roman*)]
\item $H$ is a set of guardrails,
\item for $h\in H$, $\la \dot D^h_\xi\colon  \xi\leq\pi\ra$ is a sequence such that $\dot D^h_\xi$ is a $\Por_\xi$ name of a non-principal ultrafilter on $\omega$,
\item if $\xi<\eta\leq \pi$ then $\Vdash_{\Por_\eta}\dot D^h_\xi \subseteq \dot D^h_\eta$,
\item\label{ituflim4} $\Por_\xi$ forces that $\dot D^h_\xi\cap V^{\Por^-_\xi}\in V^{\Por^-_\xi}$,
\end{enumerate}
and whenever $h\in H$, $\la\xi_n\colon  n<\omega\ra\subseteq \pi$ and $\Vdash_{\Por^-_{\xi_n}}\dot q_n \in \dot Q_{\xi_n,h(\xi_n)}$:
\begin{enumerate}[resume*]
\item if $\la\xi_n\colon  n<\omega\ra$ is constant with value $\xi$ then
\[\Vdash_{\Por_\xi}{\lim}^{\dot D^h_\xi}_{\, n}\dot q_n \Vdash_{\Qnm_\xi} \{n<\omega\colon  \dot q_n \in \dot G(\xi)\} \in \dot D^h_{\xi+1},\]
\item and if $\la\xi_n\colon  n<\omega\ra$ is increasing, then
\[\Vdash_{\Por_\pi} \{n<\omega\colon  \dot q_n \in \dot G(\xi_n)\}\in \dot D^h_\pi\]
\end{enumerate}
\end{definition}

\begin{lemma}\label{main-uflim}
    Any iteration as in \autoref{ituflim} satisfies~\ref{mainuf2} of \autoref{mainuf} for $H$.
\end{lemma}
\begin{proof}
    Let $\seq{p_n}{n<\omega}$ be an uniform $\Delta$-system in $\Por^h_\pi$. Let $\Delta$ be the root of the $\Delta$-system and define $q\in \Por_\pi$ with $\dom q := \Delta$ such that $q(\xi)$ is a $\Por^-_\xi$-name of $\lim^{\dot D^h_\xi}_n p_n(\xi)$ for $\xi\in \Delta$. Then $q$ forces that $\{n<\omega\colon p_n\in \dot G_\pi\} \in \dot D^h_\pi$.
\end{proof}

To obtain~\ref{mainuf1} of \autoref{ituflim} we could basically use $H = {\theta^-}^\pi$. However, there are steps $\xi<\pi$ of the iteration where we want $\Por^-_\xi$ to be quite small, so to guarantee~\ref{ituflim4} of \autoref{ituflim} we need that $H$ is also small. This is guaranteed by the following result.

\begin{theorem}[{\cite{EK,RinotEK}}]\label{k20}
    Let $\nu$, $\kappa$ be infinite cardinals and $L$ be a set such that  $\nu \leq \kappa \leq |L| \leq 2^{\kappa}.$ Then  there exists an $H \subseteq {}^{L} \kappa$ such that $| H | \leq \kappa^{< \nu},$ and any partial function from $L$ into $\kappa$ with domain of size ${<}\nu$ can be extended by a function in $H.$ 
\end{theorem}

The following two theorems indicate how to construct iterations as in \autoref{ituflim}.

\begin{theorem}\label{limsucc}
    Let $\Por_{\pi+1}$ be a  $\theta$-$\Lambda^{\lim}_{\rm uf}$-iteration of length $\pi+1$ and $H$ be a set of guardrails such that, up to $\pi$, it has ultrafilter limits for $H$.

    Assume that $\Por^-_\pi\subsetdot \Por_\pi$ and $\Por_\pi$ forces $\dot D^h_\pi \cap V^{\Por^-_\pi}\in V^{\Por^-_\pi}$. Then, for each $h\in H$, we can find a $\Por_{\pi+1}$-name $\dot D^h_{\pi+1}$ of an ultrafilter extending $\dot D^h_\pi$, which together make $\Por_{\pi+1}$ have ultrafilter limits for $H$.
\end{theorem}
\begin{proof}
    Direct application of \autoref{amalguf} to $M=V[G\cap\Por^-_\pi]$ and $N=V[G]$, where $G$ is $\Por_\pi$-generic over $V$.
\end{proof}

\begin{theorem}\label{limlim}
    Assume that $\pi$ is a limit ordinal and $\Por_\pi$ is a $\theta$-$\Lambda^{\lim}_{\rm uf}$-iteration of length $\pi$. Further assume that $h$ is a guardrail and $\la \dot D^h_\xi\colon \xi<\pi\ra$ is a sequence witnessing that, for any $\xi<\pi$, $\Por_\xi$ is an iteration with uf-limits for $h$.

    If, for any $\xi<\pi$, $\Por^-_\xi$ forces that $\dot Q_{\xi,h(\xi)}$ is centered, then we can find a $\dot D^h_\pi$ that makes $\Por_\pi$ have uf-limits for $h$. 
\end{theorem}
\begin{proof}
	    Let $E$ be the collection of sequences $\tau=\seq{(\xi_n,\dot q_n)}{n<\omega}$ such that $\seq{\xi_n}{n<\omega}$ is an increasing sequence of ordinals with limit $\pi$ and $\Vdash_{\xi_n}\dot q_n\in \dot Q_{\xi_n,h(\xi_n)}$ for all $n<\omega$. For such a $\tau$, let $\dot d_\tau$ be a $\Por_\pi$-name of the set $\set{ n<\omega }{ \dot{q}_{n} \in \dot{G}(\xi_n) }$. It is enough to show that $\Por_\pi$ forces that $\bigcup_{\xi<\pi}\dot D^h_\xi\cup\set{\dot d_\tau}{\tau\in E}$ has the finite intersection property. 
    
    Fix $p\in\Por_\pi$, a $\Por_\pi$-name $\dot d$ of a member of $\bigcup_{\xi<\pi}\dot D^h_\xi$, $i^*<\omega$ and, for each $i<i^*$, $\tau^i = \seq{(\xi^i_n,\dot q^i_n)}{n<\omega}\in E$. Denote $\dot d_i:=\dot d_{\tau^i}$. Without loss of generality, by strengthening $p$ if necessary, we may assume that, for some $\alpha<\pi$, $\dot d$ is a $\Por_\alpha$-name of a member of $\dot D^h_\alpha$. We can also increase $\alpha$ and say that $p\in\Por_\alpha$. Since $\dot D^h_\alpha$ is forced non--principal, $p$ forces in $\Por_\alpha$ that $\dot d\cap\bigcap_{i<i^*}\set{n<\omega}{\xi^i_n\geq\alpha}\in \dot D^h_\alpha$, so there are $p'\leq p$ in $\Por_\alpha$ and $k<\omega$ such that $\xi^i_k\geq\alpha$ for all $i<i^*$ and $p'\Vdash_\alpha k\in\dot d$. To conclude the proof, it is enough to find some $q\leq p'$ in $\Por_\pi$ such that $q\Vdash \dot q^i_k\in\dot G(\xi^i_k)$ for all $i<i^*$.

    Let $L:=\set{\xi_j}{j<m}$ be the increasing enumeration of $\set{\xi^i_k}{i<i^*}$. 
    For $j<m$, let $I_j:=\set{i<i^*}{\xi^i_k = \xi_j}$. Since $\Por_{\xi_j}$ forces that $\dot Q_{\xi_j,h(\xi_j)}$ is centered, there is some condition in $\Qnm_{\xi_j}$ stronger than $\dot q^i_k$ for all $i\in I_j$. This allows to define $q\in\Por_\pi$ with $\dom q = \dom p'\cup L$ such that $q(\xi):=p'(\xi)$ for $\xi\in\dom p'$ and, for $j<m$ and $i\in I_j$, $\Vdash_{\xi_j}q(\xi_j)\leq \dot q^i_k$. It is clear that $q$ is as required.
\end{proof}

We are now ready to present the main forcing construction of this section.

\begin{theorem}[cf.~{\cite{GMS,GKS}}]\label{main2}
    Let $\aleph_1\leq \theta_1\leq \theta_2\leq \theta_3 \leq \theta_4$ be regular cardinals, and assume $\lambda$ is a cardinal such that $\lambda = \lambda^{\aleph_0}$ and $\cf([\lambda]^{<\theta_i}) = \lambda$ for $i=1,\ldots,4$. Further assume that one of the following holds:
    \begin{enumerate}[label=\normalfont(\roman*)]
        \item\label{main2i} $\theta_3=\theta_4$.
        \item\label{main2ii} 
        $\theta_3^-<\theta_4$, 
        $\theta^{\aleph_0}<\theta_4$ for every cardinal $\theta<\theta_4$, and $\lambda \leq 2^\kappa$ for some cardinal $\kappa<\theta_4$.
    \end{enumerate}    
    Then, we can construct a FS iteration of length (and size) $\lambda$ of ccc posets forcing $\add(\Nwf) = \theta_1$, $\cov(\Nwf) = \theta_2$, $\bfrak = \theta_3$, $\non(\Ewf) =\non(\Mwf) = \theta_4$ and $\cov(\Mwf) = \cfrak = \lambda$ (see \autoref{fig:main1}). 
\end{theorem}
\begin{proof}
    In case~\ref{main2i} the result follows directly from \autoref{main1}, so we focus on the assumptions of case~\ref{main2ii}, in which we can further assume that $\kappa\geq \theta_3^-$. We proceed exactly as in the proof of \autoref{main1} to construct a FS iteration of length $\pi:=\lambda+\lambda$, using Cohen forcing at the first $\lambda$ stages, but we modify the construction for $\alpha\in K_4$, the steps where we increase $\non(\Mcal)$ (and even $\non(\Ecal)$) using $\Ebb$, to obtain a $\theta_3$-$\Lambda^{\lim}_{\rm uf}$-iteration with ultrafilter limits on some $H$ of size ${<}\theta_4$. We aim to apply \autoref{mainuf} and~\ref{main-uflim} to conclude that the iteration is $\theta_3$-$\Fr$-Knaster, hence ensuring, by \autoref{FrKpres}, that the first $\lambda$-many Cohen reals added in the iteration form a $\theta_3$-$\omega^\omega$-unbounded family in the final extension, so the remaining $\bfrak\leq\theta_3$ will be forced.

    Using that $\theta^-_3\leq \kappa<\theta_4$ and $\lambda \leq 2^\kappa$, by \autoref{k20} we can find $H_0\subseteq \kappa^{\pi}$ of size ${\leq}\kappa^{\aleph_0}<\theta_4$ (by \ref{main2ii}) such that any countable partial function from $\pi$ into $\kappa$ can be extended by a function in $H_0$. 
    For any $g\in \kappa^\pi$ define $g'\in {\theta_3^-}^\pi$ by $g'(\xi):= g(\xi)$ if $g(\xi)<\theta_3^-$, and $g'(\xi):=0$ otherwise. 
    Then $H:=\set{g'}{g\in H_0}\subseteq {\theta_3^-}^\pi$ has size ${<}\theta_4$ and any countable partial function from $\pi$ into $\theta^-_3$ can be extended by a function in $H$. 
    This guarantees requirement~\ref{mainuf1} of \autoref{mainuf}.
    

    To construct the iteration, proceed by recursion, starting with an ultrafilter $D^h_0$ on $\omega$ for $h\in H$. In the successor step $\xi\to \xi+1$, we do some work in the case $\xi=\lambda+\alpha$ with $\alpha\in K_4$
    because in other cases we proceed as in \autoref{main1} and just pick $\mu_\xi<\theta_3$ such that $\Vdash_{\Por_\xi}\Qnm_\xi = \{\dot q^\xi_\zeta\colon \zeta<\mu_\xi\}$, so we let $\Por^-_\xi:= \Por_\xi$ and $\dot Q_{\xi,\zeta}$ be a $\Por_\xi$-name of $\{\dot q^\xi_\zeta\}$, so any $\Por_{\xi+1}$-name $\dot D^h_{\xi+1}$ of an ultrafilter extending $\dot D^h_\xi$ is suitable.
    
    Using the book-keeping for $K_4$, in stage $\xi=\lambda+\alpha$ we have picked some $\Por_{\xi}$-name $\dot F_\alpha$ of a subset of $\omega^\omega$ of size ${<}\theta_4$, and aim to add an eventually different real over $\dot F_\alpha$ in the following step by using a restriction of $\Ebb$. Since $\Por_{\xi}$ has the ccc, we can find some $\nu_\alpha<\theta_4$ such that $\dot F_\alpha$ is represented by $\{\dot x_{\alpha,i}\colon i<\nu_\alpha\}$. Using the assumption~\ref{main2ii}, for large enough $\chi$ we can find $M\prec H_\chi$ of size ${<}\theta_4$, closed under countable sequences, such that $\Por_{\xi}$ and each $\dot x_{\alpha,i}$ ($i<\nu_\alpha$) and $\dot D^h_{\xi}$ ($h\in H$) are in $M$. Consider $\Por^-_{\xi}:= \Por_{\xi}\cap M$, which is a complete suborder of $\Por_{\xi}$ because the latter has the ccc and $M$ is closed under countable sequences. Then, we force with $\Qnm_\xi:=\Ebb^{V^{\Por^-_{\xi}}}$ to advance to the next stage. Note that this is a $\Por^-_\xi$-name (for $\Ebb$). Enumerate $\omega^{<\omega}\times \omega = \set{(s_k,m_k)}{k<\omega}$ and let $\mu_\xi:=\omega$ and $\dot Q_{\xi,k}$ be a $\Por^-_\xi$-name of $E_{s_k,m_k}$ for $k<\omega$.

    By the construction of $\Por^-_{\xi}$, for any $h\in H$ we can find a $\Por^-_{\xi}$-name $\dot D^{h,-}_{\xi}$ of $\dot D^h_\xi \cap V^{\Por^-_\xi}$ (which exists because $M$ is countably closed). 
    Then, \autoref{limsucc} applies.


    Limit steps are guaranteed by \autoref{limlim}, since all the components $\dot Q_{\xi,\zeta}$ are centered.
\end{proof}

Notice that the condition ``$2^\kappa\geq\lambda$ for some $\kappa<\theta_4$" in~\ref{main2ii} of \autoref{main2} is incompatible with GCH in general, which would be a problem for applying the methods to force Cicho\'n's maximum in the following sections. This requirement can be weakened to a condition compatible with GCH, as below. 
Recall the poset $\Fn_{<\kappa}(I,B)$ of partial functions from $I$ into $B$ of size ${<}\kappa$, ordered by $\supseteq$.

\begin{theorem}[{\cite{GKS}}]\label{main2'}
    The conclusion of \autoref{main2} is valid for the case:
    \begin{enumerate}[label=\normalfont (\roman*')]
    \setcounter{enumi}{1}
        \item $\theta_3^-<\theta_4$, $\theta^{\aleph_0}<\theta_4$ for every cardinal $\theta<\theta_4$, and $\kappa^{<\kappa} = \kappa$ for some uncountable cardinal $\kappa<\theta_4$.
    \end{enumerate}
\end{theorem}
\begin{proof}
    Let $V'$ be a generic extension of $\Fn_{<\kappa}(\lambda,\kappa)$. Notice that $2^\kappa\geq\lambda$ holds in $V'$, and that the assumptions of this theorem are preserved. Then, the construction of \autoref{main2} can be executed in $V'$. However, it is possible to construct the iteration $\Por=\Por_{\lambda+\lambda}$ in $V$ even though the set of guardrails and the sequences of (names of) ultrafilters live in $V'$. The book-keeping can be executed in $V$, and the strong unbounded families live in $V^\Por$, so $V^\Por$ forces the desired conclusion.
\end{proof}

\begin{remark}
    \autoref{main2} has been improved in~\cite{BCM} without need of the requirements~\ref{main2i} and~\ref{main2ii}, even obtaining a model where $\theta_4\leq\cov(\Mwf)=\theta_5 \leq \dfrak=\non(\Nwf)=\cfrak$ for a given regular $\theta_5$. The construction comes from a two-dimensional iteration with ultrafilters.
\end{remark}

A method preceding ultrafilter-limits, which is more powerful, is the method of iterations with finitely additive measures (fams). Shelah~\cite{ShCov} introduced this method for random forcing to prove the consistency of ZFC with $\cf(\cov(\Nwf)) = \omega$, and it was formalized in~\cite{KST} with applications in Cicho\'n's diagram. Recently, Andr\'es Uribe-Zapata~\cite{Uribethesis} formalized the general framework of iterations with fam-limits, which was further refined with Cardona and the author in~\cite{CMU}.

\begin{definition}
    Let $\Por$ be a poset.
    \begin{enumerate}[label = \normalfont (\arabic*)]
        \item\label{faml1} Let $\Xi\colon \pts(\omega) \to [0,1]$ be a fam (with $\Xi(\omega)=1$ and $\Xi(\{n\})=0$ for all $n<\omega$), $I = \la I_n\colon n<\omega\ra$ be a partition of $\omega$ into finite sets, and $\varp>0$.

        A set $Q\subseteq\Por$ is \emph{$(\Xi,I,\varp)$-linked} if there is a function $\lim\colon Q^\omega\to \Por$ and a $\Por$-name $\dot \Xi'$ of a fam on $\pts(\omega)$ extending $\Xi$ such that, for any $\bar p = \seq{ p_\ell}{\ell<\omega} \in Q^\omega$,
        \[\lim \bar p \Vdash \int_\omega \frac{|\set{\ell \in I_k}{p_\ell \in \dot G}|}{|I_k|}d\dot \Xi' \geq 1-\varp.\]

        \item\label{faml2} The poset $\Por$ is \emph{$\mu$-FAM-linked}, witnessed by $\la Q_{\alpha,\varp}\colon \alpha<\mu,\ \varp\in(0,1)\cap \Q\ra$, if:
        \begin{enumerate}[label = \rm (\roman*)]
            \item Each $Q_{\alpha,\varp}$ is $(\Xi,I,\varp)$-linked for any $\Xi$ and $I$.
            \item For $\varp\in(0,1)\cap \Q$, $\bigcup_{\alpha<\omega} Q_{\alpha,\varp}$ is dense in $\Por$.
        \end{enumerate}

        \item\label{faml3} The poset $\Por$ is \emph{uniformly $\mu$-FAM-linked} if there is some $\la Q_{\alpha,\varp}\colon \alpha<\mu,\ \varp\in(0,1)\cap \Q\ra$ as above, such that in~\ref{faml1} the name $\dot \Xi'$ only depends on $\Xi$ (and not on any $Q_{\alpha,\varp}$).
    \end{enumerate}
\end{definition}

\begin{example}
    \ 
    \begin{enumerate}[label = \normalfont (\arabic*)]
        \item Any singleton is $(\Xi,I,\varp)$-linked. Hence, any poset $\Por$ is uniformly $|\Por|$-FAM-linked. In particular, Cohen forcing is uniformly $\sigma$-FAM-linked.

        \item Shelah~\cite{ShCov} proved, implicitly, that random forcing is uniformly $\sigma$-FAM-linked.        
        More generally, any measure algebra with Maharam type $\mu$ is uniformly $\mu$-FAM-linked~\cite{MUrandom}.

        \item The creature ccc forcing from~\cite{HSh} adding eventually different reals is (uniformly) $\sigma$-FAM-linked. This is proved in~\cite{KST} and also in~\cite{Egorro} in a more general setting.
    \end{enumerate}
\end{example}

Iterations with fam limits can be constructed as in~\autoref{ituflim}. For all the details, see~\cite{Uribethesis,CMU}. These can be used to prove:

\begin{theorem}[{\cite{KST}}]\label{altKST}
Let $\aleph_1\leq \theta_1\leq \theta_2\leq \theta_3 \leq \theta_4$ be regular cardinals, and assume $\lambda$ is a cardinal such that $\lambda = \lambda^{\aleph_0}$ and $\cf([\lambda]^{<\theta_i}) = \lambda$ for $i=1,\ldots,4$. Further assume that one of the following holds:
    \begin{enumerate}[label=\normalfont(\roman*)]
        \item $\theta_2=\theta_3$.
        \item\label{KSTii} 
        $\theta_2^-<\theta_3$, 
        $\theta^{\aleph_0}<\theta_i$ for every cardinal $\theta<\theta_i$ and $i\in\{3,4\}$, and $\kappa = \kappa^{<\kappa}$ for some uncountable cardinal $\kappa<\theta_3$.
    \end{enumerate}    
    Then, we can construct a FS iteration of length (and size) $\lambda$ of ccc posets forcing $\add(\Nwf) = \theta_1$, $\bfrak = \theta_2$, $\cov(\Nwf) = \theta_3$, $\non(\Ewf) =\non(\Mwf) = \theta_4$ and $\cov(\Mwf) = \cfrak = \lambda$ (see \autoref{fig:altKST}).
\end{theorem}

\begin{figure}[ht]
\centering
\begin{tikzpicture}
\small{
 \node (aleph1) at (-2,2.5) {$\aleph_1$};
 \node (addn) at (0,2.5){$\add(\Nwf)$};
 \node (covn) at (2,5){$\cov(\Nwf)$};
 \node (cove) at (7,3.75) {$\cov(\Ecal)$};
 \node (b) at (2,0) {$\bfrak$};
 \node (nonm) at (4,2.5) {$\non(\Mcal)$} ;
 \node (none) at (3,1.25) {$\non(\Ecal)$};
 \node (d) at (8,5) {$\dfrak$};
 \node (covm) at (6,2.5) {$\cov(\Mcal)$} ;
 \node (nonn) at (8,0) {$\non(\Ncal)$} ;
 \node (cfn) at (10,2.5) {$\cof(\Ncal)$} ;
  \node (c) at (12,2.5) {$\cfrak$};

\foreach \from/\to in {
aleph1/addn, addn/covn, addn/b, covn/nonm, b/none, none/nonm, covm/nonn, covm/cove, cove/d, nonn/cfn, d/cfn, cfn/c}
{
\path[-,draw=white,line width=3pt] (\from) edge (\to);
\path[->,] (\from) edge (\to);
}
\path[->,draw=red,line width=1pt] (nonm) edge (covm);

\draw[color=sug,line width=1]
(-1,0)--(-1,5)
(-1,5)--(3,0)
(0,0)--(4,5)
(5,5)--(5,0);

\draw[circle, fill=suy,color=suy] 
(0,1) circle (0.3)
(2,3.75) circle (0.3)
(1,0) circle (0.3)
(4,3.75) circle (0.3)
(8,2.5) circle (0.3);

\node at (0,1) {$\theta_1$};
\node at (2,3.75) {$\theta_3$};
\node at (1,0) {$\theta_2$};
\node at (4,3.75) {$\theta_4$};
\node at (8,2.5) {$\lambda$};

}
\end{tikzpicture}
\caption{The constellation of Cicho\'n's diagram forced in \autoref{altKST}.}\label{fig:altKST}
\end{figure}
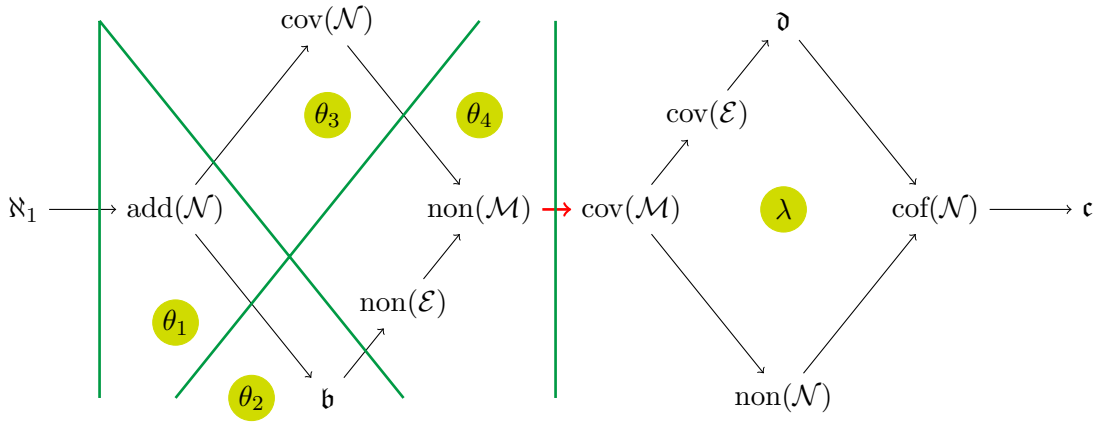

\begin{remark}
    Several extensions of \autoref{main2'} and \autoref{altKST} (and Cicho\'n's Maximum) have been obtained by separating more cardinals simultaneously, like $\mfrak$, $\pfrak$, $\hfrak$, $\sfrak$ and $\rfrak$~\cite{GKMSnoreal,GKMSsplit,GKMScol}. Very recently, Yamazoe~\cite{Yamazoe} developed iterations with \emph{closed ultrafilter limits} to separate $\efrak$ (and its dual) in additon. He defines the linkedness notion $\Lambda^{\lim}_{{\rm c}D}$ and $\Lambda^{\lim}_{\rm cuf}$ as a variation of \autoref{defuflim} where $\lim^D\colon Q^\omega\to Q$ (i.e., $Q$ is closed under the ultrafilter limit), modifies the iteration theory with ultrafilters for \emph{closed-ultrafilter limits} (i.e.\ \autoref{ituflim} for $\Lambda^{\lim}_{\rm cuf}$), and proves that closed-ultrafilter limits help control $\efrak$.

    Preceding Yamazoe's work, the author with Goldstern, Kellner and Shelah~\cite{GKMSe} claimed the same result, using that fam-limits help to control $\efrak$ (and its dual). For their forcing iteration costruction, they claimed that ultrafilter-limits at limit steps (of countable cofinality) of the iteration can be obtained by weaking the ``centered" requirement for $\dot Q_{\xi,h(\xi)}$ in \autoref{limlim}. However, this way to proceed is still unclear, which represents a hole in their argument.

    As a follow up, the author with Cardona~\cite{CardonaRIMS,CMU} proved that fam-limits help to control $\non(\Ewf)$ and $\cov(\Ewf)$.
\end{remark}

\section{Boolean ultrapowers}

Goldstern, Kellner and Shelah~\cite{GKS} proved that applying Boolean ultrapowers to the model from \autoref{main2'} yields a ccc poset that forces Cicho\'n's maximum.

The effect of Boolean ultrapowers to Tukey connections between relational systems is very relevant to this work. Indeed, we see how the main theorems proved in the previous sections can be reformulated in terms of Tukey connections. First notice:

\begin{lemma}\label{5.1}
    Let $K$ be a relational system, $\kappa$ be an uncountable regular cardinal and let $\Por$ be a $\kappa$-cc poset.
    \begin{enumerate}[label = \normalfont (\alph*)]
        \item If $\kappa\leq \bfrak(K)^V$ then $\Por$ forces $\bfrak(K) = \bfrak(K)^V$.

        \item If $\kappa\leq \dfrak(K)^V$ then $\Por$ forces $\dfrak(K) = \dfrak(K)^V$.

        \item\label{5.1c} If $\kappa\leq\theta\leq\lambda$ are cardinals, then $\Por$ forces $[\lambda]^{<\theta} \eqT [\lambda]^{<\theta}\cap V$ and $\Cbf_{[\lambda]^{<\theta}} \eqT \Cbf_{[\lambda]^{<\theta}\cap V}$.
    \end{enumerate}
\end{lemma}

Due to~\ref{5.1c}, since we mainly work with ccc forcing extensions, we can use $[\lambda]^{<\theta}$ with $\theta$ uncountable without specifying the model we are looking at.

\begin{lemma}[{\cite[Fact~3.8]{CM23}}]
    Let $\theta\leq\lambda$ be infinite cardinals. Then $\Cbf_{[\lambda]^{<\theta}}\eqT [\lambda]^{<\theta}$ iff $\theta$ is regular and $\cof([\lambda]^{<\theta}) = \lambda$.
\end{lemma}

\begin{lemma}\label{brel}
    For any relational system $R=\la X,Y,\sqsubset\ra$, if $\bfrak(R)$ exists, then $R \leqT \Cbf_{[X]^{<\bfrak(R)}}$.
\end{lemma}
\begin{proof}
    Use the maps $x\mapsto x$ and $B\in [X]^{<\bfrak(R)} \mapsto y_B$ where $y_B$ is an upper bound of $B$.
\end{proof}

Also recall from \autoref{strongTukey} that $\Cbf_{[I]^{<\theta}}\leqT R$ iff there is some $\theta$-$R$-unbounded family indexed by $I$. As a consequence, we can reformulate \autoref{main2'}:

\begin{theorem}\label{main1Tukey}
    Let $\aleph_1\leq \theta_1\leq \theta_2\leq \theta_3 \leq \theta_4$ be regular cardinals, and assume $\lambda$ is a cardinal such that $\lambda = \lambda^{\aleph_0}$ and $\cf([\lambda]^{<\theta_i}) = \lambda$ for $i=1,\ldots,4$. Further assume that one of the following holds:
    \begin{enumerate}[label=\normalfont(\roman*)]
        \item $\theta_3=\theta_4$.
        \item\label{mm2} 
        $\theta_3^-<\theta_4$, 
        $\theta^{\aleph_0}<\theta_4$ for every cardinal $\theta<\theta_4$, and $\kappa = \kappa^{<\kappa}$ for some uncountable cardinal $\kappa<\theta_4$.
    \end{enumerate}    
    Then, we can construct a FS iteration of length (and size) $\lambda$ of ccc posets forcing $\cfrak=\lambda$, $\Nwf \eqT \Cbf_{[\lambda]^{<\theta_1}}$, $\Cbf_\Nwf^\perp\eqT \Cbf_{[\lambda]^{<\theta_2}}$, $\omega^\omega \eqT \Cbf_{[\lambda]^{<\theta_3}}$ and $\Cbf_\Mwf \eqT \Cbf_\Ewf \eqT \Cbf_{[\lambda]^{<\theta_4}}$ and, for $i=1,\ldots,4$, $\Cbf_{[\lambda]^{<\theta_i}} \eqT \Cbf_{[\lambda]^{<\theta_i}\cap V} \eqT [\lambda]^{<\theta_i}\cap V \eqT [\lambda]^{<\theta_i}$.
\end{theorem}

Tukey connections have certain types of witnesses as follows.

\begin{lemma}
    Let $R = \la X,Y,\sqsubset\ra$ and $K = \la A,B,\ltrg\ra$ be relational systems. Then:
    \begin{enumerate}[label = \normalfont (\alph*)]
    \item $R\leqT K$ iff there is some $\la y_b\colon b\in B\ra \subseteq Y$ satisfying:
        \[\forall x\in X\, \exists\, a\in A\, \forall\, b\in B\colon a \ltrg b \imp x \sqsubset y_b.\]
    
        \item $K\leqT R$ iff there is some $\la x_a\colon a\in A\ra \subseteq X$ satisfying:
        \[\forall\, y\in Y\, \exists\, b\in B\, \forall\, a\in A\colon a\not\ltrg b \Rightarrow x_a \not\sqsubset y.\]
    \end{enumerate}
\end{lemma}

In the context of the previous lemma, the role of $K$ is a constant relational system which is not interpreted (in principle) in forcing generic extensions, i.e.\ we keep its meaning with respect to $V$. On the other hand, $R$ is interpreted according to its definition.

Elementary embeddings modify Tukey connections as follows. Recall that a model $M$ is \emph{${<}\kappa$-closed} if it is closed under sequences of length ${<}\kappa$.

\begin{lemma}\label{BUP-Tukey}
    Let $\theta$ be an uncountable regular cardinal and $\Por$ a $\theta$-cc poset. 
    Let $M\subseteq V$ be a ${<}\theta^+$-closed transitive model, $j\colon V\to M$ an elementary embedding with critical point ${>}\theta$, $R=\la X,Y,\sqsubset\ra$ a relational system of the reals (\autoref{realsys}) and let $K = \la A,B,\ltrg\ra$ be a relational system. Then:
    \begin{enumerate}[label = \normalfont (\alph*)]
        \item\label{jPccc} $j(\Por)$ has the $\theta$-cc and $\Por\subsetdot j(\Por)$.
        \item\label{BUPa} If $\Vdash_\Por R\leqT K$ is witnessed by $\bar y = \la \dot y_b \colon b\in B \ra$, then $\Vdash_{j(\Por),V} R\leqT j(K)$ is witnessed by $j(\bar y) = \la \dot y^*_{b'} \colon b'\in j(B) \ra$.
        
        \item\label{BUPb} If $\Vdash_\Por K\leqT R$ is witnessed by $\bar x = \la\dot x_a\colon a\in A\ra$, then $\Vdash_{j(\Por),V} j(K)\leqT R$ is witnessed by $j(\bar x) = \la \dot x^*_{a'}\colon a'\in j(A)\ra$.
    \end{enumerate}
\end{lemma}
\begin{proof}
    \ref{jPccc}: If $C\subseteq j(\Por)$ has size $\theta$, then $C\in M$ because $M$ is ${<}\theta^+$-closed. Then, $M\models$``$C\subseteq j(\Por)$". On the other hand, since $j\colon V\to M$ is an elementary embedding and $\theta$ is smaller than its critical point, $M\models$``$j(\Por)$ has the $\theta$-cc", so $C$ is not an antichain in $j(\Por)$ (both in $M$ and in $V$). 
    A similar argument shows that $\Por\subsetdot j(\Por)$.

    \ref{BUPa}: Let $\dot x$ be a (nice) $j(\Por)$-name of a real in $X$. By~\ref{jPccc}, $\dot x\in M$. On the other hand, by elementarity, in $M$ we have that $R\leqT J(K)$ is witnessed by $j(\bar y)$, so $\exists\, a'\in j(A)\ \forall\, b'\in j(B)\colon a' \mathrel{j(\ltrg)} b' \imp \dot x \sqsubset \dot y^*_{b'}$. The same holds in $V$ by absoluteness, which concludes the proof.
    
    \ref{BUPb} follows by~\ref{BUPa} applied to $R^\perp$ and $K^\perp$. 
\end{proof}

Given the previous result, we plan to apply elementary embeddings to the poset of \autoref{main1Tukey} to obtain a forcing that modifies Tukey equivalences in such a way that Cicho\'n's maximum is forced.

\newcommand{\BUP}{\mathrm{BUP}}

The elementary embeddings come from Boolean ultrapowers. Let $\kappa$ be a strongly compact cardinal and $\Bor$ a ${<}\kappa$-distributive $\kappa^+$-cc complete Boolean algebra containing an antichain of size $\kappa$. Let $V_-^\Bor$ be the class of nice $\Bor$-names of members of the ground model $V$, i.e.\ they are coded by a function $h\colon A\to V$ with domain a maximal antichain $A$ in $\Bor$. Given a $\kappa$-complete ultrafilter $D$ on $\Bor$, define the following relations on $V^\Bor_-$:
\begin{align*}
  \sigma =_D \tau & \text{ iff } \|\sigma = \tau \|\in D,\\
  \sigma \in_D \tau & \text{ iff } \|\sigma \in \tau \|\in D.
\end{align*}
It is easy to show that $=_D$ is an equivalence relation on $V^\Bor_-$, which defines the \emph{Boolean ultrapower} $\BUP(\Bor,D):= V^\Bor_-/{=_D}$. The relation obtained from $\in_D$ makes $\BUP(\Bor,D)$ a well-founded model of ZFC, and $\iota\colon V\to\BUP(\Bor,D)$ is an elementary embedding, where $\iota(x)$ is the equivalence class of $\check{x} = \set{(\check{y},1_\Bor)}{y\in x}$. 
Let $c\colon \BUP(\Bor,D)\to M$ be the Mostowski collapse of $\BUP(\Bor,D)$. Then, $j:=c\circ \iota\colon V\to M$ is an elementary embedding. The class $M$ is a ${<}\kappa$-closed (even ${<}\kappa^+$-closed) transitive model of ZFC. Details can be found in~\cite{KTT}.

\begin{lemma}[cf.~\cite{KTT}]\label{tonti}
    Let $\kappa$ be a strongly compact cardinal and $\lambda\geq \kappa$ a cardinal such that $\lambda^\kappa = \lambda$. Then, there is a $\kappa$-complete ultrafilter on the completion of $\Fn_{<\kappa}(\lambda,\kappa)$ such that its corresponding elementary embedding $j\colon V\to M$ satisfies:
    \begin{enumerate}[label = \normalfont (\alph*)]
        \item $M$ is ${<}\kappa$-closed (and even ${<}\kappa^+$-closed).
        \item The critical point of $j$ is $\kappa$.
        \item If $|a|<\kappa$ then $j(a) = j[a]$.
        \item If $\lambda'\geq\kappa$ then $\max\{\lambda,\lambda'\} \leq |j(\lambda')| \leq \max\{\lambda,\lambda'\}^\kappa$.
        \item If $S$ is a directed preorder and $\bfrak(S)>\kappa$, then $j[S]$ is cofinal in $j(S)$. In particular, $S\eqT j(S)$.
    \end{enumerate}
\end{lemma}

\begin{corollary}\label{BUPtool}
    Under the framework of \autoref{tonti}: if $\theta\leq\lambda$ are infinite cardinals, then:
    \begin{enumerate}[label = \normalfont (\alph*)]
        \item\label{BUPtool1} If $\theta<\kappa$ then $j([\lambda]^{<\theta}) = [j(\lambda)]^{<\theta}$.
        \item\label{BUPtool2} If $\cf(\theta)>\kappa$ then $j([\lambda]^{<\theta}) \eqT [\lambda]^{<\theta}$.
    \end{enumerate}
\end{corollary}

We have developed enough material to prove Cicho\'n's maximum, first by using large cardinals.

\begin{theorem}\label{mainBUP1}
    Assume that $\aleph_1 <\kappa_0 < \theta_1 < \kappa_1 <\theta_2 <\kappa_2 < \theta_3 < \kappa_3 < \theta_4 \leq \lambda_4 \leq \lambda_3 \leq \lambda_2 \leq \lambda_1 \leq \lambda_0$ such that each $\theta_i$ is regular, $\cf([\lambda_4]^{{<}\theta_i}) = \lambda_4 = \lambda_4^{\aleph_0}$, condition~\ref{mm2} of \autoref{main1Tukey} holds, $\kappa_\ell$ is strongly compact and $\lambda_\ell^{\kappa_\ell} = \lambda_\ell$ for $\ell\leq 3$. Then there is a ccc poset of size $\lambda_0$ forcing $\cfrak=\lambda_0$, $\Nwf \eqT \Cbf_{[\lambda_1]^{<\theta_1}}$, $\Cbf^\perp_\Nwf \eqT \Cbf_{[\lambda_2]^{<\theta_2}}$, $\omega^\omega \eqT \Cbf_{[\lambda_3]^{<\theta_3}}$ and $\Cbf_\Mwf \eqT \Cbf_{\Ewf} \eqT \Cbf_{[\lambda_4]^{<\theta_4}}$. In particular (see \autoref{fig:CMax1}),
    \begin{align*}
        \aleph_1 & < \add(\Nwf) = \theta_1 < \cov(\Nwf) = \theta_2 < \bfrak = \theta_3 <\non(\Ewf) = \non(\Mwf) = \theta_4\\
        & \leq \cov(\Mwf) = \cov(\Ewf) = \lambda_4 \leq \dfrak=\lambda_3 \leq  \non(\Nwf) = \lambda_2 \leq \cof(\Nwf) = \lambda_1 \leq \cfrak =\lambda_0.
    \end{align*}
\end{theorem}

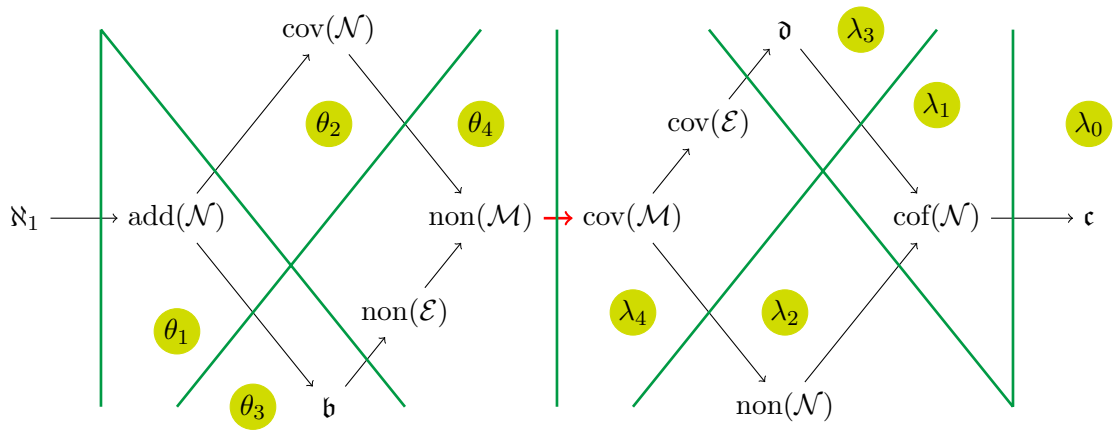
\begin{figure}[ht]
\centering
\begin{tikzpicture}
\small{
 \node (aleph1) at (-2,2.5) {$\aleph_1$};
 \node (addn) at (0,2.5){$\add(\Nwf)$};
 \node (covn) at (2,5){$\cov(\Nwf)$};
 \node (cove) at (7,3.75) {$\cov(\Ecal)$};
 \node (b) at (2,0) {$\bfrak$};
 \node (nonm) at (4,2.5) {$\non(\Mcal)$} ;
 \node (none) at (3,1.25) {$\non(\Ecal)$};
 \node (d) at (8,5) {$\dfrak$};
 \node (covm) at (6,2.5) {$\cov(\Mcal)$} ;
 \node (nonn) at (8,0) {$\non(\Ncal)$} ;
 \node (cfn) at (10,2.5) {$\cof(\Ncal)$} ;
  \node (c) at (12,2.5) {$\cfrak$};

\foreach \from/\to in {
aleph1/addn, addn/covn, addn/b, covn/nonm, b/none, none/nonm, covm/nonn, covm/cove, cove/d, nonn/cfn, d/cfn, cfn/c}
{
\path[-,draw=white,line width=3pt] (\from) edge (\to);
\path[->,] (\from) edge (\to);
}
\path[->,draw=red,line width=1pt] (nonm) edge (covm);

\draw[color=sug,line width=1]
(-1,0)--(-1,5)
(-1,5)--(3,0)
(0,0)--(4,5)
(5,5)--(5,0)
(6,0)--(10,5)
(7,5)--(11,0)
(11,0)--(11,5);

\draw[circle, fill=suy,color=suy] 
(0,1) circle (0.3)
(2,3.75) circle (0.3)
(1,0) circle (0.3)
(4,3.75) circle (0.3)
(6,1.25) circle (0.3)
(9,5) circle (0.3)
(8,1.25) circle (0.3)
(10,4) circle (0.3)
(12,3.75) circle (0.3);

\node at (0,1) {$\theta_1$};
\node at (2,3.75) {$\theta_2$};
\node at (1,0) {$\theta_3$};
\node at (4,3.75) {$\theta_4$};
\node at (6,1.25) {$\lambda_4$};
\node at (9,5) {$\lambda_3$};
\node at (8,1.25) {$\lambda_2$};
\node at (10,4) {$\lambda_1$};
\node at (12,3.75) {$\lambda_0$};

}
\end{tikzpicture}
\caption{The constellation of Cicho\'n's maximum forced in \autoref{mainBUP1}.}\label{fig:CMax1}
\end{figure}

The previous theorem is originally from~\cite{GKS} under the restriction that all $\lambda_i$ are regular. In joint work of the author with Goldstern, Kellner, and Shelah, we discovered~\autoref{BUP-Tukey} and realized that all $\lambda_i$ are allowed to be singular.

\begin{proof}
    Let $\Por_4$ be the ccc poset obtained in \autoref{main1Tukey} for $\lambda=\lambda_4$. For each $\ell\leq 3$ let $j_\ell\colon V\to M_\ell$ be an elementary embedding as in \autoref{tonti} for $\kappa=\kappa_\ell$ and $\lambda = \lambda_\ell$, and define $\Por_\ell := j_\ell(\Por_{\ell+1})$ (by recursion). By using \autoref{BUP-Tukey} and \autoref{BUPtool}, we show that $\Por_\ell$ forces the Tukey-equivalences and the value of $\cfrak$ as illustrated in \autoref{tableBUP}. In particular, $\Por_0$ is the desired ccc poset.
    
\begin{table}[ht]
\centering
\begin{tabular}{|c||c|c|c|c|c|}
\hline
 & $\Por_4$ & $\Por_3$ & $\Por_2$ & $\Por_1$ & $\Por_0$\\
 \hline\hline
 $\Cbf_\Mwf,\ \Cbf_\Ewf$ & $[\lambda_4]^{<\theta_4}$ & $[\lambda_4]^{<\theta_4}$ & $[\lambda_4]^{<\theta_4}$ & $[\lambda_4]^{<\theta_4}$ & $[\lambda_4]^{<\theta_4}$\\
 \hline
 $\omega^\omega$ & $[\lambda_4]^{<\theta_3}$ & $[\lambda_3]^{<\theta_3}$ & $[\lambda_3]^{<\theta_3}$ & $[\lambda_3]^{<\theta_3}$ & $[\lambda_3]^{<\theta_3}$\\
 \hline
 $\Cbf_\Nwf^\perp$ & $[\lambda_4]^{<\theta_2}$ & $[\lambda_3]^{<\theta_2}$ & $[\lambda_2]^{<\theta_2}$ & $[\lambda_2]^{<\theta_2}$ & $[\lambda_2]^{<\theta_2}$\\
 \hline
 $\Nwf$ & $[\lambda_4]^{<\theta_1}$ & $[\lambda_3]^{<\theta_1}$ & $[\lambda_2]^{<\theta_1}$ & $[\lambda_1]^{<\theta_1}$ & $[\lambda_1]^{<\theta_1}$\\
 \hline\hline
 $\cfrak$ & $\lambda_4$ & $\lambda_3$ & $\lambda_2$ & $\lambda_1$ & $\lambda_0$\\
 \hline
\end{tabular}
\caption{This table illustrates that the forcing on the top forces a Tukey-equivalence with the relational systems in the first column for the first four rows, and the value it forces to the continuum in the last row.}\label{tableBUP}
\end{table}
    
By \autoref{main1Tukey}, $\Por_4$ forces the Tukey-equivalences and $\cfrak=\lambda_4$ as illustrated in \autoref{tableBUP}. Then, by \autoref{BUP-Tukey}, $\Por_3=j_3(\Por_4)$ has the ccc and it forces Tukey-equivalences with $j_3([\lambda_4]^{<\theta_i})$ for $1\leq i\leq 4$. Since $\theta_4>\kappa_3$, by \autoref{BUPtool}~\ref{BUPtool2}, $j_3([\lambda_4]^{<\theta_4})\eqT [\lambda_4]^{<\theta_4}$; and for $1\leq i\leq 3$, since $\theta_i<\kappa_3$, by \autoref{BUPtool}~\ref{BUPtool1}, $j_3([\lambda_4]^{<\theta_i}) = [j_3(\lambda_4)]^{<\theta_i}$. On the other hand, by~\autoref{tonti}, $|j_3(\lambda_4)|=\lambda_3^{\kappa_3} = \lambda_3$, so $j_3([\lambda_4]^{<\theta_i}) \eqT [\lambda_3]^{<\theta_i}$. Finally, by elementarity, in $M_3$, $\Por_3$ forces $\cfrak = |j_3(\lambda_4)| = \lambda_3$, which is also forced in $V$ because $j_3(\Por_3)$ has the ccc and $M_3$ is countably closed (so any $\Por_3$-name of a real is in $M_3$).

Similarly, by using the elementary embedding $j_2$, $\Por_2 = j_2(\Por_3)$ has the ccc, $j_2([\lambda_i]^{<\theta_i}) \eqT [\lambda_i]^{<\theta_i}$ for $i=3,4$,  $j_2([\lambda_3]^{<\theta_i}) = [j_2(\lambda_3)]^{<\theta_i} \eqT [\lambda_2]^{<\theta_i}$ for $i=1,2$, and $\Por_2$ forces $\cfrak = j_2(\lambda_3) = \lambda_2$. Proceeding in the same way using $j_1$ and $j_0$, we can check the remaining information collected in \autoref{tableBUP}.
\end{proof}

In the same way, we can apply Boolean ultrapowers to the construction of \autoref{altKST} to force an alternative order of Cicho\'n's maximum. The result is originally from~\cite{KST} for $\lambda_i$ regular, but our methods allow them to be singular.

\begin{theorem}\label{mainBUP2}
    Assume that $\aleph_1 <\kappa_0 < \theta_1 < \kappa_1 <\theta_2 <\kappa_2 < \theta_3 < \kappa_3 < \theta_4 \leq \lambda_4 \leq \lambda_3 \leq \lambda_2 \leq \lambda_1 \leq \lambda_0$ such that each $\theta_i$ is regular, $\cf([\lambda_4]^{{<}\theta_i}) = \lambda_4$, condition~\ref{KSTii} of \autoref{altKST} holds, $\kappa_\ell$ is strongly compact and $\lambda_\ell^{\kappa_\ell} = \lambda_\ell$ for $\ell\leq 3$. Then there is a ccc poset of size $\lambda_0$ forcing $\cfrak=\lambda_0$, $\Nwf \eqT \Cbf_{[\lambda_1]^{<\theta_1}}$, $\omega^\omega \eqT \Cbf_{[\lambda_2]^{<\theta_2}}$, $\Cbf^\perp_\Nwf \eqT \Cbf_{[\lambda_3]^{<\theta_3}}$ and $\Cbf_\Mwf \eqT \Cbf_{\Ewf} \eqT \Cbf_{[\lambda_4]^{<\theta_4}}$. In particular (see \autoref{fig:CMax2}),
    \begin{align*}
        \aleph_1 & < \add(\Nwf) = \theta_1 < \bfrak = \theta_2 < \cov(\Nwf) = \theta_3 < \non(\Ewf) =\non(\Mwf) = \theta_4\\
        & \leq \cov(\Mwf)= \cov(\Ewf)  = \lambda_4 \leq \non(\Nwf)=\lambda_3 \leq \dfrak = \lambda_2 \leq \cof(\Nwf) = \lambda_1 \leq \cfrak =\lambda_0.
    \end{align*}
\end{theorem}

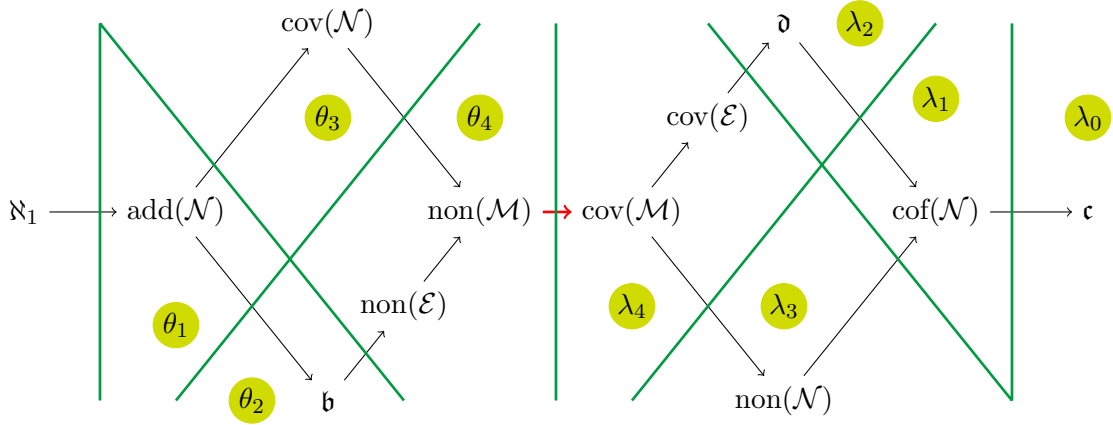
\begin{figure}[ht]
\centering
\begin{tikzpicture}
\small{
 \node (aleph1) at (-2,2.5) {$\aleph_1$};
 \node (addn) at (0,2.5){$\add(\Nwf)$};
 \node (covn) at (2,5){$\cov(\Nwf)$};
 \node (cove) at (7,3.75) {$\cov(\Ecal)$};
 \node (b) at (2,0) {$\bfrak$};
 \node (nonm) at (4,2.5) {$\non(\Mcal)$} ;
 \node (none) at (3,1.25) {$\non(\Ecal)$};
 \node (d) at (8,5) {$\dfrak$};
 \node (covm) at (6,2.5) {$\cov(\Mcal)$} ;
 \node (nonn) at (8,0) {$\non(\Ncal)$} ;
 \node (cfn) at (10,2.5) {$\cof(\Ncal)$} ;
  \node (c) at (12,2.5) {$\cfrak$};

\foreach \from/\to in {
aleph1/addn, addn/covn, addn/b, covn/nonm, b/none, none/nonm, covm/nonn, covm/cove, cove/d, nonn/cfn, d/cfn, cfn/c}
{
\path[-,draw=white,line width=3pt] (\from) edge (\to);
\path[->,] (\from) edge (\to);
}
\path[->,draw=red,line width=1pt] (nonm) edge (covm);

\draw[color=sug,line width=1]
(-1,0)--(-1,5)
(-1,5)--(3,0)
(0,0)--(4,5)
(5,5)--(5,0)
(6,0)--(10,5)
(7,5)--(11,0)
(11,0)--(11,5);

\draw[circle, fill=suy,color=suy] 
(0,1) circle (0.3)
(2,3.75) circle (0.3)
(1,0) circle (0.3)
(4,3.75) circle (0.3)
(6,1.25) circle (0.3)
(9,5) circle (0.3)
(8,1.25) circle (0.3)
(10,4) circle (0.3)
(12,3.75) circle (0.3);

\node at (0,1) {$\theta_1$};
\node at (2,3.75) {$\theta_3$};
\node at (1,0) {$\theta_2$};
\node at (4,3.75) {$\theta_4$};
\node at (6,1.25) {$\lambda_4$};
\node at (9,5) {$\lambda_2$};
\node at (8,1.25) {$\lambda_3$};
\node at (10,4) {$\lambda_1$};
\node at (12,3.75) {$\lambda_0$};

}
\end{tikzpicture}
\caption{The constellation of Cicho\'n's maximum forced in \autoref{mainBUP2}.}\label{fig:CMax2}
\end{figure}

There are four possible constellations of Cicho\'n's maximum with $\non(\Mwf)<\cov(\Mwf)$. We have proved two of them, but the consistency of each of the other two is not known.

\begin{question}
    Is each of the constellations in \autoref{fig:CMaxQ1} and \autoref{fig:CMaxQ2} consistent with $\thzfc$ (even under large cardinals)?
\end{question}

On the other hand, 
no instance of Cicho\'n's maximum with $\cov(\Mwf) < \non(\Mwf)$ is known to be consistent with ZFC.

\begin{question}
    Is Cicho\'n's maximum with $\cov(\Mwf) < \non(\Mwf)$ consistent with $\thzfc$ (even under large cardinals)?
\end{question}

\begin{figure}[ht]
\centering
\begin{tikzpicture}
\small{
 \node (aleph1) at (-2,2.5) {$\aleph_1$};
 \node (addn) at (0,2.5){$\add(\Nwf)$};
 \node (covn) at (2,5){$\cov(\Nwf)$};
 \node (cove) at (7,3.75) {$\cov(\Ecal)$};
 \node (b) at (2,0) {$\bfrak$};
 \node (nonm) at (4,2.5) {$\non(\Mcal)$} ;
 \node (none) at (3,1.25) {$\non(\Ecal)$};
 \node (d) at (8,5) {$\dfrak$};
 \node (covm) at (6,2.5) {$\cov(\Mcal)$} ;
 \node (nonn) at (8,0) {$\non(\Ncal)$} ;
 \node (cfn) at (10,2.5) {$\cof(\Ncal)$} ;
  \node (c) at (12,2.5) {$\cfrak$};

\foreach \from/\to in {
aleph1/addn, addn/covn, addn/b, covn/nonm, b/none, none/nonm, covm/nonn, covm/cove, cove/d, nonn/cfn, d/cfn, cfn/c}
{
\path[-,draw=white,line width=3pt] (\from) edge (\to);
\path[->,] (\from) edge (\to);
}
\path[->,draw=red,line width=1pt] (nonm) edge (covm);

\draw[color=dred,line width=1]
(-1,0)--(-1,5)
(-1,5)--(3,0)
(0,0)--(4,5)
(5,5)--(5,0)
(6,0)--(10,5)
(7,5)--(11,0)
(11,0)--(11,5);

\draw[circle, fill=suy,color=suy] 
(0,1) circle (0.3)
(2,3.75) circle (0.3)
(1,0) circle (0.3)
(4,3.75) circle (0.3)
(6,1.25) circle (0.3)
(10,4) circle (0.3)
(12,3.75) circle (0.3);

\draw[circle, fill=orange, color=orange]
(9,5) circle (0.3)
(8,1.25) circle (0.3);

\node at (0,1) {$\theta_1$};
\node at (2,3.75) {$\theta_2$};
\node at (1,0) {$\theta_3$};
\node at (4,3.75) {$\theta_4$};
\node at (6,1.25) {$\lambda_4$};
\node at (9,5) {$\lambda_2$};
\node at (8,1.25) {$\lambda_3$};
\node at (10,4) {$\lambda_1$};
\node at (12,3.75) {$\lambda_0$};

}
\end{tikzpicture}
\caption{Cicho\'n's maximum's constellation like in \autoref{mainBUP1}, but with the values of $\dfrak$ and $\non(\Nwf)$ interchanged.}\label{fig:CMaxQ1}
\end{figure}
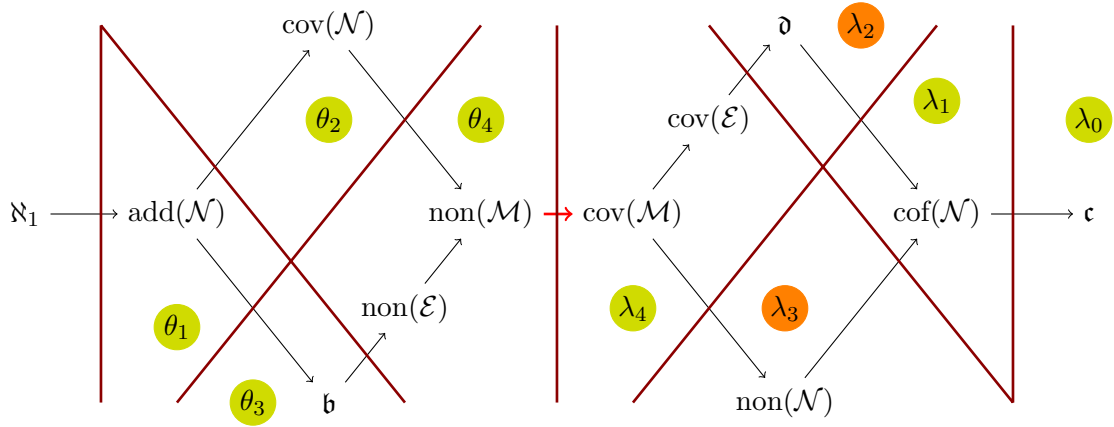

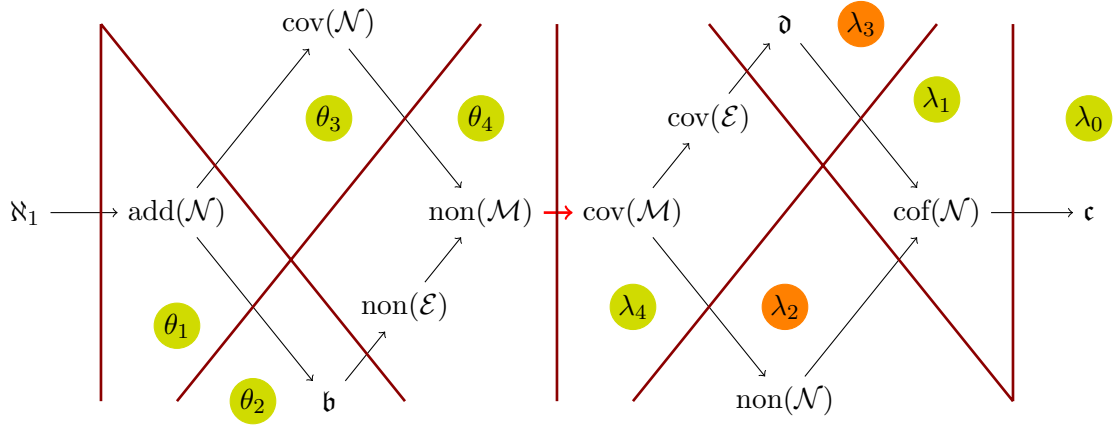
\begin{figure}[ht]
\centering
\begin{tikzpicture}
\small{
 \node (aleph1) at (-2,2.5) {$\aleph_1$};
 \node (addn) at (0,2.5){$\add(\Nwf)$};
 \node (covn) at (2,5){$\cov(\Nwf)$};
 \node (cove) at (7,3.75) {$\cov(\Ecal)$};
 \node (b) at (2,0) {$\bfrak$};
 \node (nonm) at (4,2.5) {$\non(\Mcal)$} ;
 \node (none) at (3,1.25) {$\non(\Ecal)$};
 \node (d) at (8,5) {$\dfrak$};
 \node (covm) at (6,2.5) {$\cov(\Mcal)$} ;
 \node (nonn) at (8,0) {$\non(\Ncal)$} ;
 \node (cfn) at (10,2.5) {$\cof(\Ncal)$} ;
  \node (c) at (12,2.5) {$\cfrak$};

\foreach \from/\to in {
aleph1/addn, addn/covn, addn/b, covn/nonm, b/none, none/nonm, covm/nonn, covm/cove, cove/d, nonn/cfn, d/cfn, cfn/c}
{
\path[-,draw=white,line width=3pt] (\from) edge (\to);
\path[->,] (\from) edge (\to);
}
\path[->,draw=red,line width=1pt] (nonm) edge (covm);

\draw[color=dred,line width=1]
(-1,0)--(-1,5)
(-1,5)--(3,0)
(0,0)--(4,5)
(5,5)--(5,0)
(6,0)--(10,5)
(7,5)--(11,0)
(11,0)--(11,5);

\draw[circle, fill=suy,color=suy] 
(0,1) circle (0.3)
(2,3.75) circle (0.3)
(1,0) circle (0.3)
(4,3.75) circle (0.3)
(6,1.25) circle (0.3)
(10,4) circle (0.3)
(12,3.75) circle (0.3);

\draw[circle, fill=orange, color=orange]
(9,5) circle (0.3)
(8,1.25) circle (0.3);

\node at (0,1) {$\theta_1$};
\node at (2,3.75) {$\theta_3$};
\node at (1,0) {$\theta_2$};
\node at (4,3.75) {$\theta_4$};
\node at (6,1.25) {$\lambda_4$};
\node at (9,5) {$\lambda_3$};
\node at (8,1.25) {$\lambda_2$};
\node at (10,4) {$\lambda_1$};
\node at (12,3.75) {$\lambda_0$};

}
\end{tikzpicture}
\caption{Cicho\'n's maximum's constellation like in \autoref{mainBUP2}, but with the values of $\dfrak$ and $\non(\Nwf)$ interchanged.}\label{fig:CMaxQ2}
\end{figure}


To finish the paper, we briefly mention some facts around 
the method of intersection with submodels, originally from~\cite{GKMS}, which is used to prove the consistency of Cicho\'n's maximum without using large cardinals. 
We only do a very brief presentation for comparison with the method of Boolean ultrapowers, but
we recommend the reader to check and learn the details from~\cite[Sec.~4 \&~5]{CM22}.

 From now on, we fix:
 \begin{enumerate}[label= \normalfont (F\arabic*)]
     \item an uncountable regular cardinal $\kappa$, a $\kappa$-cc poset $\Por$;
     \item a definable relational system $R =\la X,Y,\sqsubset\ra$ of the reals; and
     \item a large enough regular cardinal $\chi$ such that $\Por\in H_\chi$, and $H_\chi$ contains all the parameters defining $R $.
 \end{enumerate}
 

 When intersecting a $\kappa$-cc poset with a ${<}\kappa$-closed model, we obtain a completely embedded subforcing.

 \begin{lemma}
  If $N\preceq H_\chi$ is ${<}\kappa$-closed and $\Por\in N$, then $\Por\cap N\subsetdot\Por$.
 \end{lemma}

 Note that this lemma was used implicitly in the proof of \autoref{main2} (to define $\Por_{\lambda+\alpha}^-$ for $\alpha\in K_4$).

 Semantically, there is a correspondence between some $\Por\cap N$-names and $\Por$-names belonging to $N$, and we can also have a correspondence for the forcing relation for some formulas.

 \begin{fact}
   If $N$ is ${<}\kappa$ closed then there is a one-to-one correspondence between:
   \begin{enumerate}[label=(\roman*)]
     \item $\Por$-names $\tau\in N$ and
     \item $\Por\cap N$-names $\sigma$
   \end{enumerate}
   of members of $H_\kappa$ (in particular, reals).  
   Thus, if $G$ is $\Por$-generic over $V$ then $H_\kappa^{V[G]}\cap N[G] =  H_\kappa^{V[G\cap N]}$.
 \end{fact}

 \begin{corollary}
   For absolute $\varphi(\bar x)$ (e.g. Borel on the reals) if $p\in\Por\cap N$, $\bar\tau\in N$ is a finite sequence of $\Por$-names of members of $H_\kappa$, and $\bar\sigma$ is the sequence of $\Por\cap N$-names corresponding to $\bar\tau$, then
   \[p\Vdash_\Por \varphi(\bar\tau) \sii p\Vdash_{\Por\cap N}\varphi(\bar\sigma).\]
 \end{corollary}

 The following result illustrates the main motivation to intersect $\kappa$-cc posets with ${<}\kappa$-closed models, since it affects the Tukey relations forced by the posets. As before, $K$ denotes a constant relational system from the ground model, while $R$ is interpreted in the model of discourse.

 \begin{lemma}\label{fct:restrN}
   Let $N\preceq H_\chi$ be ${<}\kappa$-closed and let $K=\la A,B,\lhd\ra$ be a relational system. Assume that $\Por$, $K$ and the parameters of $R $ are in $N$.
   \begin{enumerate}[label=\normalfont (\alph*)]
     \item If $\Vdash_{\Por} R \leqT K$ then $\Vdash_{\Por\cap N} R \leqT K\cap N$ where $K\cap N:=\la A\cap N,B\cap N,\lhd\ra$. Moreover, if $\seq{\dot y_b}{b\in B}$ is a sequence of $\Por$-names of a witness of $R\leqT K$, then $\seq{\dot y_b}{b\in B\cap N}$ is forced by $\Por\cap N$ to witness $R\leqT K\cap N$.
     \item If $\Vdash_{\Por} K\leqT R $ then $\Vdash_{\Por\cap N} K\cap N\leqT R $. Moreover, if $\seq{\dot x_a}{a\in A}$ is a sequence of $\Por$-names of a witness of $K\leqT R$, then $\seq{\dot x_a}{a\in A\cap N}$ is forced by $\Por\cap N$ to witness $K\cap N \leqT R$.
   \end{enumerate}
 \end{lemma}



Therefore, if $\Por$ forces $R\eqT K$, then $\Por\cap N$ forces $R\eqT K\cap N$, i.e.\ $\bfrak(R)=\bfrak(K\cap N)$ and $\dfrak(R)=\dfrak(K\cap N)$. This indicates that, to force different values of cardinal characteristics using intersection with submodels, one should construct an $N$ such that $\bfrak(K\cap N)$ and $\dfrak(K\cap N)$ has the values one desires. In the case of Cicho\'n's diagram, 
one constructs a $\sigma$-closed $N\preceq H_\chi$ such that $\Por\cap N$ forces Cicho\'n's maximum, where $\Por$ is the poset obtained from \autoref{main1Tukey} (and \autoref{altKST}).

\begin{theorem}[{\cite{GKMS}}]
 Let $\aleph_1 \leq \theta_1 \leq\theta_2 \leq \theta_3 \leq \theta_4 \leq \lambda_4 \leq \lambda_3 \leq \lambda_2 \leq \lambda_1$ be regular cardinals and $\lambda_0= \lambda_0^{\aleph_0}\geq \lambda_1$. Then, for each one of the constellations in \autoref{fig:CMax1} and \autoref{fig:CMax2}, there is a ccc poset of size $\lambda_0$ forcing it.
\end{theorem}

We do not prove this theorem in detail but mention how the construction of $N$ goes. 
For details, see~\cite[Sec.~5]{CM22}. We pick a large enough regular cardinal $\kappa_0> \lambda_0$ and assume, for the moment, that GCH holds above $\kappa_0$, i.e.\ $2^\kappa = \kappa^+$ for every cardinal $\kappa\geq\kappa_0$. We explain in \autoref{remGKMS} how to remove this assumption. 

For $j=1,\ldots,4$, pick a regular cardinal $\nu_j>\kappa_0$, which is the successor of another regular cardinal $\nu^-_j$, such that $\nu_j<\nu^-_{j+1}$ for $j=1,2,3$. Finally, pick a regular cardinal $\nu_\infty>\nu_4$.

To force \autoref{fig:CMax1}, we start with a forcing $\Por$ as in \autoref{main1Tukey} but applied to $\nu_1,\nu_2,\nu_3,\nu_4,$ and $\nu_\infty$, i.e.\ it is a ccc poset of size $\nu_\infty$ forcing $\Ncal\eqT [\nu_\infty]^{<\nu_1}$, $\Cbf^\perp_\Ncal\eqT [\nu_\infty]^{<\nu_2}$, $\omega^\omega\eqT [\nu_\infty]^{<\nu_3}$, $\Cbf_\Mcal\eqT\Cbf_\Ecal\eqT [\nu_\infty]^{<\nu_4}$, and $\cfrak=\nu_\infty$. Recall that, for $1\leq j\leq 4$, $\bfrak([\nu_\infty]^{<\nu_j})=\nu_1$ and $\dfrak([\nu_\infty]^{<\nu_j})=\nu_\infty$.

Note that $\Por$ forces the values in Cicho\'n's diagram as in the top part of \autoref{fig:cichoncoll}. Afterward, one constructs a $\sigma$-closed model $N$ such that $\Por\cap N$ forces the constellation at the bottom of \autoref{fig:cichoncoll}.

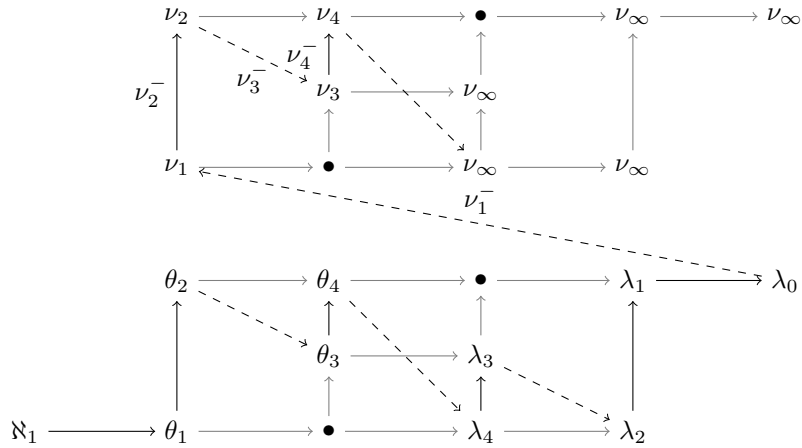
\begin{figure}[ht]
\centering
\begin{tikzpicture}[xscale=2/1]
\footnotesize{
\node (addn) at (0,3.5){$\nu_1$};
\node (covn) at (0,5.5){$\nu_2$};
\node (nonn) at (3,3.5) {$\nu_\infty$} ;
\node (cfn) at (3,5.5) {$\nu_\infty$} ;
\node (addm) at (1,3.5) {$\bullet$} ;
\node (covm) at (2,3.5) {$\nu_\infty$} ;
\node (nonm) at (1,5.5) {$\nu_4$} ;
\node (cfm) at (2,5.5) {$\bullet$} ;
\node (b) at (1,4.5) {$\nu_3$};
\node (d) at (2,4.5) {$\nu_\infty$};
\node (c) at (4,5.5) {$\nu_\infty$};

\draw (addn) edge[->] node[left] {$\nu^-_2$} (covn);
\draw[gray]
      (covn) edge [->] (nonm)
      (nonm)edge [->] (cfm)
      (cfm)edge [->] (cfn)
      (addn) edge [->]  (addm)
      (addm) edge [->]  (covm)
      (addm) edge [->] (b)
      (d)  edge[->] (cfm)
      (b) edge [->] (d)
      (nonn) edge [->]  (cfn)
      (cfn) edge[->] (c)
      (covm) edge [->] (d)
      (covm) edge [->]  (nonn);

\draw (b)  edge [->] node[left] {$\nu^-_4$} (nonm);

\draw[dashed] (covn) edge[->] node[below] {$\nu^-_3$} (b);
\draw[dashed] (nonm) edge[->] (covm);

\node (aleph1) at (-1,0) {$\aleph_1$};
\node (addn-f) at (0,0){$\theta_1$};
\node (covn-f) at (0,2){$\theta_2$};
\node (nonn-f) at (3,0) {$\lambda_2$} ;
\node (cfn-f) at (3,2) {$\lambda_1$} ;
\node (addm-f) at (1,0) {$\bullet$} ;
\node (covm-f) at (2,0) {$\lambda_4$} ;
\node (nonm-f) at (1,2) {$\theta_4$} ;
\node (cfm-f) at (2,2) {$\bullet$} ;
\node (b-f) at (1,1) {$\theta_3$};
\node (d-f) at (2,1) {$\lambda_3$};
\node (c-f) at (4,2) {$\lambda_0$};

\draw[dashed] (c-f) edge[->] node[above] {$\nu^-_1$} (addn);
\draw (aleph1) edge[->] (addn-f)
      (addn-f) edge[->] (covn-f)
      (nonn-f) edge [->]  (cfn-f)
      (b-f)  edge [->] (nonm-f)
      (covm-f) edge [->] (d-f)
      (cfn-f) edge[->] (c-f);

\draw[gray]
   (covn-f) edge [->] (nonm-f)
   (nonm-f)edge [->] (cfm-f)
   (cfm-f)edge [->] (cfn-f)
   (addn-f) edge [->]  (addm-f)
   (addm-f) edge [->]  (covm-f)
   (covm-f) edge [->]  (nonn-f)
   (addm-f) edge [->] (b-f)
   (d-f)  edge[->] (cfm-f)
   (b-f) edge [->] (d-f);

\draw[dashed] (covn-f) edge [->] (b-f)
              (nonm-f) edge [->] (covm-f)
              (d-f) edge [->] (nonn-f);

}
\end{tikzpicture}
\caption[Strategy to force Cicho\'n's maximum.]{Strategy to force Cicho\'n's maximum: we construct a ccc poset $\Por$ forcing the constellation at the top, and find a $\sigma$-closed model $N$ such that $\Por\cap N$ forces the constellation at the bottom.}\label{fig:cichoncoll}
\end{figure}

One constructs a sequence of elementary submodels of $H_\chi$, containing all the relevant information, as below:
\[N^\cfrak\in N^\bfrak_1 \in N^\dfrak_1 \in N^\bfrak_2 \in N^\dfrak_2 \in N^\bfrak_3 \in N^\dfrak_3 \in N^\bfrak_4 \in N^\dfrak_4.\]
For $1\leq j\leq 4$, $N_j^\dfrak$ is constructed from a $\in$-increasing sequence of length $\lambda_j$ of ${<}\nu_j$-closed elementary submodels of $H_\chi$ of size $\nu_j$, and $N_j^\bfrak$ is constructed from a $\in$-increasing sequence of length $\theta_j$ of ${<}\nu^-_j$-closed elementary submodels of $H_\chi$ of size $\nu^-_j$.  The models are constructed by decreasing recursion on $j$, first $N^\dfrak_j$, and $N^\bfrak_j$ afterwards. At the end, $N^\cfrak$ is chosen as a $\sigma$-closed elementary submodel of $H_\chi$ of size $\lambda_0$. 

Finally, $N:=N^\cfrak\cap\bigcap_{j=1}^4 N^\bfrak_j\cap N^\dfrak_j$. For $1\leq i\leq 4$, each model $\bigcap_{j=i}^4 N^\bfrak_j\cap N^\dfrak_j$, after intersection with $\Por$, collapses the values forced at the top of \autoref{fig:cichoncoll} to the values at the bottom, taking care of $\theta_j,\lambda_j$ for $i\leq j\leq 4$ in particular.
The final intersection with $N^\cfrak$ is to force that the continuum is $\lambda_0$ and to conclude $\bfrak([\nu_\infty]^{<\nu_i}\cap N)=\theta_i$ and $\dfrak([\nu_\infty]^{<\nu_i}\cap N)=\lambda_i$ for $1\leq i\leq 4$. Moreover, it can be proved that $[\nu_\infty]^{<\nu_i}\cap N \eqT \prod_{j=i}^4 \theta_j\times\lambda_j$ (with the coordinate-wise order).

\begin{remark}\label{remGKMS}
The method from~\cite{GKMS} (and~\cite{CM22}) uses \emph{eventual GCH} (i.e. for some cardinal $\kappa_0$, $2^\kappa = \kappa^+$ for every cardinal $\kappa\geq\kappa_0$) as hypothesis to force Cicho\'n's maximum. However, thanks to an observation from Elliot Glazer (private communication), this assumption can be removed: Let $\kappa^*$ be a large enough regular cardinal, and let $W$ be a set of ordinals coding $R_{\kappa^*}$, the $\kappa^*$-th level of the Von Neumann hierarchy of the universe of sets. Since $L[W]$ models ZFC with eventual GCH, working inside $L[W]$ we can find a ccc poset $\Qor$ forcing Cicho\'n's maximum, which has size $\lambda_0$, and hence, inside $R_{\alpha}$ for some $\alpha$ relatively small with respect to $\kappa^*$. As $\kappa^*$ is \emph{large enough}, we actually have that the collection of nice $\Qor$-names of reals  are in $R_{\kappa^*}$, so $R_{\kappa^*}$, and hence $V$, satisfies that $\Q$ is as required.
\end{remark}

\begin{remark}
	In contrast with the method of Boolean ultrapowers, we do not know how this method can be improved to allow $\lambda_i$ singular for $1\leq i\leq 4$.
\end{remark}

\begin{acknowledgements}
  These lecture notes correspond to a mini-course of six sessions at the University of Vienna, from November 30th, 2023 to January 25th, 2024. I am very thankful to Professor Vera Fischer for allowing me to teach this mini-course, and to all the participants (in person and online) who joined the sessions. I also thank my student Andres Uribe-Zapata (TU Wien) who checked a great part of these lecture notes before publication.
  
  My visit to Vienna was financially supported by Professor Martin Goldstern's grant (TU Wien) and by the Grants-in-Aid for Scientific Research (C) 23K03198, Japan Society for the Promotion of Science.
  
  I also thank the referee for the careful reading of this manuscript and for many useful suggestions.
\end{acknowledgements}

{\small
\bibliography{left}
\bibliographystyle{alpha}
}

\end{document}